\documentclass[11pt]{amsart}
\usepackage{amsthm}
\usepackage{amssymb}
\usepackage{amsmath}

\usepackage[all]{xy}

\usepackage[usenames]{color}
\usepackage{cite}

\usepackage{fullpage}

\swapnumbers  

\newtheorem{lemma}[equation]{Lemma}
\newtheorem{theorem}[equation]{Theorem}
\newtheorem*{theorem*}{Theorem}

\newtheorem{proposition}[equation]{Proposition}
\newtheorem*{corollary*}{Corollary}

\theoremstyle{definition}
\newtheorem{definition}[equation]{Definition}

\newtheorem{remark}[equation]{Remark}
\newtheorem{example}[equation]{Example}
\newtheorem{notation}[equation]{Notation}
\newtheorem {empt}[equation]      {}

\numberwithin{equation}{subsection}

\newcommand{\fp}{{\mathfrak p}}
\newcommand{\cU}{{\mathcal U}}
\newcommand{\cP}{{\mathcal P}}
\newcommand{\cV}{{\mathcal V}}

\newcommand{\N}{\mathbb{N}}
\newcommand{\R}{\mathbb{R}}
\newcommand{\C}{\mathbb{C}}
\newcommand{\ZZ}{\mathbb{Z}}
\newcommand{\PP}{\mathbb{P}}

\newcommand{\V}{\mathbb{V}}
\newcommand{\U}{\mathbb{U}}
\newcommand{\Cone}{\mathsf{Cone}}
\newcommand{\Vect}{\mathsf{Vect}}
\newcommand{\Set} {\mathsf{Set}}
\newcommand{\sfC} {\mathsf{C}}
\newcommand{\Man} {\mathsf{Man}}

\newcommand{\Cat} {\mathsf{Cat}}

\newcommand{\Graph}{\mathsf{Graph}}
\newcommand{\FinGraph}{\mathsf{FinGraph}}
\newcommand{\FinTree}{\mathsf{FinTree}}
\newcommand{\Ctrl}{\mathsf{Ctrl}}

\newcommand{\Finset}{\mathsf{FinSet}}
\newcommand{\FinSet}{\mathsf{FinSet}}
\newcommand{\Euc}{\mathsf{Euc}}
\newcommand{\FinVect}{\mathsf{FinVect}}
\newcommand{\FV}{\mathsf{FV}}
\newcommand{\CT}{\mathsf{CT}}

\newcommand{\A}{\mathsf{A}}
\newcommand{\B}{\mathsf{B}}
\newcommand{\D}{\mathsf{D}}
\newcommand{\G}{\mathsf{G}}
\newcommand{\X}{\mathsf{X}}
\newcommand{\Y}{\mathsf{Y}}
\newcommand{\Z}{\mathsf{Z}}

\newcommand{\inv}{^{-1}}

\newcommand{\toto}{\rightrightarrows}
\newcommand{\Hom}{\mathrm{Hom}}

\newcommand{\Aut}{\mathrm{Aut}}

\newcommand{\op}[1]{{#1}^{\mbox{\sf{\tiny{op}}}}}
\newcommand{\et}[1]{{#1}_{\mbox{\sf{\tiny{et}}}}}
\newcommand{\opet}[1]{{{#1}_{\mbox{\sf{\tiny{et}}}}}^{\mbox{\sf{\tiny{op}}}}}
\DeclareMathOperator{\lv}{lv}
\DeclareMathOperator{\rt}{rt}
\newcommand{\an}{\mathtt{an}}
\newcommand{\col}{\colon}

\begin{document}
\title[Dynamics on Networks I.]{Dynamics on networks I.
  Combinatorial categories of modular continuous-time systems}

\author{R. E. Lee DeVille and Eugene Lerman}
\thanks{Supported in part
   by NSF grants} 

\address{Department of
  Mathematics, University of Illinois, Urbana, IL 61801, USA}

\begin{abstract}
  We develop a new framework for the study of complex continuous time
  dynamical systems based on viewing them as collections of
  interacting control modules.  This framework is inspired by and
  builds upon the groupoid formalism of Golubitsky, Stewart and their
  collaborators.  Our approach uses the tools and---more
  importantly---the stance of category theory.  This enables us to put
  the groupoid formalism in a coordinate-free setting and to extend it
  from ordinary differential equations to vector fields on manifolds.
  In particular, we construct combinatorial models for categories of
  modular continuous time dynamical systems.  Each such model, as a
  category, is a fibration over an appropriate category of labeled
  directed graphs.  This makes precise the relation between dynamical
  systems living on networks and the combinatorial structure of the
  underlying directed graphs, allowing us to exploit the relation in
  new and interesting ways.
\end{abstract}
\maketitle

\tableofcontents

\section{Introduction}
\subsection{Background}
The goal of this paper is to develop a new framework for dynamics on
networks using the techniques and methods of category
theory. Networks, and the dynamical systems defined on them, are
ubiquitous in science, engineering and the social sciences;
understanding dynamics on networks constitutes a major scientific
challenge with applications across a variety of fields.  For example,
the mathematical study of dynamics on networks plays an important role
in the study and design of communications
networks~\cite{Rogers.Kincaid.book}; in cognitive science,
computational neuroscience, and robotics (see, for example
\cite{Knight.72, EK84, Kuramoto.91, Kuramoto.book, MacKay.book,
  Dayan.Abbott.book}); in the study of gene regulatory
networks~\cite{Bower.Bolouri.book, Shmulevich.etal.02,
  Alvarez-Buylla.etal.07} and more general complex biochemical
networks~\cite{Wilkinson.book}; and finally in complex active
media~\cite{Peskin.75, Tyson.Keener.88, Kapral.Showalter.book,
  Ermentrout.Rinzel.96, Keener.Sneyd.book}.  Current approaches to
understanding dynamics on networks include ideas from statistical
physics and random graph theory (see~\cite{Watts.Strogatz.98,
  Barabasi.Albert.99, Strogatz.01, Schwartz.etal.02,
  Albert.Barabasi.02, Barabasi.03, Barabasi.etal.04,
  Boccaletti.etal.06} and the many references therein).

We propose to develop a new approach to this problem which is inspired
by and builds upon the {\sf groupoid formalism} of Golubitsky, Stewart
and their collaborators~\cite{Golubitsky.Stewart.84,
  Golubitsky.Stewart.85, Golubitsky.Stewart.86,
  Golubitsky.Stewart.86.2, Golubitsky.Stewart.87,
  Golubitsky.Stewart.Schaeffer.book, Field.Golubitsky.Stewart.91,
  Golubitsky.Stewart.Dionne.94,
  Dellnitz.Golubitsky.Hohmann.Stewart.95,
  Dionne.Golubitsky.Silber.Stewart.95, Dionne.Golubitsky.Stewart.96.1,
  Dionne.Golubitsky.Stewart.96.2, Golubitsky.Stewart.Buono.Collins.98,
  Golubitsky.Stewart.98, Golubitsky.Stewart.99, Golubitsky.Stewart.00,
  Golubitsky.Knobloch.Stewart.00, Golubitsky.Stewart.02,
  Golubitsky.Stewart.02.2, Stewart.Golubitsky.Pivato.03,
  Golubitsky.Nicol.Stewart.04, Golubitsky.Pivato.Stewart.04,
  Golubitsky.Stewart.05, Golubitsky.Stewart.Torok.05,
  Golubitsky.Stewart.06, Golubitsky.Josic.Brown.06,
  Golubitsky.Shiau.Stewart.07} and references therein).  The key idea
of the groupoid formalism is this: many networks are {\sf modular},
and the modes of interaction between pieces of the network are {\sf
  repeated} across the network.  This repetition is a symmetry in a
very broad sense of the word.  In the case of network dynamics modeled
by ordinary differential equations (ODEs), Golubitsky et al.\ found a
mathematical formulation of this symmetry as a groupoid symmetry of
the governing equations.  We follow their lead and develop a new
framework for the groupoid formalism using the tools and, perhaps more
pertinently, the stance of category theory
\cite{MacLane}.\footnote{The elements of category theory that are used
  in this paper are reviewed in Appendix~\ref{sec:category}.}

In this first paper in a series we put the groupoid formalism of
Golubitsky et al.\ in a coordinate-free framework and then extend it
from ordinary differential equations to vector fields on manifolds.
In particular {\sf we construct combinatorial models for categories of
  modular continuous-time dynamical systems.}  Discrete time systems,
groupoid-compatible numerical methods, and stochastic dynamical
systems defined on networks will all be taken up in subsequent
papers~\cite{DeVille.Lerman.10.2,DeVille.Lerman.10.3}.

To explain what our work is  about we start with an example.
Consider an ODE in $({\mathbb R}^n)^3$ of the form
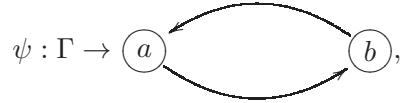
\begin{equation}\label{eq1.2.1}
  \dot x_1 = f(x_2),\quad   \dot x_2 = f(x_1), \quad  \dot x_3 = f(x_2)
\end{equation}
for some smooth function $f: {\mathbb R}^n\to {\mathbb R}^n$.  That
is, consider the flow of the vector field
\[
F: ({\mathbb R}^n)^3\to ({\mathbb R}^n)^3, \quad 
F(x_1, x_2, x_3) = (f(x_2), f(x_1), f(x_2)).
\]
It is easy to check that $F$ is tangent to the diagonal 
\[
{\mathbb R}^n \simeq  \Delta 
= \{(x_1, x_2, x_3) \in (\mathbb{R}^n)^3 \mid x_1 = x_2 = x_3\}
\]
and that the restriction of the flow of $F$ to $\Delta$
is the flow of the ODE
\[
\dot u = f(u).
\] 
One can also see another invariant submanifold of $F$:
\[
({\mathbb R}^n)^2 \simeq  \Delta' = \{(x_1, x_2, x_3) 
\in ({\mathbb R}^n)^3 \mid x_1 = x_3\}.
\]
On $\Delta'$ the flow of $F$ is the flow of the ODE
\[
\dot v_1 = f(v_2), \quad \dot v_2 = f(v_1).
\]
Moreover the projection 
\[
\pi: ({\mathbb R}^n)^3 \to \Delta', \quad \pi (x_1, x_2, x_3) = (x_1, x_2, x_1)
\]
intertwines the flows of $F$ on $({\mathbb R}^n)^3$ and on $\Delta'$.  
We have thus observed two subsystems of
$(({\mathbb R}^n)^3, F)$ and three maps between the three dynamical systems:
\begin{equation}\label{eq1.2.2}
\xy
(-40, 0)*+{(\Delta, F|_\Delta)}="1";
(0,0)*+{(({\mathbb R}^n)^3, F)}="2";
(10,1)*+{}="2+"; 
(10,-1)*+{}="2-"; 
(40, 0)*+{(\Delta', F|_{\Delta'})}="3";
(30, 1)*+{}="3+";
(30, -1)*+{}="3-";
 {\ar@{^{(}->} "1";"2"};
 {\ar@{_{(}->} "3+";"2+"};
 {\ar@{->}_{\pi} "2-";"3-"}; 
\endxy
\end{equation}
Where do these subsystems and maps come from?  There is no obvious
symmetry of $({\mathbb R}^n)^3$ that preserves the vector field $F$ and fixes
the diagonal $\Delta$ and thus could account for the existence of this
invariant submanifold.  Nor is there any $F$-preserving symmetry that
fixes $\Delta'$.  In fact the vector field $F$ doesn't seem to have
any symmetry.  The graph $\Gamma$ recording the interdependence of the
variables $(x_1, x_2, x_3)$ in the ODE \eqref{eq1.2.1} has three
vertices and three arrows:
\[
\xy
(-40,0)*+{\Gamma =}="n1";
(-15,0)*++[o][F]{1}="1"; 
(15,0)*++[o][F]{2}="2"; 
(45,0)*++[o][F]{3}="3"; 
{\ar@/_1.5pc/ "1";"2"};
{\ar@{->}_{} "2";"3"};
{\ar@/_1.5pc/ "2";"1"};
\endxy
\]
The graph has no non-trivial symmetries.  Nonetheless, the existence
of the subsystems $(\Delta, F|_\Delta)$, $(\Delta', F|_{\Delta'})$ and
the whole diagram of the dynamical systems \eqref{eq1.2.2} can be
deduced from certain properties of the graph $\Gamma$.  There are two
surjective maps of graphs:
\[
\varphi: \Gamma \to \xymatrix{ *+[o][F]{} \ar@(dr,ur)}
\]
and 
\[
\psi: \Gamma \to \xy
(-15,0)*++[o][F]{a}="1"; 
(15,0)*++[o][F]{b}="2"; 
{\ar@/_1.5pc/ "1";"2"};
{\ar@/_1.5pc/ "2";"1"};
\endxy,
\]
with $\psi $ defined on the vertices by $\psi(2) = b$, $\psi(1) = a =
\psi(3)$, and one embedding
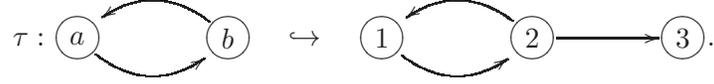
\[
\tau: 
\xy 
(-10,0)*++[o][F]{a}="1"; 
(10,0)*++[o][F]{b}="2"; 
{\ar@/_1.2pc/ "1";"2"};
{\ar@/_1.2pc/ "2";"1"};
\endxy \quad
\hookrightarrow \quad 
\xy
(-10,0)*++[o][F]{1}="1"; 
(10,0)*++[o][F]{2}="2"; 
(30,0)*++[o][F]{3}="3"; 
{\ar@/_1.2pc/ "1";"2"};
{\ar@{->}_{} "2";"3"};
{\ar@/_1.2pc/ "2";"1"};
\endxy.
\]
We can collect all of these maps into one diagram
\begin{equation}\label{eq:Gdiag}
\xy
(-17, 0)*+{ \xymatrix{ *+[o][F]{} \ar@(dr,ur)}}="1";
(15,0)*+{ \xy
(-9,0)*++[o][F]{1}="1"; 
(9,0)*++[o][F]{2}="2"; 
(24,0)*++[o][F]{3}="3"; 
{\ar@/_1.2pc/ "1";"2"};
{\ar@{->}_{} "2";"3"};
{\ar@/_1.2pc/ "2";"1"};
\endxy}="2";
(68, 0)*+{\xy 
(-9,0)*++[o][F]{a}="1"; 
(9,0)*++[o][F]{b}="2"; 
{\ar@/_1.2pc/ "1";"2"};
{\ar@/_1.2pc/ "2";"1"};
\endxy }="3";
  {\ar@{->}_{\varphi} (-10,0); (-20, 0)};
 {\ar@{->}_{\tau} (50,1); (38,1)};
  {\ar@{->}_{\psi} (38,-1); (50, -1)}; 
\endxy
\end{equation}

A comparison of \eqref{eq1.2.2} and~\eqref{eq:Gdiag} evokes a  pattern:
for every map which intertwines dynamical systems in~\eqref{eq1.2.2},
there is a corresponding map of graphs in~\eqref{eq:Gdiag} with the
arrows reversed, and vice versa. \\

The same pattern holds when we replace the vector space $\mathbb{R}^n$
by an arbitrary manifold $M$.  Given a pair of manifolds $U$ and $N$, we
think of a map $X:U\times N\to TN$ with $X(u,n)\in T_nN$ as a control
system with the points of $U$ controlling the the dynamics on $N$.
Now consider a vector field 
\[
F:M^3 \to T(M^3) = TM\times TM \times TM
\]
of the form
\[
F(x_1, x_2, x_3) = (f(x_2, x_1), f(x_1, x_2), f(x_2, x_3))
\]
for some control system
\[
f: M\times M \to TM, \quad \textrm{with} \quad f(u,v)\in T_v M.
\]
Then once again the three maps of graphs in the
diagram~\eqref{eq:Gdiag} give rise to maps of dynamical systems
\begin{equation}\label{eq.EqDiag}
\xy
(-40, 0)*+{(\Delta_M, F|_{\Delta_M})}="1";
(0,0)*+{(M^3, F)}="2";
(10,1)*+{}="2+"; 
(10,-1)*+{}="2-"; 
(42, 0)*+{(\Delta'_M, F|_{\Delta'_M})}="3";
(30, 1)*+{}="3+";
(30, -1)*+{}="3-";
 {\ar@{^{(}->} "1";"2"};
 {\ar@{_{(}->} "3+";"2+"};
 {\ar@{->}_{\pi} "2-";"3-"}; 
\endxy
\end{equation}

\noindent
What accounts for the patterns we have seen? Notice that the dynamical
systems (\ref{eq.EqDiag}) are constructed out of {\em one} control
system $f:N\times N\to TN$.  At the same time, in each of the graphs
in (\ref{eq:Gdiag}), every vertex has exactly {\em one} incoming arc.
This is not a coincidence.  The rough idea for the technology which
generalizes this example is this: if we have a dynamical system made
up of repeated control system modules whose couplings are encoded in
graphs, then the appropriate maps of graphs lift to maps of dynamical
systems. Making this precise requires a number of constructions and
theorems; these make up the bulk of this paper.

The first construction that we need is that of a phase space functor
$\mathbb{P}$ that consistently assigns to the sets of vertices of our
graphs products of manifolds.  It turns out to be naturally contravariant:
\[
\mathbb{P}: {\mathsf{Graph}}^{op} \to \mathsf{Man}.
\]
Here $\mathsf{Graph}$ denotes the category of finite directed graphs
and $ \mathsf{Man}$ the category of finite dimensional smooth
manifolds.\\

Associated to each node $a$ of a directed graph $\Gamma$ there is a
set of edges of $\Gamma$ with target $a$.  These edges form the
``input tree'' $I(a)$ of $a$, which is itself a directed graph.  With the
help of the phase space functor $\mathbb{P}$, one can associate---to each
such input tree $I(a)$---a vector space $\mathsf{Ctrl}(I(a))$ of control
systems.  These control systems are easy to describe; namely, the manifold
assigned to the leaves (which is itself a product of the individual manifolds associated to each leaf) controls the dynamics on the manifold assigned
to the root.

The collection of input trees and their isomorphisms of a given graph
$\Gamma$ form a groupoid $G(\Gamma)$.  This groupoid acts on the
vector spaces of control systems attached to the input trees of the
graph. We thus have a groupoid representation
\[
\mathsf{Ctrl}_\Gamma\colon G(\Gamma) \to \mathsf{Vect},
\]
where $\mathsf{Vect}$ is the category of (not necessarily finite-dimensional) real vector spaces and linear maps.  It is natural to
think of the limit of the functor $\mathsf{Ctrl}$ as the space
$\mathbb{V}\Gamma$ of invariants of the representation.  We think of
$\mathbb{V}\Gamma$ as the collection of {\em virtual} groupoid-invariant
vector fields on the phase space $\mathbb{P}(\Gamma)$.  (We'll explain
the meaning of the word ``virtual'' shortly.)  It is not hard to check
that the graphs in (\ref{eq:Gdiag}) produce the dynamical systems in
(\ref{eq.EqDiag}).  The story with the maps is a bit more complicated.
\\

To extend the assignment $\Gamma\mapsto \mathbb{V}\Gamma$ to a functor
we need to restrict ourselves to those morphisms of graphs that preserve the
input trees.  We call such maps of graphs {\em
  \'etale}\footnote{Equivalently a map $\varphi:\Gamma\to \Gamma'$ of
  directed graphs is {\em \'etale} if for any vertex $a$ of $\Gamma$
  and any edge $e'$ of $\Gamma'$ ending at $\varphi(a)$ there is a
  unique edge $e$ of $\Gamma$ ending at $a$ with $\varphi (e) = e'$.
  \label{foot:etale}}. We have a subcategory $\et{\mathsf{Graph}}$ of
the category of directed graphs and a contravariant functor
\[
\mathbb{V}\colon\op{(\et{{\mathsf{Graph}}})}\to \mathsf{Vect}.
\]
which extends the assignment $\Gamma\to \mathbb{V}\Gamma$ to a contravariant functor on $\et{\Graph}$.

The elements of $\mathbb{V}(\Gamma)$ are not, strictly speaking, vector
fields.  But for each graph $\Gamma$ there is a linear map
\[
S= S_\Gamma\colon\mathbb{V}(\Gamma)\to \chi (\mathbb{P} (\Gamma)),
\]
where $\chi (\mathbb{P} (\Gamma))$ denotes the space of vector fields on
the manifold $\mathbb{P} (\Gamma)$. One may think of the image of $S$
as the space of $G(\Gamma)$ invariant vector fields on $\mathbb{P}
(\Gamma)$.  In general the map $S$ need not be injective.  We are now
in position to state the first result that explains the maps in
(\ref{eq.EqDiag}):
\\

\noindent {\bf Theorem.}\quad
The functor $\mathbb{V}$ and the collection of maps $\{S:
\mathbb{V}(\Gamma)\to \chi(\mathbb{P}(\Gamma))\}$ are compatible: for
  any virtual groupoid invariant vector field $w\in
  \mathbb{V}(\Gamma')$ and any \'etale map of graphs
  $\varphi:\Gamma\to \Gamma'$ we have a map of dynamical systems
\[
\mathbb{P}\varphi: (\mathbb{P}(\Gamma'), S(w))\to (\mathbb{P}(\Gamma),
S((\mathbb{V}(\varphi)\, w)).
\]
\mbox{}\\

In short, given any \'etale map between two colored graphs (e.g.~the graph maps given in~\eqref{eq:Gdiag}), there is a gadget  which intertwines all dynamical systems on the phase spaces associated to these graphs (the maps intertwining the dynamical systems in~\eqref{eq.EqDiag}).

Finally, the compatibility of $\mathbb{V}$ and $S$ can be expressed
succinctly as follows.  Let $\mathcal{V}$ denote the category of
elements of the functor $\mathbb{V}$.  The objects of $\mathcal{V}$
are pairs $(\Gamma, w)$ where $\Gamma$ is a graph and $w\in
\mathbb{V}\Gamma$ is a virtual groupoid invariant vector field.  We
have a fibration of categories $\pi: \mathcal{V}\to
\op{(\et{\Graph})}$. The assignment $S$ extends to a functor
$S: \mathcal{V} \to \mathsf{DynSys}$, and the diagram
\begin{equation}\label{eq:star}
\xymatrix{
{\mathcal V}  \ar[r]^{S \quad\quad } \ar[d]_{\pi}
&  \mathsf{DynSys} \ar[d]^{U}  \\
 \op{(\et{\Graph})}  \ar[r]^{\quad\mathbb{P} }
&  \mathsf{Man} }
\end{equation}
commutes.  The objects of the category $ \mathsf{DynSys}$ are pairs
$(M,X)$ consisting of a manifold $M$ and a vector field $X$ on $M$.
The functor $U$ is the functor that forgets the vector field: $U(M,X)
= M$.\\

One final comment: the story above is not complete, since it treats
all the nodes of the graphs, and all the arcs of the graphs, as being
the same.  In general one would want to consider the case where different vertices in the graph could correspond to different manifolds, and additionally, not all controls should be interchangeable.  This is
taken care of by passing to a relative version of the theory outlined
above.  More concretely, we fix a graph $C$ of ``colors'', choose an assignment
of (possibly) different manifolds to different nodes of $C$, and an assignment of (possibly) different control systems to different edges of $C$.  The
category $\et{\Graph}$ is then replaced by the slice category
$\et{(\mathsf{Graph}/C)}$.  We give a concrete example of the kind of dynamical system which requires such colors in Example~\ref{eg:color}.

\subsection{Structure of paper}

The paper is organized as follows.  In Section~\ref{sec:Euc} we
develop the relevant mathematics under the assumptions that all of the
manifolds are Euclidean (i.e. linear) spaces and the maps are smooth,
i.e., the setting we consider are modular ODEs defined on copies of $\R^n$.
This assumption allows for a direct comparison of our results with the
results of Golubitsky et al.~\cite{Golubitsky.Stewart.06}. It will also
be used in a subsequent paper on groupoid-compatible numerical
methods.  In Section~\ref{sec:FinVect} we specialize the results of
Section~\ref{sec:Euc} to the case where all the vector fields and maps
are linear.  We plan to use the results of this section in our
subsequent work on stochastic systems.  In Section~\ref{sec:Man} we
generalize the theory of Section~\ref{sec:Euc} to allow our phase
spaces to be manifolds.  In Section~\ref{sec:inv} we study groupoid
invariant vector fields in the case where the underlying graph has
nontrivial group symmetries and prove that the space of groupoid
invariant vector fields is contained in the space of group-invariant
vector fields.  Section~\ref{sec:2slice} contains several technical
results about categorical limits.  Finally in the Appendix we review
the minimal amount of category theory that we need in this paper.

\subsection*{Acknowledgments} We thank Charles Rezk, Bertrand Guillou,
and Matthew Ando for a number of useful conversations.

\section{Groupoid invariant vector fields on Euclidean spaces}\label{sec:Euc}
\subsection{Introduction}\mbox{}
\\[-8pt]

\noindent
In this section we present the theory sketched in the introduction in
the setting of ordinary differential equations, that is, of vector
fields on Euclidean spaces.  As we mentioned in the introduction, we
tackle Euclidean spaces first to ease a comparison of our results with
those of Golubitsky et al.~\cite{Golubitsky.Stewart.06}, and to allow
us to build a setting for developing groupoid-compatible numerical
methods among other reasons.

\begin{empt}\label{empt:euc}
  Any finite dimensional vector space $V$ has a canonical structure of
  a second countable Hausdorff manifold. We will refer to this
  manifold as the {\em Euclidean space} $V$. Consequently it makes
  sense to talk about smooth vector fields on $V$.  Moreover, there is
  a canonical trivialization of the tangent bundle $TV$ of $V$:
\[
V\times V \to TV, \quad (x, v)\mapsto \left.\frac{d}{dt}\right|_0 (x + tv),
\]
which allows us to identify the space of smooth maps $C^\infty (V,V)$
from $V$ to itself with the space $\chi (V)$ of vector fields on $V$.
Explicitly
\[
C^\infty(V,V) \ni f \mapsto X_f \in \chi(V), 
\quad X_f (v) = \left(v,  \left.\frac{d}{dt}\right|_0 (v + tf(v))\right).
\]
\end{empt}

If $V$ and $W$ are two {\em different} vector spaces, then a map from $V$ to
$W$ is not (cannot be naturally identified with) a vector field on
$W$, but it can be identified with a {\em control system} on $W$.

\begin{definition} A {\em control system} on a manifold $M$ is a pair
$(p\colon Q\to M, F\colon Q\to TM)$, where
\begin{enumerate}
\item $p\colon Q\to M$ is a surjective submersion and
\item $F\colon Q\to TM$ is a smooth map with $F(q) \in T_{p(q)} M$ for
  all $q\in Q$.
\end{enumerate}
(cf., for example, \cite{Pappas}). In particular, the following diagram 
$$\xy (-10, 6)*+{Q}
="1"; (6, 6)*+{TM} ="2"; (6,-10)*+{M}="3"; {\ar@{->}_{ p}
  "1";"3"}; {\ar@{->}^{F} "1";"2"}; {\ar@{->}^{\pi} "2";"3"}; \endxy
$$
commutes, where $\pi$ is the canonical projection from the tangent
bundle to its base.
\end{definition}

\begin{example}
  Let $Q=U\times M$ for some manifold $U$, and $p$ be the projection
  onto the second factor.  In this case, the control system consists
  of the {\em phase space} $M$ and the {\em controls} $U$; the
  function $F:Q=U\times M\to TM$ defines a nonautonomous ODE of the
  form
\begin{equation*}
  \dot x = F(u,x),\quad(u,x)\in U\times M.
\end{equation*}
\end{example}

\begin{remark}\label{remark:vec-ctrl}
  Any vector field $X\colon M\to TM$ is a control system $(id\colon
  M\to M, X\colon M\to TM)$ --- the system with {\em trivial}
  controls.
\end{remark}
\begin{remark}\label{empt:CT}
 For a fixed surjective submersion $p:Q\to M$, the space
\[
\CT(p\colon Q\to M):= \{ F\colon Q\to TM\mid F(q) \in T_{p(q)}M \textrm{ for all }q\in M \}
\]
of all smooth {\em control systems supported on} $p$ is an infinite
dimensional vector space.
\end{remark}
\begin{notation}\label{empt:2.1.3}
We identify a smooth map  $f\colon V \to W$ between two Euclidean spaces  with the control system
$(pr_2\colon V\times W\to W, F\colon V\times W \to TW)$ where
\begin{enumerate}
\item $pr_2\colon  V\times W \to W$ is the projection on the second factor and
\item $F\colon V\times W\to W$ is defined by
\[
F(v,w) = (w, f(v)) \in T_{w} W \subset TW = W\times W.
\]
\end{enumerate}
In other words given two Euclidean spaces $V$ and $W$ we have a
canonical embedding
\begin{equation} \label{eq:2.1.4}
\Hom_\Euc(V, W) \hookrightarrow \CT(V\times W\stackrel{pr_2}{\to} W), \quad f\mapsto (pr_2, F)
\end{equation}
where, as above, $\CT(V\times W\stackrel{pr_2}{\to} W)$ denotes the
space of all control systems supported by the projection $pr_2\colon
V\times W\to W$, and $\Hom_\Euc(V, W) = C^\infty(V,W)$ is the space of
infinitely differentiable maps from $V$ to $W$.
\end{notation}

\begin{definition}[Maps of control systems] \label{def:Maps of control systems}
Let $\{ p_i\colon Q_i \to M_i, F_i\colon Q_i\to TM_i\}$, $i=1,2$, be a
pair of control systems.  A {\em morphism of control systems} from $\{
p_1\colon Q_1\to M_1, F_1\colon Q_1\to TM_1\}$ to $\{ p_2\colon Q_2\to
M_2, F_2\colon Q_2\to TM_2\}$ is a pair of smooth maps $\varphi\colon
Q_1\to Q_2$, $\phi\colon M_1\to M_2$ such that the two diagrams
\[
\xymatrix{ 
Q_1\ar[r]^{\varphi} \ar[d]^{p_1} & Q_2\ar[d]^{p_2}\\
M_1 \ar[r]_{\phi} & M_2
}
\quad \textrm{and} \quad 
\xymatrix{ 
Q_1\ar[r]^{\varphi} \ar[d]^{F_1} & Q_2\ar[d]^{F_2}\\
TM_1 \ar[r]_{d\phi} & TM_2
}
\]
commute.
\end{definition}

Note that if the two control systems in question are vector fields
$X_i\colon M_i \to TM_i$ (i.e.~have trivial controls), then the
morphism from the first to the second is a single map $\phi\colon
M_1\to M_2$ with
\[
d\phi \circ X_1 = X_2 \circ \phi.
\]
That is, $X_1$ and $X_2$ are $\phi$-related.

Now suppose we have a smooth map $(X_1, \cdots, X_n)\colon V_1\times
\ldots \times V_n \to V_1\times \ldots \times V_n$ from a product of
Euclidean spaces $V_1, \ldots, V_n$ to itself thought of as a vector
field.  Then for each index $j$ the pair
\[
(p_j\colon V_1\times \cdots \times V_n \to V_j,\quad X_j\colon V_1\times \ldots
\times V_n \to V_j)
\]
is a control system.  In other words the product $V_1\times \ldots
\times V_n$ controls the dynamics on each factor $V_j$ by way of the
components $X_j$ of the vector field $X$.  Alternatively we can view
$X=(X_1,\ldots X_n)$ as a collection of interacting control systems.

Suppose we know {\em a priori} that the dynamics on a factor $V_j$ is
controlled not by the full product $\prod V_i$ but by some subproduct
$Q_j:= V_{i_1}\times \cdots V_{i_k}$ (the number $k$ and the indices
$i_1, \ldots i_k$ depend on $j$). That is, suppose we can factor
$X_j\colon \prod V_i \to V_j$ through the natural projection $\prod
V_i \to V_{i_1}\times \cdots V_{i_k}$.  Then we can encode these
facts---literally, which submodule controls which---by way of a {\em
  directed graph} (q.v. \ref{def:directed graph}).  Given a collection
$X= (X_1, \ldots, X_n)$ of interacting control systems on $V_1\times
\cdots V_n$, we encode the dependence by a directed graph $\Gamma$ in
a natural way: the dynamics on the Euclidean space $V_j$ depends only
on the product $V_{i_1}\times \cdots V_{i_{k}}$ if and only if the
graph $\Gamma$ has arrows $i_1\to j, \dots, i_{k}\to j$.

Next suppose we know {\em a priori} that some control systems $X_j$
above are ``the same.''  That is, the ``control modules'' are repeated
through the system.
Suppose further that some of the factors $V_{i_\ell}$ in the
control/state Euclidean spaces $Q_j= V_{i_1}\times \cdots V_{i_k}$ are
repeated and that $X_j$'s are invariant under the permutation of the
repeated factors.  What structure (in addition to the directed graph
$\Gamma$) do we need to encode these assumptions? And what is the
resulting spaces of these interacting control systems on $\prod V_i$
with the ``symmetric modularity'' described above?  Our answer is a
reworking and a substantial extension of ideas of Golubitsky et
al.\ \cite{Golubitsky.Pivato.Stewart.04},
\cite{Golubitsky.Stewart.Torok.05}.

The construction we are about to present requires a number of
ingredients, which we develop throughout the rest of
Section~\ref{sec:Euc}.  The remainder of this section is broken up
into two themes: in Subsections~\ref{sec:colored}---\ref{sec:Sgamma}
we define the structure we study; in
Subsections~\ref{sec:gvg}---\ref{sec:VEuc} we state and prove the main
theorems justifying the construction and exhibiting the utility of
this formalism.  In more detail:
\begin{enumerate}
\item Construction of objects of study:
\begin{enumerate}
\item Colored graphs, Subsection~\ref{sec:colored};
\item Phase space functors $\PP$, Subsection~\ref{sec:PP};
\item The groupoid of colored trees $\FinTree/C$,
  Subsection~\ref{sec:coloredtree};
\item The control functor $\Ctrl\colon \FinTree/C\to \Vect$, Subsection~\ref{sec:controlfunctor};
\item Input trees of a colored graph $\Gamma$ forming a groupoid $G(\Gamma) \subset \FinTree/
C$, Subsection~\ref{sec:groupoid};
\item The space of groupoid-invariant vector fields $\V(\Gamma)$ as a limit of
$\Ctrl|_{G(\Gamma)}$, Subsection~\ref{sec:Vgamma};
\item The relation of $\V(\Gamma)$ to vector fields on the phase space
  $\PP\Gamma_0$, Subsection~\ref{sec:Sgamma}.
\end{enumerate}
\item Results:
\begin{enumerate}
\item Group-invariant versus groupoid-invariant vector fields$^*$
  (this section may be skipped on the first reading),
  Subsection~\ref{sec:gvg};
\item The assignment $\Gamma \mapsto \V(\Gamma)$ extends to a
  contravariant functor $\V\colon \op{(\FinGraph/C)_{et}}\to \Vect$,
  Subsection~\ref{sec:functor};
\item Virtual groupoid invariant vector fields and dynamics on
  Euclidean spaces, Subsection~\ref{sec:sync};
\item The category of elements of the functor $\V$ as a combinatorial
  category of groupoid invariant vector fields,
  Subsection~\ref{sec:VEuc}.
\end{enumerate}
\end{enumerate}
Finally, we end the section with an extended example in
Subsection~\ref{sec:extex}.  We now take them up one at a time.

\subsection{Colored graphs}\label{sec:colored}
\mbox{}
\\[-8pt]

\noindent
We start by fixing the definition of a directed graph and of morphisms
(maps) of directed graphs.
\begin{definition}\label{def:directed graph}
  A {\em directed graph} (or, in this manuscript, {\em graph})
  $\Gamma$ consists of two sets $\Gamma_1$ (of arrows, or edges),
  $\Gamma_0$ (of nodes, or vertices) and two maps $s,t\colon
  \Gamma_1 \to \Gamma_0$ (source, target):
\[
\Gamma = \{s,t\colon \Gamma_1 \to \Gamma_0\}.
\]
We allow the possibility of $\Gamma_1$ being empty. We do not assume
that $\Gamma_0$ is finite.  We write $\Gamma = \{\Gamma_1\toto\Gamma_0\}$.
A graph $\Gamma = \{\Gamma_1\toto\Gamma_0\}$ is {\em finite} if the
sets $\Gamma_0, \Gamma_1$ of its  nodes and edges are finite.  A {\em map of graphs} $\varphi\colon \Gamma\to \Gamma'$ is a pair of maps of sets
$\varphi_1\colon  \Gamma_1\to \Gamma' _1$ and $\varphi_0\colon  \Gamma_0\to
\Gamma'_0$ so that the diagram
\[
\xymatrix{
  \Gamma_1\ar[r]^{\varphi_1} \ar[d]_{(s,t)} & \Gamma'_1\ar[d]^{(s,t)}\\
  \Gamma_0\times \Gamma_0 \ar[r]_{(\varphi_0, \varphi_0)} & \Gamma'_0
  \times \Gamma'_0 } 
\] 
commutes.
\end{definition}

In order to keep track of various types of Euclidean spaces associated to
the nodes and of the types of interactions we want our graph $\Gamma$
``colored.''  We note that the ``colored'' graphs defined below are
not the standard colored graphs of the graph theory literature.
Rather they are a variant of the colored graphs of Golubitsky et
al.\ ({ op.\ cit.}) who define a {\em coloring} of a graph $\Gamma =
\{\Gamma_1 \toto \Gamma_0\}$ to be a pair of compatible equivalence relations on
the spaces $\Gamma_1$ of arrows and $\Gamma_0$ of nodes of the graph.
But an equivalence relation $\sim$ on a set $X$ is the same thing as
the quotient map $X\to X/\sim$.  We find thinking in terms of maps
rather than equivalence relations more natural, powerful and flexible.
In particular it allows us to think of all compatibly colored graphs
as a category.  Thus our definition is:

\begin{definition}\label{2.5}
  Fix a directed graph $C$ (``colors'').  A \emph{graph colored by
    $C$} is a map of graphs $\varphi\colon \Gamma \to C$.  A {\em map
    of colored graphs} $f\colon (\varphi\colon \Gamma\to C) \to
  (\varphi'\colon \Gamma'\to C)$ is a map of graphs $f\colon \Gamma
  \to \Gamma'$ with $\varphi'\circ f = \varphi$.  That is, we have a
  commuting triangle:
\[
\xy (-6, 6)*+{\Gamma}
="1"; (6, 6)*+{\Gamma'} ="2"; 
(0,-4)*+{C }="3"; 
{\ar@{->}_{ \varphi}  "1";"3"}; 
{\ar@{->}^{f} "1";"2"}; 
{\ar@{->}^{\varphi'} "2";"3"}; 
\endxy .
\]
\end{definition}
\begin{remark}\label{2.6}
  Given a node $x\in \Gamma$ we think of $\varphi(x) \in C_0$ as the
  color of $x$.  Given an edge $\gamma$ in $\Gamma$ we think of
  $\varphi(\gamma)$ as the color of $\gamma$.  Since $\varphi$ is a
  map of graphs, the colors of edges and nodes are automatically
  compatible (q.v. Definition~\ref{def:directed graph}).  Note that
  our colored graphs are colored graphs in the sense of Definition
  5.1(f) in~\cite{Golubitsky.Stewart.06}.
\end{remark}

\begin{remark} \label{rm:slice-cat} The category of graphs colored by
  a graph $C$ is the slice category $\Graph/C$ (see
  Definition~\ref{empt:slice-cat} in the Appendix).  In the discussion
  that follows the color graph $C$ is fixed.  We often write $\Gamma$
  for an element of the slice category $\Graph/C$ and $f\colon \Gamma
  \to \Gamma'$ for a morphism in $\Graph/C$ (with all the maps to $C$
  suppressed in the notation).
\end{remark}

\begin{example}\label{eg:color}
Imagine that we wanted to consider the vector field on $(\R^n)^4$ given by
\begin{equation*}
x_1' = f(x_2),\quad x_2' = g(x_1),\quad x_3' = f(x_2),\quad x_4' = g(x_3),
\end{equation*}
where $f,g\colon \R^n\to\R^n$ are smooth functions.  We want to encode the dependencies of different variables on one another, but also encode the fact that there are two types of control systems.  We then encode these in the graph 
$$ 
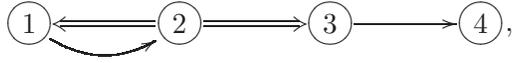
\xy
(-20,0)*++[o][F]{1}="1"; 
(0,0)*++[o][F]{2}="2"; 
(20,0)*++[o][F]{3}="3"; 
(40,0)*++[o][F]{4}="4"; 
{\ar@/_1.pc/ "1";"2"};
{\ar@{=>} "2";"1"};
{\ar@{=>} "2";"3"};
{\ar@{->} "3";"4"};
\endxy,$$
\end{example}
so that the double arrow corresponds to the $f$ control and the single arrow to the $g$ control.  One potentially misleading aspect of all of this is illustrated in the next example.  If we now consider the vector fields on $(\R^n)^4$ encoded by the graph
$$ 
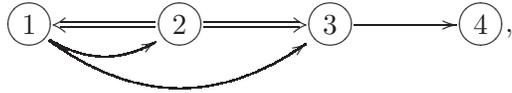
\xy
(-20,0)*++[o][F]{1}="1"; 
(0,0)*++[o][F]{2}="2"; 
(20,0)*++[o][F]{3}="3"; 
(40,0)*++[o][F]{4}="4"; 
{\ar@/_1.pc/ "1";"2"}; 
{\ar@/_2.pc/ "1";"3"};
{\ar@{=>} "2";"1"};
{\ar@{=>} "2";"3"};
{\ar@{->} "3";"4"};
\endxy,$$
then this corresponds to vector fields of the form
\begin{equation*}
x_1' = f(x_2), \quad x_2' = g(x_1),\quad x_3' = h(x_1,x_2),\quad x_4' = g(x_3),
\end{equation*}
where we assume no relationship between the function $h$ and the functions $f,g$.  We find that this can be misleading, since intuitively the inputs to vertex 3 are given by two arrows which singly correspond to controls $f$ and $g$, that the control from two vertices should somehow be related to $f$ and $g$, but in the current formalism they are not.  (We will consider formalisms the context of stochastic  dynamical systems~\cite{DeVille.Lerman.10.3} where we assume some relationship between the controls encoded by collections of arrows and those encoded by the individual arrows in that collection, but not here.)

\subsection{Phase space functors}\label{sec:PP}
\mbox{}\\[-8pt]

\noindent
We need a consistent way of assigning phase spaces to nodes and to
collections of nodes of various graphs colored by a graph
$C=\{C_1\toto C_0\}$.  Thus we need a functor from ``colored sets'' to
Euclidean spaces.  We construct such a functor $\PP$ in two steps.
But first we need some notation.

\begin{notation}
The symbol $\Finset$ denotes the category of finite sets and maps of
finite sets.  The symbol $\Finset/C_0$ denotes the slice category (q.v.\ Definition \ref{empt:slice-cat})
whose objects are maps of sets $\alpha\colon X\to C_0$, where $C_0$ is a fixed not necessarily finite set, $X$ is
a finite set and the morphisms are commuting triangles $\xy (-6, 6)*+{X}
="1"; (6, 6)*+{Y} ="2"; (0,-4)*+{C_0 }="3"; {\ar@{->}_{ \alpha}
  "1";"3"}; {\ar@{->}^{f} "1";"2"}; {\ar@{->}^{\beta} "2";"3"}; \endxy
$.
\end{notation}

\begin{notation}
  The symbol $\Euc$ denotes the category of Euclidean spaces: the
  objects are finite dimensional vector spaces thought of as smooth
  manifolds (q.v. \ref{empt:euc}) and morphisms are infinitely
  differentiable ($C^\infty$) maps.
\end{notation}

\noindent{\bf Step 1.} \quad We {\em choose} a function $\cP\colon C_0
\to \Euc_0$ from the set of colors of nodes $C_0$ to the collection
$\Euc_0$ of objects of the category $\Euc$ of Euclidean spaces. If we
think of $C_0$ as a discrete category (q.v. Example \ref{ex:discrete
  cat}) then $\cP$ is a functor from $C_0$ to $\Euc$.  We refer to
$\cP$ as the {\em phase space function.}\\

\noindent{\bf Step 2.} We use our choice of $\cP\colon C_0 \to \Euc$ to define
a contravariant phase space functor
\begin{equation}\label{eq:ph-s-funct}
\PP\colon  (\Finset/C_0)^{op} \to \Euc 
\end{equation}
as follows.  On objects of  $\Finset/C_0$ we set
\[
\PP X \equiv \PP (X \stackrel{\alpha}{\to} C_0): = 
\prod_{x\in X } \cP(\alpha (x)),
\]
the {\em categorical product} of the family $\{\cP (\alpha(x))\}_{x\in
  X}$ of the Euclidean spaces (q.v.\ Definition~\ref{def:product}).
Note that since the set $X$ is finite, the product $\prod_{x\in X}
\cP(\alpha(x))$ is a finite dimensional Euclidean space and thus an
object of the category $\Euc$ of Euclidean spaces.  In particular if
$X=\{1,2, \ldots, n\}$, then 
\[
\PP X = \cP(\alpha(1))\times \cdots
\times \cP (\alpha(n)),
\]
the Cartesian product of $n$ Euclidean spaces.

Given a morphism 
\begin{equation}\label{eq:triangle}
\xy (-6, 6)*+{Y} ="1"; (6, 6)*+{X} ="2"; (0,-4)*+{C_0 }="3";
    {\ar@{->}^{ \alpha} "2";"3"}; {\ar@{->}^{f} "1";"2"};
    {\ar@{->}_{\beta} "1";"3"}; \endxy 
\end{equation}
in the slice category $\FinSet/C_0$, we construct $\PP f\colon \PP
X\to \PP Y$ using the universal property of products: let $p_x \colon
\PP(X) \to \cP (\alpha(x))$, $x\in X$, and $q_y\colon \PP (Y)\to
\cP(\beta (y))$, $y\in Y$, denote the canonical projections.  Then for
any $y\in Y$ we have a map $p_{f(y)}\colon \PP X \to \cP (\alpha
(f(y))) = \cP (\beta (y))$.  Therefore, by the universal property of
product $\PP Y$ (q.v. \ref{def:product}) there is a unique map $\PP
f\colon \PP X \to \PP Y$ making the diagram
\begin{equation}\label{eq:2.4.3}
\xy
(-18, 10)*+{ \PP X} ="1"; 
(18, 10)*+{ \PP Y} ="4"; 
(-18,-6)*+{\cP (\alpha (f(y)))}="2";
(18,-6)*+{\cP (\beta (y))} ="3";
{\ar@{->}_{ p_{f(y)}} "1";"2"};
{\ar@{->}^{q_{y}} "4";"3"};
{\ar@{-->}^{\PP f} "1";"4"};
{\ar@{=} "2";"3"};
\endxy
\end{equation}
commute.

\begin{example}\label{ex:2.3.6}
  Suppose $C_0$ is a the two element set $\{c_1, c_2\}$,
  $\cP(c_1)$ is a Euclidean space $V$, $\cP(c_2)$ is a Euclidean space
  $W$, $X= \{x_1, x_2\}$, $Y= \{y_1, y_2\}$, $\alpha (x_i) = c_i$,
  $i=1,2$, $\beta (y_i) = c_1$ for all $i$ and $f(y_1) = f(y_2) =
  x_1$.  Then
\[
\PP X = \cP (\alpha (x_1)) \times \cP (\alpha (x_2)) = V\times W,
\]
\[
\PP Y = \cP (\beta (y_1)) \times \cP (\beta (y_2)) = V\times V,
\]
$p_{f(y_1)} = p_{f(y_2)} = p_{x_1}$, which is the canonical projection
$V\times W\to V$, while $q_{y_1}\colon V\times V \to V$ is the projection on
the first factor and $q_{y_2}\colon V\times V\to V$ is the projection on the
second.  Consequently $\PP f\colon  V\times W \to V\times V$ is the unique
map with $q_{y_1}\circ \PP f (v,w) = v$ and $q_{y_2}\circ \PP f (v,w) = v$.  That is
\[
\PP f (v,w) = (v,v).
\]
\end{example}

\begin{example}\label{ex:2.3.7}
  More generally, suppose $C_0$ is the set of natural numbers $\N$,
  $X=\{1,\ldots, N\}$, $Y=\{1, \ldots, M\}$, and $\alpha:X \to C_0$,
  $\beta:Y\to C_0$ and $f:Y\to X$ are three maps with $\alpha \circ f=
  \beta$. Then for any sequence $\{\cP(i)\}_{i\in \N}$ of Euclidean
  spaces we have
\[
\PP Y = \cP(\beta(1))\times \cdots \times \cP (\beta (M)), \quad
\PP X =  \cP(\alpha(1))\times \cdots \times \cP (\alpha(N)),
\]
$q_j: \PP Y \to \cP(\beta (j))$ is given by 
\[
q_j (w_1,\ldots, w_M) = w_j,
\]
$p_i: \PP X \to \cP(\alpha(i)))$ is given by 
\[
p_i(v_1,\ldots, v_N) = v_i,
\]
and $\PP f: \PP X\to \PP Y$ is defined by \eqref{eq:2.4.3}, that is by
\[
q_j (\PP f\, (v)) = p_{f(j)} (v) = v_{f(j)}.
\]
Hence
\begin{equation}\label{eq:PPf}
  \PP f(v_1,\dots,v_N) = (v_{f(1)}, \dots, v_{f(N)}).  
\end{equation}
\end{example}
\begin{remark} \label{rmrk:2.11}
\begin{enumerate}
\item Suppose $X\stackrel{\alpha}{\to} C_0, Y \stackrel{\beta}{\to}
  C_0 \in \FinSet/C_0$ are singletons: $X = \{x\}$, $Y=\{y\}$.  Then
  any morphism $f\colon Y\to X$ in $\FinSet/C_0$ has to map the unique
  element $y$ of $Y$ to the unique element $x$ of $X$ and we must have
  $\alpha (f(y)) = \beta (y)$.  Then $\PP X = \cP (\alpha (x)) = \cP
  (\beta (y)) = \PP Y$ and $\PP f$ is forced to be the identity map
  $id_{\PP X}$.

\item If $X\in \FinSet/C_0$ is empty then $\PP X $ is the one point
  Euclidean space $\{0\}$ (q.v.\ \ref{ex:A35}).

\item It is not hard to check that if $X= Y$ and $f$ is the identity
  map $id_X$ then $\PP (id_X) = id_{\PP X}$ (note that if $f$ is
  $id_X$ then $\alpha=\beta$).  Also, using the universal property of
  categorical products it is easy to check that if
\[\xy
(-15, 6)*+{Z} ="1"; 
(0, 6)*+{Y} ="2";
(15,6)*+{X} = "3";
(0,-6)*+{C}="4";
{\ar@{->}^{ \alpha} "3";"4"};
{\ar@{->}^{g} "1";"2"};
{\ar@{->}^{\beta} "2";"4"};
{\ar@{->}^{f} "2";"3"};
{\ar@{->}_{\gamma} "1";"4"};
\endxy 
\]
is a pair of composable morphisms in $\FinSet/C_0$, then
\[
\PP (f\circ g) = \PP g \circ \PP f.
\]
In other words $\PP$ is a contravariant functor
\begin{equation}
\PP \colon \op{(\FinSet/C_0)} \to \Euc.
\end{equation}
\end{enumerate}
\end{remark}

\begin{remark}\label{rmrk:Pf-is-linear}
  It follows from Example~\ref{ex:2.3.7} (in particular from
  \eqref{eq:PPf}) that for any map $f:Y\to X$ of finite sets over
  $C_0$ the map
\[
\PP f: \PP X \to \PP Y
\]
of Euclidean spaces is actually {\em linear.}
\end{remark}

\begin{remark}\label{rmrk:constr-of-P}
   The only property of the category $\Euc$ of Euclidean spaces that
   we used in construction the functor $\PP$ is that $\Euc$ has finite
   products.  Therefore the same construction works for the categories
   $\FinVect$ of finite dimensional vector spaces and linear maps and
   $\Man$ of manifolds. Thus given a choice of a functor $\cP\colon
   C_0\to \FinVect$ (again we think of $C_0$ as a discrete category)
   we get a contravariant functor
\[
\PP\colon  \op{(\FinSet/C_0)} \to \FinVect;
\]
 given a choice of a functor $\cP\colon C_0\to \Man$ we get a a
 contravariant functor
\[
\PP\colon  \op{(\FinSet/C_0)} \to \Man.
\]
We will use these maps in Sections~\ref{sec:FinVect} and
\ref{sec:Man}, respectively.
\end{remark}

 \begin{notation}
   Fix a directed graph $C=\{C_1 \toto C_0\}$.  The symbol
   $\FinGraph/C$ denotes the full subcategory of the slice category
   $\Graph/C$ whose objects are $(f\colon \Gamma\to C)\in (\Graph/C)_0$ with
   $\Gamma$ a {\em finite} graph.
\end{notation}

\begin{remark}\label{rmrk:PonFinGraph}
  There is an evident forgetful functor 
\[
F\colon \FinGraph/C \to \FinSet/C_0, \quad F(\Gamma \to C):= (\Gamma_0\to C_0),
\]
which forgets the arrows of a graph $\Gamma \to C$.  Following $F$ by
the phase space functor $\PP$ gives us a contravariant functor
$\op{(\FinGraph/C)}\to \Euc$.  We will abuse the notation and denote
the composite $\PP \circ F$ by $\PP$.  Thus
\[
\PP (\Gamma\stackrel{h}{\longleftarrow} \Gamma') = 
\PP \Gamma \stackrel{\PP h}{\longrightarrow}\PP \Gamma'\equiv
\PP \Gamma_0 \stackrel{\PP h}{\longrightarrow}\PP \Gamma'_0
\]
for any morphism $h\colon \Gamma'\to \Gamma$ in $\FinGraph/C$.
\end{remark}

\subsection{The groupoid of finite colored trees $\FinTree/C$}\label{sec:coloredtree}
\mbox{}\\[-8pt]

\begin{definition}  A {\em path} in a directed graph $\Gamma$ is a sequence of edges
with the matching sources and targets:
\[
a_0\stackrel{e_1}{\leftarrow}a_1\stackrel{e_2}{\leftarrow}\cdots \cdots 
\stackrel{e_n}{\leftarrow} a_n.
\]
That is, $a_i\in \Gamma_0$ are vertices,
$a_{i-1}\stackrel{e_i}{\leftarrow} a_i\in \Gamma_1$ are edges and,
additionally, $s(e_i) = a_i = t(e_{i+1})$.  Paths of length zero are
by definition vertices of $\Gamma$.
\end{definition}

\begin{definition}[Finite tree]\label{def:finite-tree}
A finite directed graph $T$ is a {\em tree} if there is exactly one
path between any two distinct vertices of $T$. The root of a tree $T$
is the unique vertex with no outgoing edges.  A vertex of $T$ that has
no incoming edges and which is not a root is called a {\em leaf} of
$T$.  We denote the singleton set consisting of the root of $T$ by
$\rt T$. We denote the set of leaves of a tree $T$ by $\lv T$.
\end{definition}

\begin{remark}
  Note that by our definition a graph $T$ consisting of a single vertex
  $a$ and no edges (that is, $T= \{\emptyset\toto \{a\}\}$) is a tree.
  The root of this tree is the vertex $a$ and the set of leaves is
  empty.
\end{remark}

\begin{definition}[The category of finite colored trees] Fix a graph
  $C$ of colors.  The category $\FinTree/C$ of {\em finite colored
    trees over $C$} is defined by
\begin{align*}
  (\FinTree/C)_0 = \mbox{finite trees colored by $C$} =
\{T\to C\mid T \textrm{ finite tree}\},\\
  \Hom_{\FinTree/C}(T,T') = \left\{ \left. \xy
(-6, 6)*+{T} ="1"; 
(6, 6)*+{T'} ="2";
(0,-4)*+{C }="3";
{\ar@{->}_{ \alpha} "1";"3"};
{\ar@{->}^{\sigma} "1";"2"};
{\ar@{->}^{\beta} "2";"3"};
\endxy \right| \sigma \textrm{ is an {\em isomorphism} of graphs over } C\right\}.
\end{align*}
Note that by definition the category $\FinTree/C$ is a groupoid
(q.v.\ Definition~\ref{def:groupoid}).
\end{definition}

\subsection{Control systems functor }\label{sec:controlfunctor}
\mbox{}\\[-8pt]

\noindent
\begin{definition}[The space of control systems associated to a tree]
\label{def:ctrl-Euc}
  Once again we fix a graph $C$ and a function $\cP\colon C_0\to
  \Euc$.  Let $T {\to}C$ be a finite tree over $C$.  We
  define the infinite-dimensional vector space of {\em control systems
    associated to the tree} $T \to C$ by
\[
\Ctrl(T) := \Hom_\Euc (\PP \lv T, \PP \rt T)
\equiv C^\infty (\PP \lv T, \PP \rt T).
\]
That is, $\Ctrl (T)$ is the space of all smooth
maps from the Euclidean space $\PP \lv T$ to the Euclidean space $\rt
T$. 
\end{definition}

\begin{remark}\label{rmrk:ctrl-euc}

\begin{enumerate}
\item By the identification in Notation~\ref{empt:2.1.3}, each map
  $f\in \Hom_\Euc (\PP \lv T, \PP \rt T)$ defines a control system
  $(\PP \lv T \times \PP \rt T \to \PP \rt T, \PP\lv T\times \PP\rt
  T\to T\PP \rt T)$
\item If the space of leaves $\lv T$ is empty (so that $T$ consists of
  a single vertex) then $\PP \lv T = \{0\}$ and $\Ctrl (T) = \Hom_\Euc (\PP \lv T,
  \PP \rt T) = \Hom_\Euc (\{0\}, \PP\rt T) =\{0\}$, the zero
  dimensional vector space.
\item Our definition only considers the top and bottom generations of
  the tree and ignores everything in the middle.  We will in fact
  only consider  one-generation trees below, so this
  restriction will not be important.
\end{enumerate}
\end{remark}

The assignment $T\mapsto \Ctrl(T)$ extends to a functor $\Ctrl\colon 
\FinTree/C \to \Vect$, from the category of finite trees over $C$ to the
category of infinite dimensional vector spaces as follows.
Let $\xy
(-6, 6)*+{T} ="1"; 
(6, 6)*+{T'} ="2";
(0,-4)*+{C }="3";
{\ar@{->}_{ \alpha} "1";"3"};
{\ar@{->}^{\sigma} "1";"2"};
{\ar@{->}^{\beta} "2";"3"};
\endxy $ be an isomorphism of trees.  Then we have two isomorphisms of
finite sets over $C_0$:
\[
\sigma|_{\lv T}\colon  \lv T \to \lv T' \quad \textrm{and} \quad
\sigma|_{\rt T}\colon \rt T \to \rt T'.
\]
Applying the phase space functor $\PP$ we get two diffeomorphisms:
\begin{equation*}
  \PP(\sigma|_{\lv T})\colon \PP \lv T'\to \PP \lv T,\quad   \PP(\sigma|_{\rt T})\colon \PP \rt T'\to \PP \rt T. 
\end{equation*}
Note that the second diffeomorphism is the identity map by
Remark~\ref{rmrk:2.11}.  Consequently we obtain the map
\begin{equation}\label{eq:1.17}
\Ctrl(\sigma)\colon  \Ctrl(T) \to \Ctrl (T')
\end{equation}
given by 
\begin{equation}\label{eq:1.18}
\Ctrl(\sigma) X \colon = X \circ \PP (\sigma|_{\lv T} )\colon  \PP \lv T' \to \PP
\rt T = \PP \rt T'
\end{equation}
for any $(X\colon  \PP \lv T \to \PP \rt T) \in \Ctrl(T) = \Hom_\Euc (\PP
\lv T, \PP \rt T)$, so that the following diagram commutes:
\begin{equation*}
\xy
(-18, 10)*+{\PP\lv T} ="1"; 
(18, 10)*+{\PP\rt T} ="2";
(-18,-10)*+{\PP\lv T'}="3";
(18,-10)*+{\PP \rt T'}="4";
{\ar@{->}^X "1";"2"};
{\ar@{->}^{\PP(\sigma|_{\lv T})} "3";"1"};
{\ar@{=}_{\PP(\sigma|_{\rt T}) = id } "4";"2"};
{\ar@{-->}_{\Ctrl(\sigma)X} "3";"4"};
\endxy
\end{equation*}
  We leave it to the reader to check that $\Ctrl$ preserves the
  composition of morphisms and thus a functor.

\subsection{Input  trees of colored graphs}\label{sec:groupoid}

\begin{definition}[Input tree]\label{def:input-tree}
  Given a vertex $a$ of a graph $\Gamma$ we define the {\em input
    tree} $I(a)$ to be the following graph:
\begin{itemize}
\item the vertices $I(a)_0$ of $I(a)$ is the set 
\[
I(a)_0 := \{a\}\sqcup \coprod t\inv (a),
\]
where $t\inv (a)$ is the set of arrows in $\Gamma$ with target $a$;
\item the edges $I(a)_1$ of $I(a)$ is the set of pairs
\[
I(a)_1 := \{(a, \gamma)\mid t(\gamma) =a \},
\]
\item  the source and target maps $I(a)_1\toto I(a)_0$ are defined by
\[
s(a, \gamma) = \gamma \quad\textrm{and}\quad  t(a, \gamma) = a.
\]
In pictures
\[
 \xy  (-8,0)*+{\gamma \bullet}="x";
(8, 0)*+{\bullet a}="z";
{\ar@/^0.5pc/^{(a,\gamma) } "x";"z"};
\endxy .
\]
\end{itemize}
It is easy to see that the graph $I(a)$ is a tree with
\[
\rt I(a) = \{a\},\quad \lv I(a) =  t\inv (a)
\]
\end{definition}
\begin{remark}\label{remark:xi}\label{empt:xi}
  We have a natural map of graphs $\xi\colon I(a)\to \Gamma$ given on
  arrows by $\xi (a, \gamma) = \gamma$.  Thus if $\Gamma \to C$ is a
  graph over $C$ then so are its input trees $I(a)$ for each vertex
  $a\in \Gamma_0$, and the composite
  $I(a)\stackrel{\xi}{\to}\Gamma\stackrel{c}{\to}C$ makes $I(a)$ into
  a graph over $C$.
\end{remark}

\begin{definition}[The groupoid $G(\Gamma)$] \label{def:groupoidGGamma}
Given a graph $\Gamma \to
  C$ over $C$ we define its {\em symmetry groupoid} $G(\Gamma)$ to be
  the groupoid with the set of objects $G(\Gamma)_0$ equal to the set
  of the input trees of $\Gamma$:
\[
G(\Gamma)_0 =\{ I(a) \mid a\in \Gamma_0\};
\]
and the sets of morphisms 
\[
\Hom_{G(\Gamma)}(I(a), I(b) ):= \{ \sigma\colon I(a) \to I(b) \mid
\sigma \textrm{ is an isomorphism of graphs over } C\}.
\]
In other words $G(\Gamma)$ is the full subcategory of $\FinTree/C$
with the set of  objects $\{I(a)\}_{a\in \Gamma_0}$.
\end{definition}

\subsection{Virtual groupoid invariant vector fields}\label{sec:Vgamma}
\mbox{}\\[-8pt]

\noindent
Given the groupoid of input trees $G(\Gamma)$ of a graph $\Gamma$ we
restrict the functor $\Ctrl$ to $G(\Gamma)$ and get a functor
\[
\Ctrl_\Gamma = \Ctrl|_{G(\Gamma)}\colon  G(\Gamma) \to \Vect.
\]
\begin{remark}
  For an input tree $I(a)$ of a graph $\Gamma \to C$ the space
  $\Ctrl(I(a))$ is the space of all smooth maps from the phase space
  $\PP \lv I(a)$ of leaves to the phase space of the root $\PP\{a\}$.
  Chasing through the definitions we see that $\PP \lv I(a)$ is a
  finite product of Euclidean spaces, one Euclidean space for each
  {\em arrow} pointing into the node $a$.  If all the arrows with
  target $a$ have distinct sources, then $\PP \lv I(a)$ is the same as
  the product of the appropriate vector spaces indexed by these
  sources, which is the definition of Golubitsky--Pivato--Stewart in
  \cite{Golubitsky.Pivato.Stewart.04}.  However, note that if there is
  a node $b$ of $\Gamma$ which is connected to $a$ by multiple arrows
  then our definition is different from the
  Golubitsky--Stewart--T\"or\"ok definition
  \cite{Golubitsky.Stewart.Torok.05}.  The difference can be seen in
  the Example~\ref{ex:1} below.
\end{remark}

\begin{definition}[Virtual groupoid invariant vector fields]\label{def:vgivf}
  Fix the graph of colors $C=\{C_1 \toto C_0\}$ and a phase space
  function $\cP\colon C_0 \to \Euc$. Given a graph $\Gamma$ over $C$
  we define the vector space of {\em virtual groupoid invariant vector
    fields} $\V(\Gamma)$ to be the limit of the functor $\Ctrl_\Gamma$
  (q.v.\ Definition~\ref{def:invariants} and the subsequent
  discussion):
\[
\V(\Gamma): = \lim (\Ctrl_\Gamma\colon  G(\Gamma) \to \Vect).
\]
\end{definition}
Note that since $\V(\Gamma)$ is a limit, it comes with a family of the
canonical projections
\[
\left\{\varpi_a\colon  \V(\Gamma) \to \Ctrl(I(a)) \right\}_{a\in\Gamma_0}.
\]
\begin{remark}\label{rmrk:2.8.3}
The vector space $\V (\Gamma)$ has the following concrete description.
Consider the product $\prod _{a\in \Gamma_0} \Ctrl(I(a))$ with its
canonical projections $p_b\colon \prod _{a\in \Gamma_0} \Ctrl(I(a))
\to \Ctrl(I(b))$, $b\in \Gamma_0$.  For a vector $X\in\prod _{a\in
  \Gamma_0} \Ctrl(I(a))$ we denote the $a$-th component $p_a(X)$ of
$X$ by $X_a$. Then
\[
\V(\Gamma) = \left\{\left. X\in \prod _{a\in \Gamma_0}
    \Ctrl(I(a))\,\,\right|\,\, \Ctrl (\sigma) X_a = X_b \textrm{ for
    all arrows } I(a)\stackrel{\sigma}{\to }I(b) \textrm{ in the
    groupoid } G(\Gamma)\right\}
\]
with the canonical projections $ \varpi_a\colon \V(\Gamma) \to
\Ctrl(I(a)) $ given by restrictions $p_a|_{\V(\Gamma)}$.  To check
that the collection $\{ \varpi_a\colon \V(\Gamma) \to
\Ctrl(I(a))\}_{a\in G_0}$ is in fact a limit of $\Ctrl_\Gamma$ it is
enough to check its universal properties, which we leave to the
reader.
\end{remark}

\begin{example}\label{ex:1}
  Let $C$ be a graph with one vertex $v$ and one edge: $C =
  \,\xymatrix{ *+[o][F]{} \ar@(dr,ur)}$\quad\quad .  Then any graph
  $\Gamma$ is canonically a graph over $C$ and the category of finite
  graphs over $C$ is isomorphic to the category of finite graphs.  Let
  $\cP\colon  C_0 =\{\circ\}\to \Euc$ be given by $\cP (\circ) = V$ for some
  Euclidean space $V$.  Let $\Gamma$ be the graph with two nodes and two
  arrows:
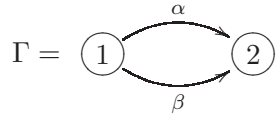
\[
\Gamma = \,\,
\xy
(-20,0)*++[o][F]{1}="1"; 
(0,0)*++[o][F]{2}="2"; 
{\ar@/_1.pc/_\beta "1";"2"};
{\ar@/^1.pc/^\alpha "1";"2"};
\endxy 
\]
We have $\Gamma_0 = \{1,2\}$ and $\PP\Gamma_0 = V\times V$.  The input
trees of the two nodes are 
\begin{equation*}
  I(1) = \xy(0,0)*++[o][F]{1};\endxy, \quad\quad\quad I(2) = \xy
(-20,5)*+{\{\alpha\}}="1a"; 
(-20,-5)*+{\{\beta\}}="1b"; 
(0,0)*++[o][F]{2}="2"; 
{\ar@{->}^{ \alpha}  "1a";"2"};
{\ar@{->}_{ \beta}  "1b";"2"};
\endxy .
\end{equation*}
Since $\lv(I(1)) = \emptyset$, we have $\Ctrl(I(1)) = \{0\}$.  The set
of leaves of the input tree $I(2)$ has two elements: $\lv(I(2)) =
\{\alpha,\beta\}$. Through the map $\xi$ (q.v.~Remark~\ref{remark:xi})
the leaves of $I(2)$ inherit the colors of their sources. Hence
$\PP\lv(2) = V\times V$.  Therefore $\Ctrl(I(2)) = \Hom_\Euc(V\times
V,V)= C^\infty(V\times V,V)$.  

The groupoid $G(\Gamma)$ has two objects: $I(1)$ and $I(2)$.  The only
nontrivial morphism of $G(\Gamma)$ is the map $\sigma:I(2)\to I(2)$
which exchanges $\alpha$ and $\beta$:
\[
G(\Gamma): \quad\quad
I(1)
 \quad
\xy
(1,3)*{}="A";
(1,-3)*{}="D";
"A"; "D" **\crv{(8,10) & (15,0)& (8,-10)}?(.99)*\dir{>}; 
 (0,0)*++{I(2)}="B";
(16,0)*++{\sigma}="C";
\endxy
\]
Consequently 
\[
(\Ctrl(\sigma) f)(v,w) = f(w,v) 
\]
for any $f\in \Ctrl(I(2))$ and any $(v,w)\in \PP \lv
I(2)$.  We conclude that
\begin{equation*}
  \V(\Gamma) = \{0\}\oplus C^\infty(V\times V,V)^{\ZZ_2} \simeq 
C^\infty (V\times V, V)^{\ZZ_2},
\end{equation*}
where $\ZZ_2$ denotes the two-element group.
Now consider the graph $\Gamma'$:
\begin{equation*}
\Gamma'=\,\,  
\xy
(-20,5)*++[o][F]{1a}="1a"; 
(-20,-5)*++[o][F]{1b}="1b"; 
(0,0)*++[o][F]{2}="2"; 
{\ar@{->}^{ \alpha}  "1a";"2"};
{\ar@{->}_{ \beta}  "1b";"2"};
\endxy
\end{equation*}
The groupoid $G(\Gamma')$ has three objects: the trees $I(1a)$, $I(1b)$ and $I(2)$.
It has three nontrivial arrows:
\[
G(\Gamma')= \quad\quad
\xy
(0,10)*++{I(1a)}="1"; 
(0,-10)*++{I(1b)}="2"; 
{\ar@/_1.pc/ "1";"2"};
{\ar@/_1.pc/ "2";"1"};
\endxy \quad
\xy
(1,3)*{}="A";
(1,-3)*{}="D";
"A"; "D" **\crv{(8,10) & (15,0)& (8,-10)}?(.99)*\dir{>}; 
 (0,0)*++{I(2)}="B";
(16,0)*++{\sigma}="C";
\endxy
\]
A calculation similar to the one above shows that 
\begin{equation*}
\V(\Gamma') =
 \{0\}\oplus  \{0\}\oplus   C^\infty(V\times V,V)^{\ZZ_2}.
\end{equation*}
\end{example}

\subsection{The relation of $\V(\Gamma)$ to vector fields on the phase
  space $\PP\Gamma_0$}\label{sec:Sgamma}
\mbox{}\\[-8pt]

\begin{empt}\label{not:chi}
  Recall that the space of vector fields $\chi(W)$ on a Euclidean
  space $W$ is canonically isomorphic to $\Hom_{\Euc}(W,W) \equiv
  C^\infty(W,W)$ (q.v.\ \ref{empt:euc}).
\end{empt}

Given a finite graph $\Gamma$ over a graph of colors $C$ and the phase
space function $\cP\colon C_0 \to \Euc$, the  functor $\PP$ assigns a
phase space $\PP \Gamma_0$ to the set $\Gamma_0$ of vertices of
$\Gamma$.  There exists a canonical {\em linear} map $S_\Gamma$ from the
space of virtual groupoid-invariant vector fields $\V(\Gamma)$ to the
vector space $\chi (\PP\Gamma_0)$ of vector fields on $\PP \Gamma_0$,
which we now define.

Since $\chi (\PP\Gamma_0) = \Hom_\Euc (\PP\Gamma_0, \PP\Gamma_0) =
\Hom_\Euc (\PP\Gamma_0, \prod_{a\in \Gamma_0} \PP \{a\})$ (q.v.\
Remark~\ref{rmrk:hom-into-product}), the space of vector fields on
$\PP \Gamma_0$ is canonically the product 
\[
\chi (\PP\Gamma_0) = 
\prod_{a\in \Gamma_0} \Hom_\Euc (\PP\Gamma_0, \PP \{a\}).
\]
To define the map into a product, it is enough to define a map into
its factors.  Thus we need maps $\V(\Gamma) \to \Hom_\Euc
(\PP\Gamma_0, \PP \{a\})$ for each vertex $a$ of the graph $\Gamma$.

For each input tree $I(a)$ we have a map of graphs $\xi\colon
I(a)\to\Gamma$ (q.v.~Remark~\ref{empt:xi}).  Therefore we have a map
$\xi|_{\lv I(a)}\colon \lv I(a)\to \Gamma_0$ of finite sets over
$C_0$.  Applying the phase space functor $\PP$ we get maps
\[
\PP(\xi|_{\lv I(a)})\colon  \PP\Gamma_0 \to \PP \lv I(a).
\]
The pullback by $\PP(\xi|_{\lv I(a)})$ gives
\[
(\PP(\xi|_{\lv I(a)}))^*\colon \Ctrl(I(a))\to \Hom_\Euc (\PP\Gamma_0, \PP \{a\}),\quad F\mapsto F\circ \PP(\xi|_{\lv I(a)}).
\]
By the universal properties of the product $\chi (\PP\Gamma_0)$ there a unique canonical  map 
$S_\Gamma\colon  \V(\Gamma) \to \chi (\PP \Gamma_0)$ making 
the diagram
\begin{equation} \label{eq:2.2}
\xy (-22, 9)*+{\V (\Gamma)}="1"; (24,9)*+{\chi (\PP \Gamma_0)}="2";
(-22, -10)*+{\Ctrl (I(a))}="3"; (24,-10)*+{\Hom_\Euc(\PP\Gamma_0, \PP
  \{a\})}="4"; {\ar@{-->}^{\exists ! S_\Gamma} "1";"2"};
    {\ar@{->}_{\varpi_a} "1";"3"}; {\ar@{->}^{p_a} "2";"4"};
    {\ar@{->}_{(\PP(\xi|_{\lv I(a)}) )^*\quad} "3";"4"}; \endxy
\end{equation}
commute (see Definition~\ref{def:limit}).  Here, as before, 
\begin{equation}\label{eq:defofvarpi}
  \varpi_a\colon  \V \Gamma\to \Ctrl ( I(a))
\end{equation}
 and 
\begin{equation}\label{eq:defofpa}
  p_a\colon \chi (\PP\Gamma_0) = \prod_{a'\in \Gamma_0} \Hom_\Euc(\PP\Gamma_0, \PP \{a'\}) \to 
\Hom_\Euc (\PP\Gamma_0, \PP \{a\})
\end{equation}
denote the canonical projections.

\begin{example}\label{ex:2}
  We return to the graphs $\Gamma$ and $\Gamma'$ of Example~\ref{ex:1}
  and the phase space function $\cP(\circ) =V$.  The space $\chi (\PP
  \Gamma_0)$ of vector fields on $\PP \Gamma_0$ is $C^\infty (V^2,
  V^2)$ and $\chi (\PP \Gamma'_0) = C^\infty (V^3, V^3)$.  Unraveling
  the definitions we see that
\[
S_\Gamma : \V(\Gamma) = C^\infty (V\times V, V)^{\ZZ_2} \to C^\infty (V^2,
V^2) = \chi (\PP \Gamma_0)
\]
is given by 
\[
(S_\Gamma (f))(v_1, v_2)  = (0, f(v_1, v_1))
\]
for all $(v_1, v_2)\in V\times V$.  This is because $\xi|_{\lv I(2)}$
is the map sending both nodes of $\lv I(2)$ to $\xy
(0,0)*++[o][F]{1}\endxy$.  In the case of $\Gamma'$ 
\[
S_{\Gamma'}: \V(\Gamma') = C^\infty (V\times V, V)^{\ZZ_2} \to C^\infty (V^3,
V^3) = \chi (\PP \Gamma'_0)
\]
is given by 
\[
(S_{\Gamma'} (f))(v_{1a}, v_{1b}, v_2)= (0, 0, f(v_{1a}, v_{1b})). 
\]
Note that $S_{\Gamma'}$ is injective, while $S_\Gamma$ is not:
$S_\Gamma$ vanishes on all the functions $f$ whose restriction to the
diagonal $\Delta_V \subset V\times V$ is zero.
\end{example}

\subsection{Group-invariant versus groupoid-invariant vector
  fields}\footnote{This subsection may be skipped on the first
    reading} \label{sec:gvg}
\mbox{}\\[-8pt]

\noindent
In laying out the main themes of the survey article
\cite{Golubitsky.Stewart.06} Golubitsky and Stewart list a number of
questions ``raised by the groupoid point of view.''  In particular
they ask (op.\ cit., p.~309):
\begin{quote}
``What are the analogies with the group case? When do these analogies fail
and why?''
\end{quote}
In this subsection we provide a partial answer and suggest an approach
to answering this question more fully. In the previous subsections we
constructed the space of virtual groupoid-invariant vector fields as ``fixed
points'' of a groupoid representation.  Just as in the case of groups,
there is a notion of a groupoid action that is more general than that
of a groupoid representation.  To recall the relevant definition we
need more notation.

\begin{notation}
  Given maps of sets $f\colon X\to Z$ and $g\colon Y\to Z$ the {\em
    fiber product} $X\times _{f,Z,g}Y$ is the set
\[
X\times _{f,Z,g}Y:= \{(x,y) \in X\times Y\mid f(x) = g(y)\}.
\]
\end{notation}
\noindent
(Note that the fiber product of $f$ and $g$ together with the
canonical projections $p_X\colon X\times _{f,Z,g}Y\to X$, $p_Y\colon
X\times _{f,Z,g}Y\to Y$ is the limit of a functor from the category
$\bullet \longrightarrow\bullet\longleftarrow \bullet$ with 3 elements
and 2 non-identity maps to the category $\Set$ of sets which sends the
non-identity arrows to $f$ and $g$ respectively; q.v.\
\ref{def:limit}.)

\begin{definition}
  A {\em action} of a groupoid $G=\{G_1\toto G_0\}$ on a set $X$ is a
  pair of maps $\an\colon X\to G_0$ (anchor) and $\mathtt{a}\colon
  G_1 \times_{s, G_0, \an}X\to X$ (action; $s\colon G_1\to G_0$ denotes
  the source map) so that
\begin{enumerate}
\item $\an (\mathtt{a} (g, x)) = t(g)$ for all $(g,x)\in G_1
  \times_{s, G_0, \an}X$,
\item $\mathtt{a}(1_{\an(x)}, x) = x$ for all $x\in X$ and
\item $\mathtt{a}(g_2g_1 , x) = \mathtt{a}(g_2, \mathtt{a}(g_1, x))$
  for all composable arrows $(g_2,g_1)\in G_1\times _{s, G_0, t}G_1$
  ($t\colon G_1\to G_0$ is the target map) and all $x\in X$ with
  $\an(x)= s(g_1)$.
\end{enumerate}
\end{definition}
It is not hard to see that a functor $\rho\colon G\to \Vect$ from a (small)
groupoid $G$ to the category $\Vect$ of vector spaces defines an
action of $G$ on the set $X=\bigsqcup_{c\in G_0} \rho (c)$.  The
anchor $\an\colon X\to G_0$ is defined by $\an(\rho(c)) = c$.  The action
$\mathtt{a}$ is given by
\[
\mathtt{a}(g,v) = \rho(g)v.
\]
There is, however, one crucial difference between {\em
  groupoid}-invariant vector fields and {\em group}-invariant vector
fields as they are studied in, for
example,~\cite{Golubitsky.Stewart.Schaeffer.book} or~\cite{Field}.  In
the case of the group-invariant vector fields on a phase space $M$ ($M$ is a
Euclidean space or, more generally, a manifold) the action of a group
$H$ on the space of vector fields  is {\em induced by an
  action of $H$ on $M$}:
\[
(h\cdot v)(m) := (Dh_M)_m ( v (h_M \inv (m)))
\]
for all group elements $h\in H$, points $m\in M$ and vector fields $v\in \chi (M)$.
Here $h_M\colon M\to M$ denotes the diffeomorphism defined by the action of $h$ on $M$:
\[
h_M (m) = h\cdot m.
\]
In the case of groupoid-invariant vector fields (q.v. \ref{def:vgivf})
the action of a groupoid $G(\Gamma)$ on the set of the associated
control systems $\bigsqcup_{a\in \Gamma_0} \Ctrl(I(a))$ is {\em not
  induced by any action of the groupoid $G(\Gamma)$ on the phase space
  $\PP \Gamma_0$.}  Consequently, the intuition acquired in the
studies of group-invariant dynamical systems can be as much of a
hindrance as of help in thinking about groupoid invariant dynamical
systems.

We speculate that to fully answer the question of Golubitsky and Stewart
quoted above one would need, on one hand, to understand dynamical
systems with group-invariant vector fields where the group action {\em
  in not induced} by an action of the group on the phase space, and,
on the other hand, to study groupoid invariant vector fields where the
action of the groupoid on the vector fields {\em is induced} by the
action of the groupoid on the phase space.

\subsection{The assignment $\Gamma \mapsto \V\Gamma$ extends to a functor}
\label{sec:functor}\mbox{}\\[-8pt]

\noindent
A map $\varphi\colon \Gamma \to \Gamma'$ of graphs over a fixed graph
$C$ induces, for each node $a\in \Gamma_0$ a map of input trees
\[
\varphi_a\colon I(a) \to I (\varphi(a)), 
\]
 on arrows 
\begin{equation}\label{eq:varphi_a}
\varphi_a (a,\gamma) = (\varphi(a), \varphi(\gamma)). 
\end{equation}
However, it does not, in
general, induce a map from the groupoid $G(\Gamma)$ to the groupoid
$G(\Gamma')$ let alone a map between the spaces of virtual invariant
vector fields $\V(\Gamma)$ and $\V(\Gamma')$.

\begin{definition}
  A map $\varphi\colon \Gamma \to \Gamma'$ of graphs over a graph $C$ of
  colors is {\em \'etale} (or a {\em local isomorphism}) if, for each node
  $a$ of the graph $\Gamma$, the induced map
\[\varphi_a\colon  I(a) \to I (\varphi(a))\] 
of input trees defined by \eqref{eq:varphi_a} is an isomorphism of
graphs over $C$.
\end{definition}
\noindent
Note that the outgoing edges play no role in the definition of an
\'etale map.
\begin{remark}
  The notion of an \'etale map of directed graphs is not new and goes
  by many different names~\cite{Vigna1, Vigna2}. Calling it a ``local
  isomorphism'' could be misleading, since we only consider incoming
  arrows, and it might better be termed ``local in-isomorphism'' as
  in~\cite{Vigna1, Vigna2}.  In Higgins~\cite{Higgins} it is called a
  covering map, but since local isomorphisms of graphs are not
  necessary surjective on nodes, calling it a ``covering map'' may
  give a wrong impression.  Calling them ``\'etale'' is a lot shorter
  than calling them ``local isomorphisms.'' See~\cite{VignaWP} for an
  extensive discussion of the history of the notion and various
  contexts in which it arouse.
\end{remark}

\begin{remark}
  If a morphism $\varphi\colon \Gamma\to \Gamma'$ in $\Graph/C$ is \'etale
  than for each node $a$ of $\Gamma$ and each edge $\gamma'$ in
  $\Gamma'$ with $t(\gamma') = \varphi(a)$ there exists a unique edge
  $\gamma$ in $\Gamma$ with $\varphi(\gamma) = \gamma'$. Consequently given
  any path in $\Gamma'$ with the end node $f(a)$ there is a unique
  lift of this path to a path in $\Gamma$ with the end node $a$
  (q.v.\ footnote \ref{foot:etale} on p.~\pageref{foot:etale}).  We
  will not use the unique lifting property of paths in this paper, but
  it will play an important role in the subsequent paper on the
  groupoid-invariant discrete-time dynamics on networks.
\end{remark}

\begin{remark}
  Finite graphs over a fixed graph $C$, along with \'etale maps, form a
  subcategory of the category $\FinGraph/C$ of finite graphs over $C$,
  since the composition of \'etale maps is \'etale.  We call this
  subcategory $\et{(\FinGraph/C)}$.
\end{remark}

\begin{theorem}\label{thmA}
  The map $\V$ that assigns to each finite graph $\Gamma$ over $C$ the vector
  space of $G(\Gamma)$-invariant virtual vector fields on $\PP\Gamma_0$
  extends to a contravariant functor
\[
\V\colon \opet{(\FinGraph/C)} \to \Vect.
\]
\end{theorem}

We first prove

\begin{lemma}\label{lemma:ff}
  An \'etale map $\varphi\colon \Gamma\to \Gamma'$ of graphs over $C$ induces
  a fully faithful functor (q.v. Definition~\ref{def:full})
\[
G(\varphi)\colon G(\Gamma)\to G(\Gamma')
\]
between the corresponding groupoids.
\end{lemma}
\begin{proof}
  We define the functor $G(\varphi)$ on objects by: $G(\varphi) (I(a))
  = I(\varphi (a))$.  Given an arrow $I(a) \stackrel{\sigma}{\to}
  I(b)$ in $G(\Gamma)$ we define $G(\varphi)(\sigma)\colon
  I(\varphi(a)) \to I(\varphi(b))$ by
\begin{equation}\label{eq:G(varphi)sigma}
G(\varphi)\sigma := \varphi_b\circ \sigma  \circ \varphi_a \inv,
\end{equation}
so that the diagram
\begin{equation}\label{eq:3.1}
\xymatrix{ 
I(a) \ar[r]^{\sigma} \ar[d]^{\varphi_a} & I(b)\ar[d]^{\varphi_b}\\
I(\varphi(a))\ar[r]_{G(\varphi)\sigma} & I(\varphi(b))
}
\end{equation}
commutes.\footnote{ Note that our definition of $G(\varphi)\sigma$
  only makes sense if $\varphi_a$ is an isomorphism, i.e., only if
  $\varphi$ is \'etale.}
Moreover, the inverse of $G(\varphi)\colon\Hom (I(a), I(b)) \to \Hom
(I(\varphi (a)), I(\varphi(b)))$ is given by
\[
G(\varphi)\inv \tau = \varphi_b\inv \circ \tau\circ \varphi_a,
\]
so $G(\varphi)$ is fully faithful.
\end{proof}

\begin{remark}\label{rm:bowtie}
  If $\bowtie $ is a balanced equivalence relation (see
  \cite{Golubitsky.Stewart.Torok.05}) on a colored graph $\Gamma\to C$
  then the quotient map $\pi\colon \Gamma \to \Gamma/\bowtie$ is
  \'etale.  Moreover the induced map of groupoids $G(\pi)$ is not only
  faithful but also surjective on objects, hence an equivalence of
  categories by Theorem~\ref{thm:ODEequiv} below.
\end{remark}

\begin{proof}[Proof of Theorem~\ref{thmA}]
Consider an \'etale map of graphs over $C$: 
\[
\xy
(-6, 6)*+{\Gamma} ="1"; 
(6, 6)*+{\Gamma'} ="2";
(0,-4)*+{C }="3";
{\ar@{->}_{ \alpha} "1";"3"};
{\ar@{->}^{\varphi} "1";"2"};
{\ar@{->}^{\beta} "2";"3"};
\endxy ,
\]
that is, a morphism in the category $\et{(\FinGraph/C)}$.  Given an
arrow $I(a)\stackrel{\sigma}{\longrightarrow}I(b)$ in the groupoid
$G(\Gamma)$ we have a commuting diagram \eqref{eq:3.1} in the category
of finite trees over $C$.  Applying the control functor $\Ctrl$ gives
us a commuting diagram in the category $\Vect$ of vector spaces:
\begin{equation}\label{eq:3.2}
\xymatrix@C+2cm{ 
\Ctrl(I(a)) \ar[r]^{\Ctrl(\sigma)} \ar[d]^{\Ctrl(\varphi_a)} 
&\Ctrl( I(b))\ar[d]^{\Ctrl(\varphi_b)}\\
\Ctrl(I(\varphi(a))\ar[r]_{\Ctrl(G(\varphi)\sigma)} & \Ctrl(I(\varphi(b)).
}
\end{equation}
Now, for any arrow $I(a)\stackrel{\sigma}{\to} I(b)\in G(\Gamma)$ we
have the following diagram:
\[
\xy
(-40,0)*+{\V (\Gamma')}="1";
(-20,12)*+{\Ctrl_{\Gamma'}(\varphi(a))}="2";
(-20,-12)*+{\Ctrl_{\Gamma'}(\varphi(b))}="3";
(20,12)*+{\Ctrl_{\Gamma}(a)}="4";
(20,-12)*+{\Ctrl_{\Gamma}(b)}="5";
(40,0)*+{\V (\Gamma)}="6";
{\ar@{->}^{\varpi_{\varphi(a)}} "1";"2"}; 
{\ar@{->}^{\varpi_{\varphi(b)}} "1";"3"};
{\ar@{->}^{\Ctrl(\varphi_a)\inv} "2";"4"}; 
{\ar@{->}_{\Ctrl(\varphi_b)\inv} "3";"5"};
{\ar@{->}^{\varpi_a} "6";"4"}; 
{\ar@{->}^{\varpi_b} "6";"5"};
{\ar@{->}^{\Ctrl(G(\varphi)\sigma)} "2";"3"}; 
{\ar@{->}_{\Ctrl(\sigma)} "4";"5"}; 
\endxy,
\]
where $  \Ctrl_\Gamma(a)$ is an abbreviation for $\Ctrl(I(a))$, etc.
Consequently for any arrow $I(a)\stackrel{\sigma}{\to} I(b)\in G(\Gamma)$ the diagram
\begin{equation*}
\xy
(-20, -12)*+{\Ctrl_\Gamma(a)} ="1"; 
(20, -12)*+{\Ctrl_\Gamma(b)} ="2";
(0, 6)*+{\V(\Gamma') }="3";
{\ar@{->}_{\varpi_{\phi(a)}\circ \Ctrl(\phi_a)^{-1}} "3";"1"};
{\ar@{->}^{\Ctrl(\sigma)} "1";"2"};
{\ar@{->}^{\varpi_{\phi(b)}\circ \Ctrl(\phi_b)^{-1}} "3";"2"};
\endxy 
\end{equation*}
commutes.
By the universal property of the limit $\V(\Gamma)$ there exists
a unique linear map $\V(\varphi)\colon \V (\Gamma') \to \V (\Gamma)$
making the diagram
\begin{equation}\label{eq:3.3}
\xy
(-40,0)*+{\V (\Gamma')}="1";
(-20,12)*+{\Ctrl_{\Gamma'}(\varphi(a))}="2";
(-20,-12)*+{\Ctrl_{\Gamma'}(\varphi(b))}="3";
(20,12)*+{\Ctrl_{\Gamma}(a)}="4";
(20,-12)*+{\Ctrl_{\Gamma}(b)}="5";
(40,0)*+{\V (\Gamma)}="6";
{\ar@{->}^{\varpi_{\varphi(a)}} "1";"2"}; 
{\ar@{->}^{\varpi_{\varphi(b)}} "1";"3"};
{\ar@{->}^{\Ctrl(\varphi_a)\inv} "2";"4"}; 
{\ar@{->}_{\Ctrl(\varphi_b)\inv} "3";"5"};
{\ar@{->}^{\varpi_a} "6";"4"}; 
{\ar@{->}^{\varpi_b} "6";"5"};
{\ar@{..>}^{\V (\varphi)} "1";"6"};
\endxy.
\end{equation}
commute for all arrows $I(a)\stackrel{\sigma}{\to} I(b)\in G(\Gamma)$.
We leave it to the reader to  to check that
\[
\V (\psi \circ \varphi) = \V (\varphi) \circ \V (\psi)
\]
for any composable pair of \'etale maps of graphs
$\Gamma\stackrel{\varphi}{\to} \Gamma' \stackrel{\psi}{\to}\Gamma''$
over $C$, that is, that $\V$ is a contravariant functor (c.f.\ the
proof of Proposition~\ref{prop:5.1.11}).
\end{proof}

\begin{example} \label{ex:3} 
Consider the graphs $\Gamma$, $\Gamma'$ and the phase
  space function of Examples~\ref{ex:1} and \ref{ex:2}.
There is an evident \'etale map of graphs $
\varphi: \Gamma'\to \Gamma$,
\[
\Gamma' =\xy
(-20,5)*++[o][F]{1a}="1a"; 
(-20,-5)*++[o][F]{1b}="1b"; 
(0,0)*++[o][F]{2}="2"; 
{\ar@{->}^{ \alpha}  "1a";"2"};
{\ar@{->}_{ \beta}  "1b";"2"};
\endxy 
\quad\stackrel{\varphi}{\longrightarrow}\quad
 \Gamma = \xy
(-20,0)*++[o][F]{1}="1"; 
(0,0)*++[o][F]{2}="2"; 
{\ar@/_1.pc/_\beta "1";"2"};
{\ar@/^1.pc/^\alpha "1";"2"};
\endxy,
\]
defined on the nodes by 
\[
\varphi(1a) = \varphi(1b)  = 1,\quad \varphi(2) = 2.
\]
The induced map of groupoids $G(\varphi):G(\Gamma')\to G(\Gamma)$
sends $I(1a)$, $I(1b)$ to $I(1)$, $I(2)$ to $I(2)$, the isomorphism
$I(1a)\to I(1b)$ to the identity arrow on $I(1)$ and the non-trivial
arrow $\sigma:I(2)\to I(2)$ to $\sigma:I(2)\to I(2)$.  In particular
$G(\varphi)$ is surjective on objects.  Since $G(\varphi)$ is always
fully faithful (q.v.\ Lemma~\ref{lemma:ff}), it is, in this case, an
equivalence of categories.  Hence by Theorem~\ref{thm:ODEequiv},
\[
\V(\varphi): \V(\Gamma)\to \V(\Gamma')
\]
is an isomorphism of vector spaces.  This is not hard to check by
hand: under the identifications $\V(\Gamma) \simeq C^\infty (V\times
V, V)^{\ZZ_2}$ and $\V(\Gamma') \simeq C^\infty (V\times V,
V)^{\ZZ_2}$, the map $\V(\varphi)$ is the identity map:
\[
\V(\varphi) f = f
\]
for all $f\in C^\infty (V\times V, V)^{\ZZ_2}$.
\end{example}

\begin{remark} \label{rmrk:2.10.11}
Our construction of the functor $\V$ (Definition~\ref{def:vgivf} and
Theorem~\ref{thmA}) that assigns to each finite graph a space of
virtual groupoid-invariant vector fields is analogous to our
construction of the phase space functor $\PP$.  In both cases we have
a fixed category $\D$ with finite limits and a collection of
categories $\cU$; we then form a new category $\U$ with objects of the
form $U\stackrel{f}{\to}D$, where $U\in \cU$ and $f$ a functor, and
show that the assignment
\[
\U_0\to \D_0, \quad (U\stackrel{f}{\to}D)\mapsto \lim (U\stackrel{f}{\to}D)
\]
extends to a functor. 

In the case of $\PP$, we took $\D=\Euc$, the category of Euclidean
spaces and smooth maps, and $\cU$ as the collection of all finite sets
over a fixed set $C_0$.  In the case of $\V$, we took $\D=\Vect$, the
category of vector spaces and linear maps, and $\cU$ as the collection
of all groupoids of the form $G(\Gamma)$ with $\Gamma$ a finite graph
over the fixed graph of colors $C$.

The essential difference between the two constructions is that while
in the case of the phase space functor $\PP$ the morphisms of $\U$
are, roughly, commuting triangles  \[\xy
(-6, 6)*+{X} ="1"; 
(6, 6)*+{Y} ="2";
(0,-4)*+{\D }="3";
{\ar@{->}_{ \alpha} "1";"3"};
{\ar@{->}^{f} "1";"2"};
{\ar@{->}^{\beta} "2";"3"};
\endxy,
\]
 in the case of the construction of the functor $\V$ the morphisms are
 more complicated: the triangles of the form
\[
\xy
(-10, 10)*+{G(\Gamma)} ="1"; 
(10, 10)*+{G(\Gamma')} ="2";
(0,-4)*+{\Vect }="3";
{\ar@{->}_{ \Ctrl |_{G(\Gamma)}} "1";"3"};
{\ar@{->}^{G(\varphi)} "1";"2"};
{\ar@{->}^{\Ctrl|_{G(\Gamma')}} "2";"3"};
\endxy 
\]
no longer strictly commute.  Instead they ``2-commute''--- there is a
natural isomorphism
\[
\Ctrl |_{G(\Gamma)}\stackrel{\phi}{\Rightarrow}\Ctrl
|_{G(\Gamma')}\circ G(\varphi)
\]
--- this is the content of \eqref{eq:3.2}: the components of $\phi$
are linear isomorphisms $\Ctrl (\varphi_a)\inv$. A standard notation for such a
2-commutative diagram is:
\[
\xy
(-10, 10)*+{G(\Gamma)} ="1"; 
(10, 10)*+{G(\Gamma')} ="2";
(0,-6)*+{\Vect }="3";
{\ar@{->}_{ \Ctrl |_{G(\Gamma)}} "1";"3"};
{\ar@{->}^{G(\varphi)} "1";"2"};
{\ar@{->}^{\Ctrl|_{G(\Gamma')}} "2";"3"};
{\ar@{=>}_<<<{\scriptstyle \phi} (4,6)*{};(-2,2)*{}} ; 
\endxy .    
\]
A category-theoretic framework that encompasses both constructions is
discussed in section~\ref{sec:2slice}. Two interesting results that
follow from this framework are:
\end{remark}

\begin{theorem}\label{thm:ODEequiv}
  Suppose the map $G(\varphi)\colon G(\Gamma)\to G(\Gamma')$ induced by an
  \'etale map $\varphi\colon \Gamma\to \Gamma'$ of graphs over $C$ is
  essentially surjective.  Then the map
\[
\V (\varphi)\colon \V (\Gamma') \to \V (\Gamma)
\]
is an isomorphism.   

In particular, if $\bowtie$ is a balanced equivalence
relation~\cite{Golubitsky.Stewart.06} on $\Gamma$ and $\pi\colon
\Gamma\to \Gamma/{\bowtie}$ is the quotient map (q.v.\
Remark~\ref{rm:bowtie}), then
\[
\V (\pi)\colon \V (\Gamma/{\bowtie}) \to \V (\Gamma)
\]
is an isomorphism.\footnote{In the language of
  \cite{Golubitsky.Stewart.Torok.05} the graphs $\Gamma$ and
  $\Gamma/{\bowtie}$ are ``ODE equivalent.'' }
\end{theorem} 

\begin{empt}
For an input tree $I(a)$ of a graph $\Gamma$ over $C$ let 
\[
\Aut (I(a))= \Hom_{\Graph/C}(I(a), I(a)),
\]
the group of isomorphisms of a colored tree $I(a)$.  Then the skeleton
of a groupoid $G(\Gamma)$ (q.v. ~\ref{ex:groupoid skeleton}) is a
disjoint union of groups of the form
\[
\Aut (I(a_1))\sqcup\ldots \sqcup \Aut (I(a_k))
\]
for some nodes $a_1$, ..., $a_k$ of $\Gamma$. 
\end{empt}
\begin{theorem}\label{thm:inv-of-groups}
  For any colored graph $\Gamma\to C$ the space of virtual invariant
  vector fields $\V(\Gamma)$ is a product of invariants of groups.
In particular let 
\[
\Aut (I(a_1))\sqcup\ldots \sqcup \Aut (I(a_k))
\]
be a skeleton of the groupoid $G(\Gamma)$.  Then 
\[
\V(\Gamma) \simeq \prod_{i=1}^k \Ctrl(I(a_i))^{\Aut(I(a_i))}.
\]
\end{theorem}
\noindent
The proofs of \ref{thm:ODEequiv} and \ref{thm:inv-of-groups} are given
in Section~\ref{sec:2slice}.

\begin{remark}
  Note that while the spaces $\V(\Gamma)$, $\V(\Gamma')$ of virtual
  groupoid-invariant vector fields in the Example~\ref{ex:3} above are
  naturally isomorphic, there is no evident natural isomorphism
  between their images $S_\Gamma (\V(\Gamma))$ and $S_{\Gamma'}
  (\V(\Gamma'))$ in the spaces of vector fields $\chi (\PP \Gamma_0)$
  and $\chi(\PP\Gamma'_0)$ respectively.  Indeed, as we computed in
  Example~\ref{ex:2},
\[
S_\Gamma (\V(\Gamma)) =\{ (0, g) \in C^\infty (V^2, V^2)
\mid g(v_1, v_2) = f(v_1, v_1)
\textrm{ for some } f\in C^\infty (V\times V, V)^{\ZZ_2}\}
\]
while
\[
S_{\Gamma'} (\V(\Gamma')) =\{ (0, 0, g) \in C^\infty (V^3, V^3)
\mid g(v_{1a}, v_{1b}, v_2) = f(v_{1a}, v_{1b})
\textrm{ for some } f\in C^\infty (V\times V, V)^{\ZZ_2}\}.
\]
As we have seen in Example~\ref{ex:2}, $S_{\Gamma'}$ is injective. So its image is isomorphic to $C^\infty (V\times V, V)^{\ZZ_2}$:
\[
S_{\Gamma'} (\V(\Gamma')) \simeq C^\infty (V\times V, V)^{\ZZ_2}.
\]  
On the other hand the map
\[
C^\infty (V\times V, V)^{\ZZ_2} \to C^\infty (V,V), f\mapsto (g(v) = f(v,v))
\]
 is surjective.  Indeed for any $g\in C^\infty (V,V)$ 
\[
g(v) = f(v,v) 
\]
where $f(u,w) = \frac{1}{2} (g(u) + g(w))$.  Consequently 
\[
S_\Gamma (\V(\Gamma)) \simeq C^\infty (V,V),
\]
and there is no natural isomorphism between $C^\infty (V,V)$ and
$C^\infty (V\times V, V)^{\ZZ_2}$.  For this reason we are confused by
the definition of the pullback map $\beta^*$ on p.~332 of
\cite{Golubitsky.Stewart.06}.  Consequently we do not fully understand
Definition~6.1 (op.\ cit.) of $G$-admissible vector fields.

While there are no natural geometric or combinatorial maps between the
{\em images} $S_\Gamma (\V(\Gamma))$ and $S_{\Gamma'} (\V(\Gamma'))$
for any \'etale map of graphs $\varphi:\Gamma' \to \Gamma$ there is a
very interesting relationship between the maps $\V(\varphi)$, $\PP
\varphi$, $S_\Gamma$ and $S_{\Gamma'}$.  This relationship is the
subject of the next two subsections.
\end{remark}

\subsection{Virtual groupoid invariant vector fields and dynamics on Euclidean spaces}\label
{sec:sync}
\mbox{}\\[-8pt]

\noindent
Let $\varphi\colon \Gamma \to \Gamma'$ be an \'etale map of finite
graphs over our graph of colors $C$.  Let $w\in \V(\Gamma')$ be a
groupoid-invariant virtual vector field on the phase space
$\PP\Gamma'_0$.  Then $S_{\Gamma'} w$ is an actual vector field on
$\PP\Gamma'_0$.  Since the phase space functor is contravariant, we
have a smooth map of phase spaces $\PP \varphi\colon \PP\Gamma'_0\to
\PP \Gamma_0$.  Since the functor $\V$ is contravariant, we have a
virtual vector field $\V (\varphi)w \in \V(\Gamma)$, which gives us an
actual vector field $S_\Gamma (\V (\varphi)w)$ on $\PP\Gamma_0$.
Remarkably the vector fields $S_{\Gamma'} w $ and $S_\Gamma (\V
(\varphi)w)$ are $\PP\varphi$-related, which is the content of the
Theorem~\ref{thm:sync} below.

\begin{remark}
Note that if $\Gamma'$ is the quotient of the graph $\Gamma$ by a
balanced equivalence relation and $\varphi $ is the quotient map, then
$\PP\varphi$ identifies $\PP\Gamma'$ with a ``polydiagonal'' subspace
of $\PP\Gamma$.  The fact that the vector fields $S_{\Gamma'} w $ and
$S_\Gamma (\V (\varphi)w)$ are $\PP\varphi$-related translates into:
``the restriction of $S_\Gamma (\V (\varphi)w)$ to the polydiagonal is
$S_{\Gamma'} w $.''  Thus Theorem~\ref{thm:sync} generalizes the
synchrony results of \cite{Golubitsky.Pivato.Stewart.04} and
\cite{Golubitsky.Stewart.Torok.05} from quotient maps of colored
graphs to arbitrary \'etale maps of colored graphs.
\end{remark}
\begin{theorem}\label{thm:sync}
Let $w\in \V(\Gamma')$ be a virtual dynamical system on $\PP\Gamma_0$.
For each \'etale map of graphs
\[
\varphi\colon  \Gamma \to \Gamma',
\]
 the diagram
\begin{equation}\label{eq:4.1}
\xymatrix{
 \PP(\Gamma'_0) \ar[r]^{D \PP\varphi } 
&   \PP(\Gamma_0)   \\
\PP(\Gamma'_0) \ar[r]^{\PP \varphi}  \ar[u]^{S_{\Gamma'}(w)}
&  \PP(\Gamma_0)\ar[u]_{S_\Gamma(\V (\varphi)w)} }
\end{equation}
commutes: for any point $x\in \PP (\Gamma'_0)$
\begin{equation}\label{eq:1.36}
D \PP\varphi (x) \left( S_{\Gamma'}(w) \, (x) \right)=
 S_\Gamma(\V (\varphi)w)\,(\PP\varphi (x)).
\end{equation}
In other words, 
\[
\PP (\varphi)\colon  (\PP(\Gamma'), S_{\Gamma'}(w))\to  (\PP(\Gamma), S_\Gamma(\V (\varphi)w))
\]
is a map of dynamical systems.
\end{theorem}

\begin{remark}\label{rmrk:2.11.5}
  Since the map $\PP \varphi$ above is linear (q.v.
  \ref{rmrk:Pf-is-linear}), $D\PP\varphi (x) = \PP \varphi$ for all
  $x\in \PP (\Gamma'_0)$.  Therefore \eqref{eq:1.36} amounts to:
\begin{equation}\label{eq:1.36+}
 \PP\varphi  \left( S_{\Gamma'}(w) \, (x) \right)=
 S_\Gamma(\V (\varphi)w)\,(\PP\varphi (x)) 
\quad \mbox{ for any point }x\in \PP (\Gamma'_0).
\end{equation}
\end{remark}

\begin{proof}[Proof of Theorem~\ref{thm:sync}]
  Recall that for each node $a$ of a graph $\Gamma$ over the graph of
  colors $C$ we have a canonical map $\xi\colon I(a)\to \Gamma$ from
  the input tree $I(a)$ of the node $a\in \Gamma_0$ to the graph
  $\Gamma$ (q.v. \ref{empt:xi}; we use the same letter $\xi$ for all
  nodes of $\Gamma$ and for all graphs over $C$ suppressing both
  dependencies and trusting that lighter notation will cause no
  confusion).

  If $\varphi\colon \Gamma\to \Gamma'$ is an {\em \'etale} map of
  graphs over $C$, then by the definition of ``\'etale'', for each
  node $a$ of $\Gamma$, we have an induced {\em isomorphism} of graphs
  over $C$:
\[
  \varphi_a\colon  I(a) \to I(\varphi(a))
\]
defined by \eqref{eq:varphi_a}.
Note that $\varphi_a$ necessarily sends the root $a$ of $I(a)$ to the
root $\varphi(a)$ of $I(\varphi(a))$ and the leaves $\lv I(a)$
bijectively to the leaves $\lv I(\varphi(a))$ of $I(\varphi(a))$.
Moreover, the diagram of graphs over $C$
\begin{equation}\label{eq:(1)}
\xy
(-14, 10)*+{I(a)}="1"; 
(14,10)*+{\Gamma}="2";
(-14, -10)*+{I(\varphi(a))}="4"; 
(14,-10)*+{\Gamma'}="5";
{\ar@{->}^{\xi} "1";"2"};
{\ar@{->}^{\varphi_a} "1";"4"};
{\ar@{->}_{\varphi} "2";"5"};
{\ar@{->}_{\xi} "4";"5"};
\endxy
\end{equation}
commutes.  Hence we have a commuting diagram
\begin{equation}\label{eq:(1')}
\xy
(-14, 10)*+{\lv I(a)}="1"; 
(14,10)*+{\Gamma_0}="2";
(-14, -10)*+{\lv I(\varphi(a))}="4"; 
(14,-10)*+{\Gamma'_0}="5";
{\ar@{->}^{\xi} "1";"2"};
{\ar@{->}^{\varphi_a} "1";"4"};
{\ar@{->}_{\varphi} "2";"5"};
{\ar@{->}_{\xi} "4";"5"};
\endxy
\end{equation}
of finite sets over $C_0$ and consequently, since $\PP$ is a functor,
the commuting diagram
\begin{equation}\label{eq:(1'')}
\xy
(-14, 10)*+{\PP\lv I(a)}="1"; 
(14,10)*+{\PP\Gamma_0}="2";
(-14, -10)*+{\PP\lv I(\varphi(a))}="4"; 
(14,-10)*+{\PP\Gamma'_0}="5";
{\ar@{->}_{\PP \xi} "2";"1"};
{\ar@{->}^{\PP \varphi_a} "4";"1"};
{\ar@{->}_{\PP \varphi} "5";"2"};
{\ar@{->}_{\PP\xi} "5";"4"};
\endxy
\end{equation}
in our category $\Euc$ of Euclidean phase spaces.
Let 
\begin{equation}\label{eq:kappa}
\kappa_a\colon  \{a\} \hookrightarrow \Gamma_0
\end{equation}
denote the canonical inclusion in the category $\FinSet/C_0$ of finite sets over
$C_0$.  We then have a commuting diagram in $\Finset/C_0$
\begin{equation}\label{eq:(2)}
\xy
(-14, 10)*+{\{a\}}="1"; 
(14,10)*+{\Gamma_0}="2";
(-14, -10)*+{\{\varphi(a)\}}="4"; 
(14,-10)*+{\Gamma'_0,}="5";
{\ar@{^{(}->}^{\kappa_a} "1";"2"};
{\ar@{->}_{\varphi_a|_{\{a\}}= \varphi|_{\{a\}}} "1";"4"};
{\ar@{->}^{\varphi} "2";"5"};
{\ar@{^{(}->}_{\kappa_{\varphi(a)}} "4";"5"};
\endxy
\end{equation}
hence a commuting diagram in $\Euc$:
\begin{equation}\label{eq:(3)}
\xy
(-14, 10)*+{\PP \{a\}}="1"; 
(14,10)*+{\PP \Gamma_0}="2";
(-14, -10)*+{\PP \{\varphi(a)\}}="4"; 
(14,-10)*+{\PP \Gamma'_0,}="5";
{\ar@{->>}_{\PP \kappa_a} "2";"1"};
{\ar@{->}^{\PP (\varphi|_{ \{a\}})} "4";"1"};
{\ar@{->}_{\PP \varphi} "5";"2"};
{\ar@{->>}^{\PP \kappa_{\varphi(a)}} "5";"4"};
\endxy
\end{equation}
Differentiating \eqref{eq:(3)} we get the commuting diagram
\begin{equation}\label{eq:(4)}
\xy
(-14, 10)*+{T\PP \{a\}}="1"; 
(14,10)*+{T\PP \Gamma_0}="2";
(-14, -10)*+{T\PP \{\varphi(a)\}}="4"; 
(14,-10)*+{T\PP \Gamma'_0,}="5";
{\ar@{->>}_{D\PP \kappa_a} "2";"1"};
{\ar@{->}^{D\PP (\varphi|_{ \{a\}})} "4";"1"};
{\ar@{->}_{D\PP \varphi} "5";"2"};
{\ar@{->>}^{D\PP \kappa_{\varphi(a)}} "5";"4"};
\endxy
\end{equation}
of tangent bundles.
Recall that the projection 
\[
p_a\colon  \chi (\PP \Gamma_0) \to \Hom_\Euc (\PP \Gamma_0, \PP \{a\})
\]
from~(\ref{eq:defofpa}) is given by
\begin{equation}\label{eq:pakappaa}
p_a (X) = D (\PP \kappa_a) \circ X.
\end{equation}
By definition of the maps $S_\Gamma$ and $S_{\Gamma'}$ (q.v.
\eqref{eq:2.2}) we have
\begin{equation}\label{eq:4*}
D\PP \kappa _a \circ S_\Gamma (v) = \varpi_a (v)\circ \PP (\xi|_{\lv I(a)}) 
\quad
\textrm{ for any $a\in \Gamma_0$ and any }v\in \V(\Gamma) 
\end{equation}
and
\begin{equation}\label{eq:4*'}
  D\PP \kappa _{\varphi(a)} \circ S_{\Gamma'} (v') 
  = \varpi_{\varphi (a)}(v') \circ \PP \xi|_{\lv I(\varphi (a))} \quad
  \textrm{ for any $a\in \Gamma_0$ and any }v'\in \V(\Gamma'). 
\end{equation}
By definition \eqref{eq:3.3} of $\V\varphi$
\begin{equation}\label{eq4:alpha}
\varpi_a (\V(\varphi)w) = \Ctrl (\varphi_a)\inv (\varpi_{\varphi (a)} w)
\end{equation}  
for any virtual invariant vector field $w\in \V(\Gamma')$ and
any node $a\in \Gamma_0$. 
By definition of $\Ctrl(\varphi_a)$ (q.v. \eqref{eq:1.17}, \eqref{eq:1.18})
\[ 
\Ctrl (\varphi_a)\inv u = u\circ (\PP \varphi_a)\inv \quad \textrm{ for any }
u\in \Ctrl (I(\varphi(a))).
\]
Hence 
\[
\varpi_a (\V (\varphi)w) = (\varpi_{\varphi(a)}w)\circ (\PP \varphi_a)\inv.
\]
Consequently
\begin{equation}\label{eq4:2alpha}
\varpi_a(\V(\varphi)w) \circ \PP \varphi_a =  \varpi_{\varphi(a)}w.
\end{equation}
Since the vector space of vector fields $\chi (\PP\Gamma_0)$ is
the product
\[
\chi (\PP\Gamma_0)= \prod_{a\in \Gamma_0} \Hom_\Euc (\PP \Gamma_0, \PP \{a\}),
\]
the diagram \eqref{eq:4.1} commutes if and only if 
\[
p_a (S_\Gamma (\V(\varphi) w) \circ \PP \varphi ) = 
p_a (D\PP \varphi \circ S_{\Gamma'} (w))
\]
for every $a\in \Gamma_0$.  We now compute:
\begin{eqnarray*}
p_a (S_\Gamma (\V(\varphi) w) \circ \PP \varphi )
  &= &  D\PP\kappa_a \circ S_\Gamma (\V (\varphi) w) \circ \PP\varphi, 
  \quad \textrm{ by (\ref{eq:pakappaa})}\\
&= & 
\left(\varpi_a(\V(\varphi)w)\circ \PP (\xi|_{\lv I(a)})\right)
\circ \PP\varphi, 
  \quad \textrm{ by (\ref{eq:4*})} \\
  & = &\varpi_a(\V(\varphi)w)\circ 
  \left(\PP (\xi|_{\lv I(a)}) \circ \PP\varphi\right)\\
  &=& \varpi_a(\V(\varphi)w)\circ \PP \varphi_a \circ 
\PP (\xi|_{\lv I(\varphi(a))})
  \quad \textrm{ by \eqref{eq:(1'')} }\\
  & = &  \varpi_{\varphi(a)} w
  \circ \PP(\xi|_{\lv I(\varphi(a))}),\quad \textrm{ by (\ref{eq4:2alpha})}\\ 
  &=&D\PP\kappa_{\varphi(a)} 
  \circ S_{\Gamma'} (w)  \quad \textrm{ by (\ref{eq:4*'})}\\ 
 &=& D \PP (\varphi|_{\{a\}}) \circ D\PP\kappa_{\varphi(a)} 
  \circ S_{\Gamma'} (w) \quad \textrm{since } D \PP (\varphi|_{\{a\}}) =
D(id_{\PP \{a\}})  =id\\
&=&D\PP \kappa_a \circ D\PP \varphi \circ  S_{\Gamma'} (w) ,
  \quad \textrm{since \eqref{eq:(4)} commutes} \\ 
  &=& p_a (D\PP \varphi ( S_{\Gamma'} (w))).
\end{eqnarray*}
And we are done.
\end{proof}

\begin{example}\label{ex:1rev}
  Here we revisit Examples~\ref{ex:1}, \ref{ex:2} and \ref{ex:3}.
  Recall that we have an \'etale map of (mono-colored) graphs
  $\varphi:\Gamma'\to \Gamma$:
\[
\Gamma' =\xy
(-20,5)*++[o][F]{1a}="1a"; 
(-20,-5)*++[o][F]{1b}="1b"; 
(0,0)*++[o][F]{2}="2"; 
{\ar@{->}^{ \alpha}  "1a";"2"};
{\ar@{->}_{ \beta}  "1b";"2"};
\endxy 
\quad\stackrel{\varphi}{\longrightarrow}\quad
 \Gamma = \xy
(-20,0)*++[o][F]{1}="1"; 
(0,0)*++[o][F]{2}="2"; 
{\ar@/_1.pc/_\beta "1";"2"};
{\ar@/^1.pc/^\alpha "1";"2"};
\endxy,
\]
and the phase space function $\cP (\circ) = V$.  Then $\PP \Gamma_0 =
V\times V=: V^2$, $\PP (\Gamma'_0) = V\times V \times V =:V^3$ and $\PP \varphi: V^2 \to V^3$ is given by 
\[
(\PP \varphi)(v_1, v_2) = (v_1, v_1, v_2)
\]
(cf.\ Examples~\ref{ex:2.3.6} and \ref{ex:2.3.7}).  Recall that
$\V(\Gamma) \simeq \C^\infty(V\times V,V)^{\ZZ_2}\simeq \V(\Gamma')$ and that 
under these identifications 
\[
\V (\varphi):\V(\Gamma) \to \V(\Gamma')
\]
is the identity map.  Recall also that for any $f\in \C^\infty(V^2 ,V)^{\ZZ_2}$
\[
S_\Gamma (f) (v_1, v_2) = (0, f(v_1, v_1))
\]
while 
\[
S_{\Gamma'} (f) (v_{1a}, v_{1b}, v_2) = (0, f(v_{1a}, v_{1b})).
\]
Consequently 
\[
((\PP \varphi) \circ (S_\Gamma f)) \, (v_1, v_2) = (0, 0, f(v_1, v_1)) =
((S_{\Gamma'} f) \circ (\PP \varphi)) \, (v_1, v_2),
\]
which is precisely the assertion of Theorem~\ref{thm:sync}.  Thus the
``polydiagonal'' $\PP \varphi (V^2) \subset V^3$ is an invariant
submanifold for any ``groupoid-invariant'' vector field $v\in
S_{\Gamma'} (\V (\Gamma'))\subset \chi (V^3)$.
\end{example}

\subsection{The  combinatorial category  $\cV^\Euc$ of dynamical systems}
\label{sec:VEuc}\mbox{}\\[-8pt]

\noindent
Having constructed a contravariant functor $\V\colon
\opet{(\FinGraph/C)} \to \Vect$ that assigns to each $C$-colored graph
$\Gamma$ the vector space of virtual groupoid-invariant vector fields
on the phase space $\PP \Gamma_0$, we now construct a category
$\cV^\Euc$ whose set of objects is the collection of all
groupoid-invariant vector fields $\bigsqcup_{\Gamma \in \FinGraph/C}
\V(\Gamma)$.  This construction is, incidentally, an example of the
construction of the category of elements of a presheaf.

\begin{definition}[The  combinatorial category $\cV^\Euc$ of dynamical systems] 
  Fix a graph of colors $C=\{C_1\toto C_0\}$ and a phase space
  function $\cP\colon C_0 \to \Euc$.  Let $\PP\colon \FinSet/C_0 \to
  \Euc$ be the associated phase space functor \eqref{eq:ph-s-funct}
  and $\V\colon \opet{(\FinGraph/C)} \to \Vect$ the corresponding
  virtual vector fields functor.
\begin{enumerate}
\item The collection of objects $({\cV^\Euc})_0$ is the collection of pairs
  \[ 
({\cV^\Euc})_0 := \{ (v,\Gamma) \mid \Gamma \in (\FinGraph/C)_0,
  v\in \V(\Gamma)\}=
\bigsqcup_{\Gamma \in (\FinGraph/C)_0} \V (\Gamma) .
\]
\item For any two objects $(v,\Gamma), (v', \Gamma')\in {\cV^\Euc}_0$, the set
  of morphisms $\Hom_{\cV^\Euc} ((v,\Gamma), (v', \Gamma'))$ is given by
\[
\Hom_{\cV^\Euc} ((v,\Gamma), (v', \Gamma')) := \left\{(h,(v,\Gamma),
(v', \Gamma') )\in \Hom_{\et{(\FinGraph/C)}}(\Gamma,\Gamma')\times
{(\cV^\Euc)}_0 \times {(\cV^\Euc)}_0 \mid \V(h)v = v'\right\}.
\]
\end{enumerate}
Composition in $\cV^\Euc$ is defined by
\[
(g, (v', \Gamma'), (v'', \Gamma''))\circ (h, (v,\Gamma), (v',\Gamma'))=
(hg, (v,\Gamma), (v'', \Gamma''))
\]
for all $\Gamma\stackrel{h}{\leftarrow} \Gamma' \stackrel{g}{\leftarrow}\Gamma''$
and $v\in \V\Gamma, v'\in \V\Gamma', v''\in \Gamma''$ with $\V(g)v' =
v''$ and $\V(h)v = v'$.
\end{definition}

\begin{remark}
  There is an evident functor $\pi\colon \cV^\Euc \to
  \opet{(\FinGraph/C)}$.  On objects $\pi( v, \Gamma) = \Gamma$ and on
  morphisms
\[
\pi \left((v,\Gamma)\xrightarrow{(h,(v,\Gamma),
  (v',\Gamma'))} (v',\Gamma')\right) = \Gamma
\stackrel{h}{\leftarrow} \Gamma' \in
\Hom_{\opet{(\FinGraph/C)}}(\Gamma, \Gamma').
\]
\end{remark}

As we mention in the introduction, there is a category
$\mathsf{Dynamical Systems}$ whose objects are pairs $(M,X)$ where $M$
is a manifold and $X$ a vector field on $M$.  A morphism
$\varphi\colon (M,X) \to (M', X')$ in $\mathsf{Dynamical \,\, Systems}$ is a
smooth map of manifolds $\varphi\colon M\to M'$ with
\[
d\varphi\circ X = X'\circ \varphi,
\]
that is, a map of control systems (q.v.\ \ref{def:Maps of control
  systems}).  There is an evident functor 
\[
U\colon  \mathsf{Dynamical \,\, Systems} \to \Man
\]
 from the category of dynamical systems to the
category of manifolds that forgets the vector fields.  The category
$\mathsf{Dynamical  \,\,Systems}$ has a full subcategory $\mathsf{Dynamical \,\,
  Systems}_\Euc$ whose objects are pairs $(W, X)$ where $W$ is a
Euclidean space. By construction
\[
U (\mathsf{Dynamical \,\,  Systems}_\Euc) = \Euc.
\]
Theorem~\ref{thm:sync} allows us to reinterpret the collection of maps 
$\{ S_\Gamma\colon  \V\Gamma \to \chi (\PP \Gamma_0)\}_{\Gamma\in (\FinGraph/C)_0}$ as a 
functor
\[
S\colon  \cV^\Euc \to \mathsf{Dynamical \,\, Systems}_\Euc.
\]
Indeed, given an object $(v,\Gamma)\in \cV^\Euc$ we define
\[
S(v,\Gamma) = (\PP \Gamma_0, S_\Gamma (v)),
\]
where, as before, $\PP$ denotes the phase space functor and
$S_\Gamma\colon \V(\Gamma)\to \chi (\PP \Gamma_0)$ is defined by
\eqref{eq:2.2} (see also \ref{not:chi}).  Given a morphism
$(h,(v,\Gamma), (v', \Gamma'))\in \cV_\Euc$ we define
\[
S(h,(v,\Gamma), (v', \Gamma')) = (\PP \Gamma_0, S_\Gamma(v))
\stackrel{\PP h}{\to} (\PP\Gamma', S_{\Gamma'} (v)).
\]
By Theorem~\ref{thm:sync}
\[
D\PP h \left(S_\Gamma (v))\right) = S_{\Gamma'}(v') \circ \PP h.
\]
Therefore
\[
\PP h\colon (\PP \Gamma_0, S_\Gamma (v)) \to (\PP \Gamma'_0, S_{\Gamma'}(v'))
\]
is a morphism of dynamical systems.  
We have thus proved:
\begin{theorem}\label{thm:2.12.3}
  Let $\cV^\Euc$ denote the category of virtual groupoid-invariant
  dynamical systems constructed above, $\pi\colon \cV^\Euc \to
  \opet{(\FinGraph/C)}$ the canonical projection, $U\colon
  \mathsf{Dynamical \,\,Systems}_\Euc \to \Euc$ the forgetful functor,
  and $\PP$ the phase space functor extended to finite graphs (q.v.\
  \ref{rmrk:PonFinGraph}). There exists a functor $S\colon \cV^\Euc
  \to \mathsf{Dynamical \,\, Systems}_\Euc$ making the diagram
\begin{equation}\label{eq:star-Euc}
\xymatrix{
\cV^\Euc  \ar[r]^{S \quad\quad } \ar[d]_{\pi}
&  \mathsf{Dynamical \,\,Systems}_\Euc \ar[d]^{U}  \\
 \opet{(\FinGraph/C)}  \ar[r]^{\quad\PP }
&  \Euc }
\end{equation}
commute.  On objects
\[
S(v,\Gamma) = (\PP \Gamma_0, S_\Gamma (v));
\]
on arrows
\[
S(h,(v,\Gamma), (v', \Gamma')) = (\PP \Gamma_0, S_\Gamma(v))
\stackrel{\PP h}{\to} (\PP\Gamma', S_{\Gamma'} (v)).
\]
\end{theorem}

\subsection{An example}\label{sec:extex}
\mbox{}\\[-8pt]

\noindent
Let us consider the two graphs $\Gamma,\Gamma'$ over the graph of
colors $C=\,\xymatrix{ *+[o][F]{} \ar@(dr,ur)}$\quad\quad depicted
in~\eqref{eq:gammapic} below, along with the maps $r\colon
\Gamma\to\Gamma', \imath \colon \Gamma'\to\Gamma$ given on nodes by
\begin{equation*}
  r(1,2,3,4,5,6) = (a,b,c,a,b,c),\quad \imath (a,b,c) = (1,2,3).
\end{equation*}
Clearly both $r$ and $\imath$ are \'etale, $r$ is surjective and $\imath$ is
injective, and $r\circ \imath=id_{\Gamma'}$.  
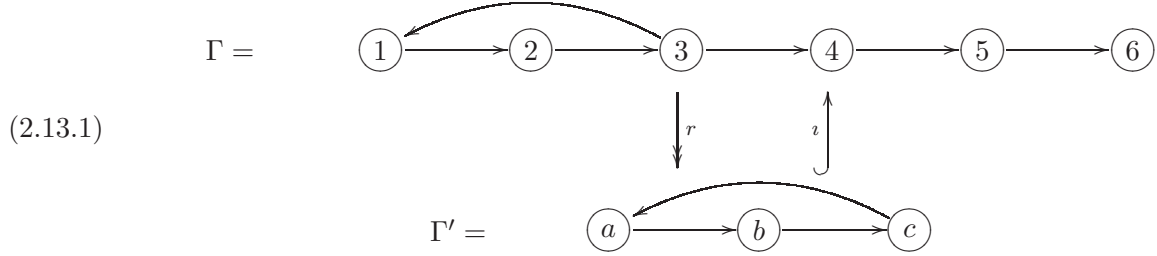
\begin{equation}\label{eq:gammapic}
\xy
(-10,12)*+{\xy
(-40,0)*+{\Gamma =}="n1";
(-20,0)*++[o][F]{1}="1"; 
(0,0)*++[o][F]{2}="2"; 
(20,0)*++[o][F]{3}="3"; 
(40,0)*++[o][F]{4}="4"; 
(60,0)*++[o][F]{5}="5"; 
(80,0)*++[o][F]{6}="6"; 
{\ar@{->}^{} "1";"2"};
{\ar@{->}_{} "2";"3"};{\ar@{->}_{} "3";"4"};
{\ar@{->}_{} "4";"5"};
{\ar@{->}_{} "5";"6"};
{\ar@/_1.5pc/ "3";"1"};
\endxy }="a";
(-10,-12)*+{ 
\xy
(-40,0)*+{\Gamma' =}="n2";
(-20,0)*++[o][F]{a}="1"; 
(0,0)*++[o][F]{b}="2"; 
(20,0)*++[o][F]{c}="3"; 
{\ar@{->}^{} "1";"2"};
{\ar@{->}_{} "2";"3"};
{\ar@/_1.5pc/ "3";"1"};
\endxy
}="b";
(-10,5)*+{}="c";
(-10,-5)*+{}="d";
(10,5)*+{}="e";
(10,-5)*+{}="f";
{\ar@{->>}^{r} "c";"d"};
{\ar@{^{(}->}^{\imath} "f";"e"};
\endxy
\end{equation}
Let Euclidean space $V$ be the value of the space function $\cP$ on the only node of the graph $C$ of colors: 
\[
V = \cP (\circ).
\]
Then the phase space associate to any node of $\Gamma$ and $\Gamma'$ is $V$.

We first consider the graph $\Gamma$.  Define the function
$\ell\colon[6]\to[6]$ by $\ell(1,2,3,4,5,6) = (3,1,2,3,4,5)$, i.e.
$\ell(i)$ is the number of the node which precedes the node $i$ in the
graph $\Gamma$.
We label the unique arrow of $\Gamma$ whose target is the node $i$ by
$\gamma_{\ell(i),i}$.  Consequently  each input tree of $\Gamma$ is of the form
\begin{equation*}
  I(i) = \quad\quad\quad  \xy
  (20,0)*+{\{\gamma_{\ell(i),i}\}}="n1";
  (0,0)*++[o][F]{i}="1";
  {\ar@{->}_{\gamma_{\ell(i),i}} "n1";"1"};
\endxy.
\end{equation*}
Thus $\Ctrl(I(i)) = C^\infty(V,V)$.  We next consider the groupoid
$G(\Gamma)$.  For any nodes $i,j$ of $\Gamma$, the map
$\sigma_{ij}\colon I(i)\to I(j)$, defined by $\sigma_{ij}\colon
(i\mapsto j, \gamma_{\ell(i),i}\mapsto \gamma_{\ell(j),j})$ is an
isomorphism of input trees.  Since $\sigma_{ij}|_{\lv I(i)}$ is the
unique map sending the one point set $\{\ell(i)\}$ to the one point
set $\{\ell(j)\}$, $\PP (\sigma_{ij}|_{\lv I(i)})\colon \PP \lv I(j)
\to \PP \lv I(i)$ is the identity map for all nodes $i, j$.
Consequently
\[
\Ctrl (\sigma_{ij})\colon\Ctrl(I(i))\to \Ctrl(I(j)),
\]
which is defined by 
\[
\Ctrl (\sigma_{ij}) (f) = id\circ f \circ \PP \sigma_{ij}
\]
for any $f\in \Ctrl(I(i))= C^\infty (\PP \lv I(i), \PP \{i\})=
C^\infty (V,V)$, is the identity map.  Informally speaking the $j$th
component of the groupoid invariant vector field $f$ on $\PP \Gamma_0$
depends on $\PP \{\ell (i)\} = \PP \lv I(i)$ in the same way that the
$i$th component of $f$ depends on $\PP \{ \ell (i)\} = \PP \lv I(i)$.  More formally by Theorem~\ref{thm:inv-of-groups}
\[
\V (\Gamma) \simeq C^\infty (V, V).
\]
Moreover it is not hard to see that $S_\Gamma\colon \V (\Gamma) \to \chi
(\PP \Gamma_0) = \chi (V^6)$ is given by
\[
S_\Gamma (f) (v_1, \ldots, v_6) = (f(v_3),f(v_1),f(v_2),f(v_3),f(v_4),f(v_5))
\]
for all $f\in C^\infty (V, V)$.
Similarly, $\PP (\Gamma'_0) = V^3$, 

\[
\V (\Gamma') \simeq C^\infty (V, V).
\]
and $S_{\Gamma'}\colon \V (\Gamma') \to \chi (\PP \Gamma'_0) = \chi (V^3)$ is
given by
\begin{equation}\label{eq:2.13.+}
  S_{\Gamma'}(g) (v_a, v_b, v_c) = (g(x_c),g(x_a),g(x_b))
\end{equation}
Using~\eqref{eq:PPf} we see that 
\[
\PP r\colon V^3=\PP\Gamma'_0 \to \PP \Gamma= \V^6
\]
is an embedding given by 
\[
\PP r(v_a,v_b,v_c)=(v_a,v_b,v_c,v_a,v_b,v_c),
\]
and $\PP \imath\colon V^6\to V^3$ is the projection given by
\[
\PP \imath(v_1,v_2,v_3,v_4,v_5,v_6) = (v_1,v_2,v_3)
\]
Note also that 
\[
\PP (\imath\circ r) = \PP r \circ \PP \imath\colon V^6\to V^6
\]
is the submersion onto the diagonal $\Delta_{V^3}\subset V^3\times
V^3\simeq V^6$ given by
\[
\PP (\imath\circ r)\, (v_1, v_2, v_3, v_4, v_5, v_6) = (v_1,v_2,v_3,
v_1,v_2,v_3).
\]

We now turn to the maps $\V(r)$ and $\V(\imath)$ between the spaces of
virtual groupoid invariant vector fields.  For any node $i$ of $\Gamma$ the map
\[
r_i\colon I(i)\to I(r(i))
\]
is an isomorphism inducing the identity map
\[
\Ctrl (r_i)\colon C^\infty(V,V) =\Ctrl (I(i))\to \Ctrl(I(r(i)) = C^\infty (V,V).
\]
It follows from the construction of the functor $\V$ on arrows (q.v.\
\eqref{eq:3.3}, \eqref{eq4:alpha}) that
\[
\V(r)\col \V(\Gamma') \to \V(\Gamma)
\]
is the identity map on $C^\infty(V,V)$.  Similarly $\V(\imath)$ is the identity map as well.

By Theorem~\ref{thm:sync} for any virtual groupoid invariant vector field $f\in \V(\Gamma))$ we have
\begin{equation}\label{eq:2.13.2}
  D \PP(\imath \circ r) \,S_\Gamma (f) = S_\Gamma (\V(\imath \circ r)f)\circ \PP (\imath \circ r)
\end{equation}
Recall that for any map $\varphi$ of graphs over $C$ the map of
Euclidean spaces $\PP (\varphi)$ is linear and consequently $D\PP
\varphi = \PP \varphi$ (q.v.\ \ref{rmrk:2.11.5}).  Therefore
\eqref{eq:2.13.2} amounts to
\[
\PP(\imath \circ r) \,S_\Gamma (f) = S_\Gamma (\V(\imath \circ r)f)\circ
\PP (\imath \circ r),
\]
c.f.\ \eqref{eq:1.36+}, or, in this case,
\[
\PP(\imath \circ r) \circ S_\Gamma (f) = S_\Gamma (f)\circ
\PP (\imath \circ r).
\]
  Therefore the diagonal $\Delta_{V^3} \subset
\PP \Gamma_0$ is an invariant submanifold for any groupoid invariant
vector field $F\in S_\Gamma (\V(\Gamma))$. Moreover, $\PP (\imath
\circ r)$ projects the dynamics of $F$ on $\PP\Gamma_0$ onto the
dynamics of $F|_{\Delta_{V^3}}$.  Thus for an equilibrium $x\in
\Delta_{V^3}$ of $F|_{\Delta_{V^3}}$, the set $(\PP(\imath \circ
r))\inv (x)$ is an invariant submanifold of $F$.  To see this, if we have $S_\Gamma(f) x = 0$ and $\PP(\imath\circ r) y = x$, then 
\begin{equation*}
  \PP(\imath\circ r) S_\Gamma(f) y = S_\Gamma(f)  \PP(\imath\circ r) y = S_\Gamma(f) x = 0,
  \end{equation*}
so that $S_\Gamma(f)$ must be tangent to the diagonal $\Delta_{V^3}$.
Similarly, for a periodic orbit $O$ of $F|_{\Delta_{V^3}}$, the
set $(\PP(\imath \circ r))\inv (O)$ is a ``relative periodic orbit''
of $F$, and so on.

By Theorem~\ref{thm:sync} again, a dynamical system of the form
$(\Delta_{V^3}, F|_{\Delta_{V^3}}$, $F\in S_\Gamma (\V(\Gamma))$ is an
isomorphic image (under $\PP (r)$) of a dynamical system of the form
$(\PP (\Gamma'_0), G)$, $G\in S_{\Gamma'}(\V(\Gamma'))$.  By
Theorem~\ref{thm:5.1.5} and Example~\ref{ex:5.1.6} any groupoid
invariant vector field $G\in S_{\Gamma'}(\V(\Gamma'))$ is also
invariant under the action of the cyclic group $\ZZ_3$.  This can also
be seen directly from \eqref{eq:2.13.+}.  Thus we can use the
machinery of equivariant dynamical systems to study the dynamics of
groupoid invariant vector fields on $\Delta_{\V^3}$ and thereby on
$\PP\Gamma_0$.

Finally note that the ``color'' maps $\mathbf{c}\col\Gamma \to C =
\,\xymatrix{ *+[o][F]{} \ar@(dr,ur)}$\quad\quad and
$\mathbf{c}'\col\Gamma'\to C$ are both \'etale.  Then by
Theorem~\ref{thm:sync} the maps $\PP \mathbf{c}\col V = \PP C_0 \to
\PP \Gamma_0$ and $\PP \mathbf{c}'\col V \to \PP \Gamma'_0$ define
invariant subsystems for any groupoid invariant vector fields $F\in
S_\Gamma (\V(\Gamma))$ and $G\in S_{\Gamma'} (\V(\Gamma'))$,
respectively.

\section{Linear groupoid invariant vector fields} \label{sec:FinVect}

\subsection{}
\mbox{}\\[-8pt]

\noindent
The functor $\V\colon \opet{(\FinGraph/C)} \to \Vect$ constructed in
the previous section associates to each finite colored graph $\Gamma\to C$ the
space of $C^\infty$ (infinitely differentiable) groupoid invariant
vector fields $\V(\Gamma)$ and subsequently $S_\Gamma (\V(\Gamma)) $
is a subspace of $C^\infty$ vector fields on the Euclidean phase space
$\PP\Gamma_0$.  It is not hard to modify our construction to associate
a different class of vector fields, such as linear or analytic.  All
one has to do is to replace the category $\Euc$ by an appropriate
subcategory.  Here is a description of what such a replacement
entails in the case of linear vector fields.

We replace the category $\Euc$ of Euclidean spaces and $C^\infty$ maps
with the category $\FinVect$ of finite dimensional vector spaces over
the reals and linear maps.  As before, we fix a graph $C=\{C_1\toto
C_0\}$ of colors and a phase space function $\cP\colon C_0\to
\FinVect$.  These choices give rise to a phase space functor
\begin{equation}
\PP\colon  \op{(\FinSet/C_0)}\to \FinVect.
\end{equation}
On objects the functor $\PP$ is defined as a categorical product (q.v.\
\ref{def:product}):
\[
\PP X \equiv \PP (X \stackrel{\alpha}{\to} C_0): = 
\prod_{x\in X } \cP(\alpha (x)).
\]
Note that while we use exactly the same notation as before, the product in
question is now a product in $\FinVect$ and not in $\Euc$.  Given a
morphism $ \xy (-6, 6)*+{Y} ="1"; (6, 6)*+{X} ="2"; (0,-4)*+{C_0
}="3"; {\ar@{->}^{ \alpha} "2";"3"}; {\ar@{->}^{f} "1";"2"};
{\ar@{->}_{\beta} "1";"3"}; \endxy $ in the slice category
$\FinSet/C_0$ we obtain
\[
\PP f\colon \PP Y\to \PP X 
\]
using the universal property of products (q.v.\ \eqref{eq:2.4.3}).  We
have already constructed the groupoid of finite colored trees
$\FinTree/C$, so we recycle the construction.  The definition of the
control functor $\Ctrl_{\FV}\colon \FinTree/C \to \Vect$ looks almost
the same as $\Ctrl\colon \FinTree/C \to \Euc$
(q.v.~\ref{sec:controlfunctor}): for a colored tree $(T\to C)\in
\FinTree/C$ we set
\[
\Ctrl_{\FV}(T):= \Hom_\FV (\PP \lv T, \PP \rt T).
\]
Note that now $\Hom _\FV (\PP \lv T, \PP \rt T)$ is the finite
dimensional vector space of linear maps from the vector space $\PP \lv
T$ to the vector space $\PP \rt T$ (in fact, $\dim\Ctrl_\FV(T) =
\dim(\PP \lv T)\dim(\PP\rt T)$).

Given an isomorphism $\xy (-6, 6)*+{T} ="1"; (6, 6)*+{T'} ="2";
(0,-4)*+{C }="3"; {\ar@{->}_{ \alpha} "1";"3"}; {\ar@{->}^{\sigma}
  "1";"2"}; {\ar@{->}^{\beta} "2";"3"}; \endxy $ of trees over $C$, we
define
 \[
\Ctrl_{\FV}(\sigma)\colon  \Ctrl_\FV(T) \to \Ctrl_\FV (T')
\]
by the same formula as in \eqref{eq:1.18}:
\[
\Ctrl_{\FV} (\sigma) X := X \circ \PP (\sigma|_{\lv T} )\colon \PP \lv
T' \to \PP \rt T = \PP \rt T'
\]
for any $(X\colon \PP \lv T \to \PP \rt T) \in \Ctrl(T) = \Hom_\FV
(\PP \lv T, \PP \rt T)$.  Note that the image of $\Ctrl_\FV$ lands in
the subcategory $\FinVect$ of $\Vect$.

We do not need to change the definitions of input trees
(Definition~\ref{def:input-tree}) or of the groupoids $G(\Gamma)$
associated to colored graphs $\Gamma\to C$
(Definition~\ref{def:groupoidGGamma}).  We thus define the
(finite-dimensional) vector space $\V_\FV(\Gamma)$ of virtual
groupoid-invariant {\em linear} vector fields by
\[
\V_\FV(\Gamma): = \lim (\Ctrl_\FV|_{G(\Gamma)}\colon  G(\Gamma)
\to \FinVect).
\]

\subsection{The relation of $\V_\FV(\Gamma)$ to vector fields on the phase
  space $\PP\Gamma_0$}\label{subsec:S-GammaFin}
\mbox{}\\[-8pt]

\noindent
Given a finite dimensional vector space $W$ the space of linear vector
fields on $W$ is (naturally isomorphic to) the space of linear maps
$\Hom_\FV(W,W)$.  Given a finite graph $\Gamma$ over a graph of colors
$C$ and the phase space function $\cP\colon C_0 \to \FinVect$, the functor $\PP$ assigns to the set $\Gamma_0$ of vertices of $\Gamma$ a
phase space $\PP \Gamma_0$, which is now a finite dimensional vector
space.  In complete analogy with Section~\ref{sec:Sgamma}, there exists
a canonical linear map $S_\Gamma$ from the space of virtual
groupoid-invariant vector fields $\V_\FV(\Gamma)$ to the vector
space $\Hom_\FV(\PP\Gamma_0,\PP\Gamma_0 )$ of { linear }vector fields
on $\PP \Gamma_0$.  It is defined as follows.

Since $\Hom_\FV (\PP\Gamma_0, \PP\Gamma_0) = \Hom_\FV (\PP\Gamma_0,
\prod_{a\in \Gamma_0} \PP \{a\})$
(q.v.\ Remark~\ref{rmrk:hom-into-product}), the space of linear vector
fields on $\PP \Gamma_0$ is canonically the product
\[
\Hom_\FV (\PP\Gamma_0, \PP\Gamma_0)= \prod_{a\in \Gamma_0} \Hom_\FV
(\PP\Gamma_0, \PP \{a\}).
\]
To define the map into a product, it is enough to define a map into
its factors.  Thus we need maps $\V_\FV(\Gamma) \to \Hom_\FV
(\PP\Gamma_0, \PP \{a\})$ for each vertex $a$ of the graph $\Gamma$.

For each input tree $I(a)$ we have a map of graphs $\xi\colon
I(a)\to\Gamma$ (q.v. Remark~\ref{remark:xi}).  Therefore we have a map
$\xi|_{\lv I(a)}\colon\lv I(a)\to \Gamma_0$ of finite sets over $C_0$.  Applying the
phase space functor $\PP$ we obtain the maps
\[
\PP(\xi|_{\lv I(a)})\colon  \PP\Gamma_0 \to \PP \lv I(a).
\]
The pullback by $\PP(\xi|_{\lv I(a)})$ gives
\[
(\PP(\xi|_{\lv I(a)}))^*\colon \Ctrl(I(a))\to \Hom_\FV (\PP\Gamma_0, \PP \{a\}).
\]
By the universal properties of the product $\chi (\PP\Gamma_0)$ there
a canonical unique map $S_\Gamma\colon \V(\Gamma) \to \Hom_\FinVect
(\PP\Gamma_0, \PP\Gamma_0)$ making the diagram
\begin{equation} \label{eq:3.0.2}
\xy
(-22, 9)*+{\V (\Gamma)}="1"; 
(24,9)*+{\Hom_\FV (\PP\Gamma_0, \PP\Gamma_0)}="2";
(-22, -10)*+{\Ctrl (I(a))}="3"; 
(24,-10)*+{\Hom_\FV(\PP\Gamma_0, \PP \{a\})}="4";
{\ar@{-->}^{\!\!\!\!\!\!\!\!\exists ! S_\Gamma} "1";"2"};
{\ar@{->}_{\varpi_a} "1";"3"};
{\ar@{->}^{p_a} "2";"4"};
{\ar@{->}_{(\PP(\xi|_{\lv I(a)}) )^*\quad} "3";"4"};
\endxy  
\end{equation} 
commute.  Here, as before, 
\[
\varpi_a\colon  \V _\FV\Gamma\to \Ctrl_\FV (\PP I(a))
\]
 and 
\[
p_a\colon \Hom_\FV(\PP\Gamma_0,\PP\Gamma_0) = \prod_{a'\in \Gamma_0}
\Hom_\FV (\PP\Gamma_0, \PP \{a'\}) \to \Hom_\FV (\PP\Gamma_0, \PP
\{a\})
\]
denote the canonical projections.

\subsection{The assignment $\Gamma \mapsto \V_\FV\Gamma$ extends to a functor}
\mbox{}\\[-8pt]

\noindent
Just as for $C^\infty $ vector fields on Euclidean spaces we have
\begin{theorem}\label{thmA'}
  The map $\V_\FV$ that assigns to each finite graph $\Gamma$ over $C$ the vector
  space of $G(\Gamma)$-invariant virtual linear vector fields on $\PP\Gamma_0$
  extends to a contravariant functor
\[
\V_\FinVect\colon  \opet{(\FinGraph/C)} \to \Vect.
\]
\end{theorem}

\begin{proof}
Consider an \'etale map of graphs over $C$: 
\[
\xy
(-6, 6)*+{\Gamma} ="1"; 
(6, 6)*+{\Gamma'} ="2";
(0,-4)*+{C }="3";
{\ar@{->}_{ \alpha} "1";"3"};
{\ar@{->}^{\varphi} "1";"2"};
{\ar@{->}^{\beta} "2";"3"};
\endxy ,
\] 
that is, a morphism in the category $\et{(\FinGraph/C)}$.  Given an
arrow $I(a)\stackrel{\sigma}{\longrightarrow}I(b)$ in the groupoid
$G(\Gamma)$ we have a commuting diagram \eqref{eq:3.1}
\[
\xymatrix{ 
I(a) \ar[r]^{\sigma} \ar[d]^{\varphi_a} & I(b)\ar[d]^{\varphi_b}\\
I(\varphi(a))\ar[r]_{G(\varphi)\sigma} & I(\varphi(b))
}
\] 
in the category of finite trees over $C$.  Applying the control
functor $\Ctrl_\FV$ gives us a commuting diagram in the category
of vector spaces (q.v.\ \eqref{eq:3.2}):
\begin{equation}\label{eq:3.2'}
\xy
(-25, 9)*+{\Ctrl_{\FV}(I(a))}="1"; 
(25, 9)*+{\Ctrl_{\FV}(I(b))}="2"; 
(-25, -9)*+{\Ctrl_{\FV}(I(\phi(a)))}="3"; 
(25, -9)*+{\Ctrl_{\FV}(I(\phi(b)))}="4"; 
{\ar@{->}^{\Ctrl_{\FV}(\sigma)} "1";"2"};
{\ar@{->}_{\Ctrl_{\FV}(\varphi_a)} "1";"3"};
{\ar@{->}^{\Ctrl_{\FV}(\varphi_b)} "2";"4"};
{\ar@{->}_{\Ctrl_{\FV}(G(\varphi)\sigma)} "3";"4"};
\endxy  
\end{equation}
We now could mimic the rest of the proof of \ref{thmA}.  However, we
choose to proceed in a more functorial fashion.  Diagram
\eqref{eq:3.2'} gives us the 2-commutative diagram
\[
\xy
(-10, 10)*+{G(\Gamma)} ="1"; 
(10, 10)*+{G(\Gamma')} ="2";
(0,-6)*+{\Vect }="3";
{\ar@{->}_{ \Ctrl_\FV |_{G(\Gamma)}} "1";"3"};
{\ar@{->}^{G(\varphi)} "1";"2"};
{\ar@{->}^{\Ctrl_\FV|_{G(\Gamma')}} "2";"3"};
{\ar@{=>}_<<<{\scriptstyle \phi} (4,6)*{};(-2,2)*{}} ; 
\endxy ,   
\]
q.v.\ Remark~\ref{rmrk:2.10.11}.  This 2-commuting triangle is a
morphism in $\Cat/\Vect$ (see \ref{empt:5.1.1}).  It is not hard to
check that the assignment
\[
\mathcal{G}\colon  \et{(\FinGraph/C)} \to \Cat/\Vect, \quad
\xy
(-6, 6)*+{\Gamma} ="1"; 
(6, 6)*+{\Gamma'} ="2";
(0,-4)*+{C }="3"; 
{\ar@{->}_{ \alpha} "1";"3"};
{\ar@{->}^{\varphi} "1";"2"};
{\ar@{->}^{\beta} "2";"3"};
\endxy  
\mapsto
\xy
(-10, 10)*+{G(\Gamma)} ="1"; 
(10, 10)*+{G(\Gamma')} ="2";
(0,-6)*+{\Vect }="3";
{\ar@{->}_{ \Ctrl_\FV |_{G(\Gamma)}} "1";"3"};
{\ar@{->}^{G(\varphi)} "1";"2"};
{\ar@{->}^{\Ctrl_\FV|_{G(\Gamma')}} "2";"3"};
{\ar@{=>}_<<<{\scriptstyle \phi} (4,6)*{};(-2,2)*{}} ;  
\endxy         
\]  
preserves the composition of morphisms and thus is a functor.
Composing the functor $\mathcal{G}$ with the contravariant functor
$L\colon \op{\Cat/\Vect} \to \Vect$ (see \ref{prop:5.1.11}) gives us the desired
contravariant functor
\[
\V_\FV = \mathcal{G}\circ L\colon  \opet{(\FinGraph/C)} \to \Vect.
\]
\end{proof}

\subsection{The combinatorial category $\cV^\FV$ of linear dynamical
  systems}\label{subsec:cVFinVect}
\mbox{}\\[-8pt]

\noindent
Associated to the functor $\V_\FV$ we have the category of elements
$\cV^\FV$ and a functor $\cV^\FV \to \opet{(\FinGraph/C)}$.  As in the
case of $\cV^\Euc$ the objects of $\cV^\FV$ are pairs of the form
$(\Gamma, v)$ with $\Gamma$ a finite graph over $C$ and $v\in
\V_\FV(\Gamma)$.  The morphisms in $\cV^\FV$ are the tuples of the
form $(h,(v,\Gamma), (v', \Gamma') )$ with $h\colon \Gamma'\to \Gamma$
a morphism in $\FinGraph/C$ and $\V_\FinGraph (h)v = v'$.

As in the case of $C^\infty$ vector fields on Euclidean spaces the
maps 
\[
\{S_\Gamma\colon \V_\FV (\Gamma) \to \Hom_\FV (\PP \Gamma_0, \PP
\Gamma_0)\}_{\Gamma\in (\FinGraph/C)_0}
\]
 assemble into a functor $S$ from the category of elements $\cV^\FV$
 to the category $\mathsf{LinDyn}$ of {\em linear} dynamical
 systems.  The key technical result in proving that $S$ is a functor
 is the analogue of Theorem~\ref{thm:sync}:
\begin{theorem}\label{thm:sync-lin}
  Let $w\in \V_\FV(\Gamma')$ be a virtual linear dynamical system on
  $\PP(\Gamma)$.  For each \'etale map of graphs
\[
\varphi\colon  \Gamma \to \Gamma'
\]
 the diagram
\begin{equation}\label{eq:4.1-lin}
\xymatrix{
 \PP(\Gamma'_0) \ar[r]^{\PP\varphi } 
&   \PP(\Gamma_0)   \\
\PP(\Gamma'_0) \ar[r]^{\PP \varphi}  \ar[u]^{S_{\Gamma'}(w)}
&  \PP(\Gamma_0)\ar[u]_{S_\Gamma(\V (\varphi)w)} }
\end{equation}
commutes.
In other words, 
\[
\PP (\varphi)\colon  (\PP(\Gamma'), S_{\Gamma'}(w))\to  
(\PP(\Gamma), S_\Gamma(\V (\varphi)w))
\]
is a map of linear dynamical systems.
\end{theorem}
\noindent
Since the proof of \ref{thm:sync-lin} is completely analogous to the
proof of \ref{thm:sync} we omit it.  We conclude the section by
stating the analogue of Theorem~\ref{thm:2.12.3}:

\begin{theorem}\label{thm:2.12.3-lin}
  Let $\cV^\FV$ denote the category of virtual groupoid-invariant
  dynamical systems constructed above, $\pi:\cV^\FV \to
  \opet{(\FinGraph/C)}$ denote the canonical projection, $U\colon
  \mathsf{LinDyn} \to \FinVect$ the forgetful functor, and $\PP$
  the phase space functor extended to finite graphs
  (q.v.\ \ref{rmrk:PonFinGraph}). There exists a functor $S\colon
  \cV^\FV \to \mathsf{LinDyn}$ making the diagram
\begin{equation}\label{eq:star-lin}
\xymatrix{
\cV^\FV  \ar[r]^{S \quad\quad } \ar[d]_{\pi}
&  \mathsf{LinDyn} \ar[d]^{U}  \\
 \opet{(\FinGraph/C)}  \ar[r]^{\quad\PP }
&  \FinVect }
\end{equation}
commute.  On objects
\[
S(v,\Gamma) = (\PP \Gamma_0, S_\Gamma (v));
\]
on arrows
\[
S(h,(v,\Gamma), (v', \Gamma')) = (\PP \Gamma_0, S_\Gamma(v))
\stackrel{\PP h}{\to} (\PP\Gamma', S_{\Gamma'} (v)).
\]
\end{theorem}

\begin{remark}
  If $\varphi\colon\Gamma\to\Gamma'$ is an isomorphism of graphs, then
  $\PP\varphi$ and $\V\varphi$ are isomorphisms of vector spaces,
  since $\PP$ and $\V$ are functors.  Therefore~\eqref{eq:4.1-lin} is
  a similarity of linear transformations.  It is not hard to see,
  however, that the only similarities we can obtain in this matter are
  permutations of blocks of identities (see~\eqref{eq:PPf}).  In any
  case, $S_\Gamma(w)$ and $S_{\Gamma'}(\V(\varphi)w)$ have precisely
  the same spectrum when thought of as linear transformations.  If
  $\varphi$ is not an isomorphism, then even though~\eqref{eq:4.1-lin}
  is in the form of a similarity transformation, the relationship
  between the spectra of $S_\Gamma(w)$ and $S_{\Gamma'}(\V(\varphi)w)$
  is a complicated question,~q.v.~\cite{Knutson.Tao.01, Li.Poon.03}.  A partial resolution of this question
  will be useful for an application which we will present
  in~\cite{DeVille.Lerman.10.3}.
\end{remark}

\section{Groupoid invariant vector fields on manifolds}\label{sec:Man}

\noindent
We now extend the results of Section~\ref{sec:Euc} from the category
$\Euc$ of Euclidean spaces to the category  $\Man$ of (Hausdorff paracompact
finite dimensional $C^\infty$) manifolds.  Since the
constructions in Section~\ref{sec:Euc} are functorial and, in
particular, do not use coordinates, such an extension is not difficult.
However, we  do need to modify our definition of the control functor $\Ctrl$.

\subsection{Control systems functor }\label{subsec:csfMan}
\mbox{}\\[-8pt]

\noindent
Recall the construction
\[\Ctrl\colon  \FinTree/C
\to \Vect
\] 
in the case of Euclidean spaces. First, a choice of a phase space
function $\cP\colon C_0 \to \Euc$ gives rise to a phase space functor
\[
\PP\colon \op{(\Finset/C_0)} \to \Euc.
\]
Then given a finite tree $T$ over a graph $C$ of colors we defined
the space of control systems $\Ctrl(T)$ by
\[
\Ctrl(T) := \Hom_\Euc (\PP \lv T, \PP \rt T),
\]
which is canonically a subspace of the space of control systems
$\CT(\PP \lv T \times \PP \rt T \to \PP \rt T)$
(q.v.\ \ref{empt:2.1.3} and \ref{rmrk:ctrl-euc}).  

Suppose we mimic this definition in the case of manifolds. As remarked
in Section~\ref{sec:PP}, just as in the case of Euclidean spaces, a
choice of a phase space function $\cP\colon C_0 \to \Man$ gives rise
to a phase space functor
\[
\PP\colon \op{(\Finset/C_0)} \to \Man.
\]
On objects
\[
\PP X \equiv \PP (X \stackrel{\alpha}{\to} C_0):= 
\prod_{x\in X } \cP(\alpha (x)),
\]
where now the product is in the category $\Man$.  Since $\Man$ has
finite products, $\PP$ is well-defined on objects.  Implicit in this
definition is a {\em choice} of a manifold $\prod_{x\in X } \cP(\alpha
(x))$ and a family of submersions $\{p_x\colon \prod_{x'\in X }
\cP(\alpha (x'))\to \cP(\alpha (x))\}_{x\in X}$.  By the universal
property of products for any morphism $ \xy (-6, 6)*+{Y} ="1"; (6,
6)*+{X} ="2"; (0,-4)*+{C_0 }="3"; {\ar@{->}^{ \alpha} "2";"3"};
{\ar@{->}^{f} "1";"2"}; {\ar@{->}_{\beta} "1";"3"}; \endxy $ in the
slice category $\FinSet/C_0$, we get a map of manifolds $\PP f\colon
\PP X \to \PP Y$ satisfying
\[
\xy
(-18, 10)*+{ \PP X} ="1"; 
(18, 10)*+{ \PP Y} ="4"; 
(-18,-6)*+{\cP (\alpha (f(y)))}="2";
(18,-6)*+{\cP (\beta (y))} ="3";
{\ar@{->}_{ p_{f(y)}} "1";"2"};
{\ar@{->}^{q_{y}} "4";"3"};
{\ar@{-->}^{\PP f} "1";"4"};
{\ar@{=} "2";"3"};
\endxy ,
\]
q.v.\ \ref{rmrk:constr-of-P}.
However, the embedding \eqref{eq:2.1.4}
\[
\Hom_\Euc(V, W) \hookrightarrow \CT(V\times W\to W), \quad f\mapsto (pr_2, F)
\]
has no analogue in the category of manifolds.
Indeed what makes the embedding work is that the tangent space $TW$ of
a Euclidean space $W$ is canonically  isomorphic to $W\times W$:
$TW\stackrel{\simeq}{\to} W\times W$. 
Thus defining the control functor $\Ctrl_\Man$ on trees by 
\[
\Ctrl_\Man (T) = 
\Hom_\Man (\PP \lv T, \PP \rt T) = C^\infty (\PP \lv T, \PP \rt T)
\]
does not make sense for a general manifold.\footnote{One can try and
  salvage this approach by restricting the definition of the $\Ctrl$
  functor to a subcategory of canonically parallelizable manifolds
  (e.g., the category of  abelian Lie groups and smooth maps) and setting
\[
\Ctrl_\Man (T) := C^\infty (\PP \lv T, T_x (\PP \rt T))
\]
for some arbitrary point $x\in \PP \rt T$.  Any
trivialization of the fiber bundle  $T(\PP \rt T) \simeq \PP \rt T \times T_x
(\PP \rt T) $  would then give an embedding $C^\infty (V, W)
\hookrightarrow \CT(V\times W\to W)$.  We do not pursue this idea
further here, but see Golubitsky, Josi{\'c}, and
Shea-Brown~\cite{Golubitsky.Josic.Brown.06} for a beautiful
application.}  However, given a finite tree $T$ over $C$ and a phase
space functor $\PP\colon \op{(\FinSet/C)}\to \Man $ we do have a
canonical projection
\[
pr_2\colon  \PP \lv T \times \PP \rt T \to \PP \rt T,
\]
which is a surjective submersion.  Consequently we can talk about the
space $\CT(pr_2\colon \PP \lv T \times \PP \rt T \to \PP \rt T)$ of
all control systems supported on $pr_2$, see \ref{empt:CT}.
Therefore,

\begin{definition}[The space of control systems associated to a tree]
  \label{def:ctrl-Man} 
  Given a finite tree $T$ over a graph $C$ of colors and a phase space
  functor $\PP\colon \op{(\FinSet/C)} \to \Man$, we define the infinite
  dimensional vector space of {\em control systems associated to the
    tree $T \to C$} by
\begin{eqnarray*}
\lefteqn{\Ctrl_\Man (T) :=}\\
&&\{F\colon \PP \lv T \times \PP \rt T \to T \PP \rt T\mid
F(u,m) \in T_m (\PP \rt T)\textrm{ for all }(u,m)\in \PP \lv T \times
\PP \rt T \}\\
&&\equiv \CT(pr_2\colon  \PP \lv T \times \PP \rt T \to \PP \rt T).
\end{eqnarray*}
\end{definition}
\begin{remark}
  If the space of leaves $\lv T$ is empty (so that $T$ consists of a
  single vertex) then $\PP \lv T $ is a point and $\Ctrl_\Man (T)$ is
  simply the space of all vector fields on the phase space $\PP \rt T$
  (q.v. Remark~\ref{remark:vec-ctrl}).
\end{remark}

The assignment $T\mapsto \Ctrl_\Man (T)$ extends to a functor $\Ctrl_\Man\colon 
\FinTree/C \to \Vect$, from the category of finite trees over $C$ to the
category of infinite dimensional vector spaces as follows.
Let 
\[
\xy
(-6, 6)*+{T} ="1"; 
(6, 6)*+{T'} ="2";
(0,-4)*+{C }="3";
{\ar@{->}_{ \alpha} "1";"3"};
{\ar@{->}^{\sigma} "1";"2"};
{\ar@{->}^{\beta} "2";"3"};
\endxy 
\]
 be an isomorphism of trees.  Then we have two isomorphisms of
finite sets over $C_0$:
\[
\sigma|_{\lv T}\colon  \lv T \to \lv T' \quad \textrm{and} \quad
\sigma|_{\rt T}\colon \rt T \to \rt T'.
\]
Applying the phase space functor $\PP$ we get two diffeomorphisms of manifolds
\begin{equation*}
  \PP(\sigma|_{\lv T})\colon \PP \lv T'\to \PP \lv T,\quad
  \PP(\sigma|_{\rt T})\colon \PP \rt T'\to \PP \rt T.
\end{equation*}
Note that the second diffeomorphism is the identity map, see \ref{rmrk:2.11}.
In other words the square
\[\xymatrix@C+2cm{  \PP \lv T'\times  \PP \rt T'
\ar[r]^{\PP (\sigma|_{\lv T})\times \PP (\sigma|_{\rt T}) } \ar[d]^{pr_2} 
& \PP \lv T\times \PP \rt T\ar[d]_{pr_2}\\
\PP \rt T' \ar[r]_{\PP (\sigma|_{\rt T}) = id_{\PP(\rt T)}} & \PP \rt T
}
\]
commutes. Consequently we get a  map 
\begin{equation}\label{eq:1.17Man}
\Ctrl_\Man(\sigma)\colon  \Ctrl_\Man(T) \to \Ctrl_\Man (T')
\end{equation}
which is the pull-back by $\PP (\sigma|_{\lv T} ) \times \PP
(\sigma|_{\rt T}) = \PP (\sigma|_{\lv T} ) \times id_{\PP \rt T}$ composed with the pushforward by $d
(\PP \sigma|_{\rt T})$:
\begin{equation}\label{eq:1.18Man}
  \Ctrl_\Man (\sigma) X := 
d(\PP \sigma|_{\rt T}) \circ X 
\circ (\PP (\sigma|_{\lv T} ) \times \PP (\sigma|_{\rt T}))\colon  
  \PP \lv T' \times \PP \rt T'\to T\PP \rt T'
\end{equation}
for any $(X\colon  \PP \lv T \times \PP \rt T \to T\PP \rt T) \in \Ctrl_\Man
(T)$.  Once again we leave it to the reader to check that $\Ctrl_\Man$
preserves the composition of morphisms, that is, that $\Ctrl_\Man$ is a functor.

\begin{remark}
  Suppose the phase space function $\cP\colon C_0\to \Man$ happens to
  take its values in the Euclidean spaces.  The the functors $\Ctrl
  =\Ctrl_\Euc$ and $\Ctrl_\Man$ are quite different:

  Given a tree $T$, we defined $\Ctrl_\Euc(T) = C^\infty(\PP \lv T,\PP
  \rt T)$, and identified it with a space of control 
  system by mapping a function $f\in \Ctrl_\Euc(T)$ to the control
  system $F\colon \PP\lv T\times \PP\rt T\to T\PP\rt T$ by $F(v,w) =
  (w,f(v))$.  In particular $F$ so defined does not depend on the
  points of the Euclidean space $\PP \rt T$.
  On the other hand we defined $\Ctrl_\Man(T)$ as the space of all
  smooth maps $F\colon \PP\lv T\times \PP \rt T\to T\PP\rt T$ with
  $F(u,m)\in T_m(\PP \rt T)$.  So such a control system $F$ does in general
  depend on the points of the Euclidean space $\PP \rt T$.

  As a consequence the spaces of virtual groupoid invariant vector
  fields $\V(\Gamma)$ and $\V_\Man(\Gamma)$ and their respective
  images $S_\Gamma (\V(\Gamma))$, $S_\Gamma (\V_\Man (\Gamma))$ are
  different even if the phase space function $\cP$ is the same (the official 
  definitions of $\V_\Man(\Gamma)$ and $S_\Gamma \colon \V_\Man(\Gamma) \to
  \chi (\PP\Gamma_0)$ are given in the next two subsections).  For
  example, for the graph
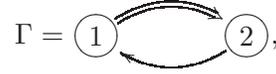
\[
\Gamma = \xy
 (-20,0)*++[o][F]{1}="1"; (0,0)*++[o][F]{2}="2"; 
{\ar@/^1.pc/  "2";"1"};  {\ar@{=>}@/^1.pc/ "1";"2"};
\endxy, 
\]
(this is a graph colored by a graph $C$ with one vertex and {\em two}
distinct edges) the vector fields in $S_\Gamma(\V(\Gamma))$ are of the
form
\begin{equation*}
  \dot x_1 = f(x_2),\quad \dot x_2 = g(x_1),
\end{equation*}
whereas the vector fields in $S_\Gamma(\V_\Man(\Gamma))$ are  of the form
\begin{equation*}
  \dot x_1 = k(x_2,x_1),\quad \dot x_2 = h(x_1,x_2).
\end{equation*}
\end{remark}

\subsection{Virtual groupoid invariant vector fields on manifolds}
\mbox{}\\[-8pt]

\noindent
In complete analogy with Definition~\ref{def:vgivf} we now have:

\begin{definition}[Virtual groupoid invariant vector fields on manifolds]
\label{def:vgivfMan} 
Fix the graph of colors $C=\{C_1 \toto C_0\}$ and a phase space
  function $\cP\colon  C_0 \to \Man$. Given a graph $\Gamma$ over $C$ we
  define the vector space of {\em virtual groupoid invariant vector
  fields} $\V_\Man (\Gamma)$ to be the limit of the restriction of the functor $\Ctrl_\Man$ to the 
groupoid $G(\Gamma)$:
\[
\V_\Man(\Gamma)  
:= \lim (\Ctrl_\Man|_{G(\Gamma)}\colon  G(\Gamma) \to \Vect).
\]
\end{definition}
\noindent
Note that since $\V_\Man(\Gamma)$ is a limit, it comes with a family of the
canonical projections
\[
\left\{\varpi_a\colon  \V_\Man(\Gamma) \to \Ctrl_\Man(I(a)) \right\}_{a\in\Gamma_0}.
\]
Immediately we have the analogue of Remark~\ref{rmrk:2.8.3}:
\begin{remark}\label{rmrk:2.8.3Man}
  The vector space $\V_\Man (\Gamma)$ has the following concrete
  description.  Consider the product $\prod _{a\in \Gamma_0}
  \Ctrl _\Man (I(a))$ with its canonical projections $p_b\colon \prod _{a\in
    \Gamma_0} \Ctrl _\Man (I(a)) \to \Ctrl _\Man (I(b))$, $b\in \Gamma_0$.  For a
  vector $X\in\prod _{a\in \Gamma_0} \Ctrl _\Man (I(a))$, denote the
  $a$-th component $p_a(X)$ of $X$ by $X_a$. Then
\[
\V_\Man (\Gamma) = \{X\in \prod _{a\in \Gamma_0} \Ctrl_\Man (I(a))\mid
\Ctrl _\Man (\sigma) X_a = X_b \textrm{ for all arrows }
I(a)\stackrel{\sigma}{\to }I(b)\in G(\Gamma)\}
\]
with the canonical projections $ \varpi_a\colon \V_\Man (\Gamma) \to
\Ctrl_\Man (I(a)) $ given by restrictions $p_a|_{\V_\Man (\Gamma)}$.
\end{remark}

\subsection{The relation of $\V_\Man (\Gamma)$ to vector fields on the phase
  space $\PP\Gamma_0$}\label{subsec:S-GammaMan}
\mbox{}\\[-8pt]

\noindent
Given a product $M= M_1\times\ldots \times M_k$ of manifolds, the space
of $C^\infty $ vector fields $\chi (M)$ is naturally a product of
control systems
\[
\chi (M_1\times\ldots \times M_k) = 
\prod_{i=1}^k \CT (p_i\colon  M_1\times\ldots \times M_k \to M_i).
\]
Therefore, given a phase space functor $\PP \colon  \op{(\FinGraph/C)} \to
\Man$ for each graph $\Gamma$ over the graph of colors $C$ we have
\[
\chi (\PP \Gamma_0) = 
\prod_{a\in \Gamma_0} \CT(\PP\kappa_a\colon  \PP \Gamma_0 \to \PP \{a\})
\]
where $\kappa_a\colon \{a\}\hookrightarrow \Gamma_0$ is the canonical
inclusion of a vertex $\{a\}$ into the set of vertices of the graph
$\Gamma$.  Note that the canonical projections 
\begin{equation}\label{eq:pi_aMan}
\pi_a\colon \prod_{a'\in \Gamma_0} \CT(\PP\kappa_{a'}\colon \PP
\Gamma_0 \to \PP \{a'\})\to \CT(\PP\kappa_a\colon \PP \Gamma_0 \to \PP
\{a\})
\end{equation}
are given by push-forwards:
\begin{equation}\label{eq:pi_aMan'}
\pi_a (X) = d\PP \kappa_a \circ X\equiv (d\PP \kappa_a )_* X.
\end{equation}
Since $\V _\Man (\Gamma)$ is a limit, for each node $a$ we also  have a canonical
projection
\[
\varpi_a \colon  \V _\Man \Gamma \to \Ctrl_\Man(I(a)).
\]
Recall that by definition $\Ctrl_\Man(I(a))$ is the space of all control
systems supported by the surjective submersion $\PP \lv I(a)\times \PP
\rt I(a) \xrightarrow{pr_2} \PP \rt I(a)$:
\[
\Ctrl_\Man (I(a)) = \CT(pr_2\colon\PP \lv I(a)\times \PP \rt I(a) \to \PP \rt I(a)).
\]
Since the set $I(a)_0$ of nodes of the input tree $I(a)$ is simply
the union of the leaves $\lv I(a)$ and the root $\rt I(a) =\{a\}$ we have
\begin{equation}
\Ctrl_\Man(I(a)) = \CT(\PP I(a)_0\to \PP \{a\}).
\end{equation}
Recall that for each input tree $I(a)$ of a graph $\Gamma$ we have a
map of graphs $\xi\colon I(a)\to\Gamma$ (q.v.\ \ref{empt:xi}), hence a map 
\[
\xi\colon  I(a)_0 \to \Gamma_0
\]
of finite sets over the set of colors $C_0$.  Applying the phase space
functor $\PP$ we obtain canonical maps of manifolds
\[
\PP(\xi)\colon  \PP\Gamma_0 \to \PP I(a)_0.
\]
The pullback by $\PP(\xi)$ gives
\[
\PP(\xi)^*\colon \Ctrl_\Man (I(a))= \CT(\PP I(a)_0\to \PP \{a\}) \to
\CT(\PP\Gamma_0\to \PP \{a\}).
\]
By the universal properties of the product 
\[
\{ \pi_a \colon \chi (\PP\Gamma_0)\to \CT (\PP \Gamma_0 \to \PP \{a\})
\}_{a\in \Gamma_0}
\]
there a  unique canonical map $S_\Gamma\colon  \V_\Man (\Gamma) \to \chi (\PP \Gamma_0)$ 
making 
the diagram
\begin{equation} \label{eq:4.3.4}
\xy
(-22, 9)*+{\V_\Man \Gamma}="1"; 
(24,9)*+{\chi (\PP \Gamma_0)}="2";
(-22, -10)*+{\Ctrl_\Man (I(a))}="3"; 
(24,-10)*+{\CT(\PP\Gamma_0, \PP \{a\})}="4";
{\ar@{-->}^{\exists ! S_\Gamma} "1";"2"};
{\ar@{->}_{\varpi_a} "1";"3"};
{\ar@{->}^{\pi_a} "2";"4"};
{\ar@{->}_{\PP(\xi)^*} "3";"4"};
\endxy
\end{equation}
commute.

\subsection{The assignment $\Gamma\mapsto \V_\Man (\Gamma)$ extends to
  a functor} 
\mbox{}\\[-8pt]

\noindent
As in the case of groupoid invariant smooth vector fields on
Euclidean spaces and in the case of groupoid invariant linear vector
fields we have
\begin{theorem}\label{thmA''}
  The map $\V_\Man$ that assigns to each finite graph $\Gamma$ over
  $C$ the vector space of $G(\Gamma)$ invariant virtual  vector
  fields on $\PP\Gamma_0$ extends to a contravariant functor
\[
\V_\Man\colon  \opet{(\FinGraph/C)} \to \Vect.
\]
\end{theorem}

\begin{proof}
We mimic the proof of Theorem~\ref{thmA'}:
Consider an \'etale map of graphs over $C$: 
\[
\xy
(-6, 6)*+{\Gamma} ="1"; 
(6, 6)*+{\Gamma'} ="2";
(0,-4)*+{C }="3";
{\ar@{->}_{ \alpha} "1";"3"};
{\ar@{->}^{\varphi} "1";"2"};
{\ar@{->}^{\beta} "2";"3"};
\endxy ,
\] 
that is, a morphism in the category $\et{(\FinGraph/C)}$.  Given an
arrow $I(a)\stackrel{\sigma}{\longrightarrow}I(b)$ in the groupoid
$G(\Gamma)$ we have a commuting diagram \eqref{eq:3.1}
\[
\xymatrix{ 
I(a) \ar[r]^{\sigma} \ar[d]^{\varphi_a} & I(b)\ar[d]^{\varphi_b}\\
I(\varphi(a))\ar[r]_{G(\varphi)\sigma} & I(\varphi(b))
}
\] 
in the category of finite trees over $C$.  Applying the control
functor $\Ctrl_\Man$ gives us a commuting diagram in the category
of vector spaces (q.v.\ \eqref{eq:3.2} and \eqref{eq:3.2'}):
\begin{equation}\label{eq:3.2''}
  \xymatrix@C+2cm{ 
    \Ctrl_\Man(I(a)) \ar[r]^{\Ctrl_\Man(\sigma)} 
\ar[d]_{\Ctrl_\Man(\varphi_a)} 
    &\Ctrl_\Man( I(b))\ar[d]^{\Ctrl_\Man(\varphi_b)}\\
    \Ctrl_\Man(I(\varphi(a))
\ar[r]_{\Ctrl_\Man(G(\varphi)\sigma)} & \Ctrl_\Man(I(\varphi(b)),
  }
\end{equation}
which, in turn, gives us the 2-commutative diagram
\[
\xy
(-10, 10)*+{G(\Gamma)} ="1"; 
(10, 10)*+{G(\Gamma')} ="2";
(0,-6)*+{\Vect }="3";
{\ar@{->}_{ \Ctrl_\Man |_{G(\Gamma)}} "1";"3"};
{\ar@{->}^{G(\varphi)} "1";"2"};
{\ar@{->}^{\Ctrl_\Man|_{G(\Gamma')}} "2";"3"};
{\ar@{=>}_<<<{\scriptstyle \phi} (4,6)*{};(-2,2)*{}} ;  
\endxy ,      
\] 
q.v.\ Remark~\ref{rmrk:2.10.11}.  This 2-commuting triangle is a
morphism in $\Cat/\Vect$ (see \ref{empt:5.1.1}).  It is not hard to
check that the assignment
\[
\mathcal{G}: (\FinGraph/C)_{et} \to \Cat/\Vect, \quad
\xy
(-6, 6)*+{\Gamma} ="1"; 
(6, 6)*+{\Gamma'} ="2";
(0,-4)*+{C }="3";
{\ar@{->}_{ \alpha} "1";"3"};
{\ar@{->}^{\varphi} "1";"2"};
{\ar@{->}^{\beta} "2";"3"};
\endxy  
\mapsto
\xy
(-10, 10)*+{G(\Gamma)} ="1"; 
(10, 10)*+{G(\Gamma')} ="2";
(0,-6)*+{\Vect }="3";
{\ar@{->}_{ \Ctrl_\Man |_{G(\Gamma)}} "1";"3"};
{\ar@{->}^{G(\varphi)} "1";"2"};
{\ar@{->}^{\Ctrl_\Man|_{G(\Gamma')}} "2";"3"};
{\ar@{=>}_<<<{\scriptstyle \phi} (4,6)*{};(-2,2)*{}} ;   
\endxy          
\] 
preserves the composition of morphisms and thus is a functor.
Composing the functor $\mathcal{G}$ with the contravariant functor
$L\colon \op{\Cat/\Vect} \to \Vect$ (see \ref{prop:5.1.11}) gives us the desired
contravariant functor
\[
\V_\Man = \mathcal{G}\circ L\colon  {(\FinGraph/C)_{et}}^{op} \to \Vect.
\] 
\end{proof} 

\begin{remark}\label{rmrk:4alphaMan}
  Unraveling the definition of $\V_\Man$ we see that for any \'etale
  map of graphs $\varphi\colon \Gamma\to \Gamma'$, any node $a$ of $\Gamma$
  and any virtual groupoid invariant vector field $w\in \V_\Man
  (\Gamma')$ we have
\begin{equation}\label{eq:444}
\varpi_a (\V(\varphi)w) = \Ctrl _\Man (\varphi_a)\inv (\varpi_{\varphi (a)} w)
\end{equation}
Compare with \eqref{eq4:alpha} and \eqref{eq:3.3}.
\end{remark}

\subsection{The  combinatorial category  $\cV^\Man$ of dynamical systems on manifolds}
\mbox{}\\[-8pt]

\noindent
Associated to the functor $\V_\Man$ we have the category of elements
$\cV^\Man$ and a functor $\cV^\Man \to \op{(\FinGraph/C)_{et}}$.  As
in the case of $\cV^\Euc$ the objects of $\cV^\Man$ are pairs of the
form $(\Gamma, v)$ with $\Gamma$ a finite graph over $C$ and $v\in
\V_\Man(\Gamma)$.  The morphism in $\cV^\Man$ are the tuples of the
form $(h,(v,\Gamma), (v', \Gamma'))$ with $h\colon \Gamma'\to \Gamma$
a morphism in $\FinGraph/C$ and $\V_\FinGraph (h)v = v'$.
As in the case of $C^\infty$ vector fields on Euclidean spaces the
maps 
\[
\{S_\Gamma\colon  \V_\Man (\Gamma) \to \chi (\PP
\Gamma_0, \PP \Gamma_0)\}_{\Gamma\in (\FinGraph/C)_0}
\]
 assemble into a
functor $S$ from the category of elements $\cV^\Man$ to the
category $\mathsf{Dynamical\,\, Systems}$ of  dynamical systems.
The key technical result in proving that $S$ is a functor is the
analogue of Theorem~\ref{thm:sync}:

\begin{theorem}\label{thm:sync-Man}
Let $w\in \V_\Man(\Gamma')$ be a virtual dynamical system on $\PP\Gamma_0$.
For each \'etale map of graphs 
\[
\varphi\colon  \Gamma \to \Gamma'
\]
 the diagram 
\begin{equation}\label{eq:4.1Man}
\xymatrix{
 \PP(\Gamma'_0) \ar[r]^{d \PP\varphi } 
&   \PP(\Gamma_0)   \\
\PP(\Gamma'_0) \ar[r]^{\PP \varphi}  \ar[u]^{S_{\Gamma'}(w)}
&  \PP(\Gamma_0)\ar[u]_{S_\Gamma(\V (\varphi)w)} }
\end{equation}
commutes: for any point $x\in \PP (\Gamma'_0)$
\begin{equation}\label{eq:1.36Man}
(d \PP\varphi) _x  \left( S_{\Gamma'}(w) \, (x) \right)=
 S_\Gamma(\V (\varphi)w)\,(\PP\varphi (x)).
\end{equation}
In other words, 
\[
\PP (\varphi)\colon  (\PP(\Gamma'), S_{\Gamma'}(w))\to  (\PP(\Gamma), S_\Gamma(\V (\varphi)w))
\]
is a map of dynamical systems.
\end{theorem}
\begin{proof}
  We modify the proof of Theorem~\ref{thm:sync} to take into account
  the difference between the functors $\Ctrl= \Ctrl_\Euc$ and
  $\Ctrl_\Man$.  As before, for each node $a$ of a graph $\Gamma$ over
  the graph of colors $C$ we have a canonical map $\xi\colon I(a)\to \Gamma$
  from the input tree $I(a)$ to the graph
  $\Gamma$. As above we use the same letter $\xi$ for all nodes of
  $\Gamma$ and for all graphs over $C$ suppressing both dependencies.

  If $\varphi\colon \Gamma\to \Gamma'$ is an \'etale map of graphs over
  $C$, then by definition of ``\'etale'' for each node $a$ of $\Gamma$
  we have an induced isomorphism of graphs over $C$:
\[
  \varphi_a\colon  I(a) \to I(\varphi(a)).
\]
Moreover, the diagram of graphs over $C$
\begin{equation}\label{eq:(1)Man}
\xy
(-14, 10)*+{I(a)}="1"; 
(14,10)*+{\Gamma}="2";
(-14, -10)*+{I(\varphi(a))}="4"; 
(14,-10)*+{\Gamma'}="5";
{\ar@{->}^{\xi} "1";"2"};
{\ar@{->}^{\varphi_a} "1";"4"};
{\ar@{->}_{\varphi} "2";"5"};
{\ar@{->}_{\xi} "4";"5"};
\endxy
\end{equation}
commutes.   Hence we have a commuting diagram
\begin{equation}\label{eq:(1'')Man}
\xy
(-14, 10)*+{\PP I(a)_0}="1"; 
(14,10)*+{\PP\Gamma_0}="2";
(-14, -10)*+{\PP I(\varphi(a))_0}="4"; 
(14,-10)*+{\PP\Gamma'_0}="5";
{\ar@{->}_{\PP \xi} "2";"1"};
{\ar@{->}^{\PP \varphi_a} "4";"1"};
{\ar@{->}_{\PP \varphi} "5";"2"};
{\ar@{->}_{\PP\xi} "5";"4"};
\endxy
\end{equation}
in our category $\Man$ of phase spaces.  As before let 
\[
\kappa_a\colon  \{a\} \hookrightarrow \Gamma_0
\]
denote the canonical inclusion in the category of finite sets over
$C_0$.  We then have a commuting diagram in $\Finset/C_0$
\begin{equation}\label{eq:(2)Man}
\xy
(-14, 10)*+{\{a\}}="1"; 
(14,10)*+{\Gamma_0}="2";
(-14, -10)*+{\{\varphi(a)\}}="4"; 
(14,-10)*+{\Gamma'_0,}="5";
{\ar@{^{(}->}^{\kappa_a} "1";"2"};
{\ar@{->}_{\varphi_a|_{\{a\}}= \varphi|_{\{a\}}} "1";"4"};
{\ar@{->}^{\varphi} "2";"5"};
{\ar@{^{(}->}_{\kappa_{\varphi(a)}} "4";"5"};
\endxy
\end{equation}
hence a commuting diagram in $\Man$:
\begin{equation}\label{eq:(3)Man}
\xy
(-14, 10)*+{\PP \{a\}}="1"; 
(14,10)*+{\PP \Gamma_0}="2";
(-14, -10)*+{\PP \{\varphi(a)\}}="4"; 
(14,-10)*+{\PP \Gamma'_0,}="5";
{\ar@{->>}_{\PP \kappa_a} "2";"1"};
{\ar@{->}^{\PP (\varphi|_{ \{a\}})} "4";"1"};
{\ar@{->}_{\PP \varphi} "5";"2"};
{\ar@{->>}^{\PP \kappa_{\varphi(a)}} "5";"4"};
\endxy
\end{equation}
Differentiating \eqref{eq:(3)Man} we get the commuting diagram
\begin{equation}\label{eq:(4)Man}
\xy
(-14, 10)*+{T\PP \{a\}}="1"; 
(14,10)*+{T\PP \Gamma_0}="2";
(-14, -10)*+{T\PP \{\varphi(a)\}}="4"; 
(14,-10)*+{T\PP \Gamma'_0,}="5";
{\ar@{->>}_{d\PP \kappa_a} "2";"1"};
{\ar@{->}^{d\PP (\varphi|_{ \{a\}})} "4";"1"};
{\ar@{->}_{d\PP \varphi} "5";"2"};
{\ar@{->>}^{d\PP \kappa_{\varphi(a)}} "5";"4"};
\endxy
\end{equation}
of tangent bundles.  Recall that the projection 
\[
\pi_a\colon \chi (\PP \Gamma_0) \to \CT(\PP \Gamma_0\to \PP \{a\})
\]
is given by 
\[
\pi_a (X) = d (\PP \kappa_a) \circ X,
\]
(q.v.\ \eqref{eq:pi_aMan'}). 
By definition of the maps $S_\Gamma$  (q.v.
\eqref{eq:2.2}) 
we have
\begin{equation}\label{eq:4*Man}
d\PP \kappa _a \circ S_\Gamma (v) = \varpi_a (v)\circ \PP (\xi) 
\quad
\textrm{ for any $a\in \Gamma_0$ and any }v\in \V(\Gamma) .
\end{equation}
Similarly for $\Gamma'$ we have
\begin{equation}\label{eq:4*'Man}
  d\PP \kappa _{\varphi(a)} \circ S_{\Gamma'} (v') 
  = \varpi_{\varphi (a)}(v') \circ \PP( \xi) \quad
  \textrm{ for any $a\in \Gamma_0$ and any }v'\in \V(\Gamma'). 
\end{equation}
Recall that the definition  of $\V_\Man (\varphi)$ implies that
\begin{equation}\label{eq4:alphaMan}
\varpi_a (\V(\varphi)w) = \Ctrl (\varphi_a)\inv (\varpi_{\varphi (a)}
w) 
\end{equation}
(q.v.\ Remark~\ref{rmrk:4alphaMan}).  Since the set of nodes of the
input tree $I(a)$ is the union of its root and its leaves
\eqref{eq:1.18Man} and \eqref{eq4:alphaMan} imply that
\begin{equation}\label{eq4:2alphaMan}
\varpi_a(\V(\varphi)w) \circ \PP \varphi_a = 
d(\PP\varphi|_{\{a\}})\circ \varpi_{\varphi(a)}w
\end{equation}
for any node $a$ of $\Gamma$ and any $w\in\V(\Gamma')$.

Since the vector space of vector fields $\chi (\PP\Gamma_0)$ is
the product
\[
\chi (\PP\Gamma_0)= \prod_{a\in \Gamma_0} CT(\PP \Gamma_0\to \PP \{a\})
\]
the diagram \eqref{eq:4.1Man} commutes if and only if 
\[
\pi_a (S_\Gamma (\V_\Man (\varphi) w) \circ \PP \varphi ) = 
\pi_a (d\PP \varphi \circ S_{\Gamma'} (w))
\]
for every $a\in \Gamma_0$.
We now compute: 
\begin{eqnarray*}
\pi_a (S_\Gamma (\V(\varphi) w) \circ \PP \varphi )
  &= &  d\PP\kappa_a \circ S_\Gamma (\V (\varphi) w) \circ \PP\varphi\\
&= & 
\left(\varpi_a(\V(\varphi)w)\circ \PP \xi\right)
\circ \PP\varphi, 
  \quad \textrm{ by (\ref{eq:4*Man})} \\
& = &\varpi_a(\V(\varphi)w)\circ 
  \left(\PP \xi \circ \PP\varphi\right)\\
&=& \varpi_a(\V(\varphi)w)\circ \PP \varphi_a \circ \PP \xi
  \quad \textrm{ since \eqref{eq:(1'')Man} commutes}\\
& = &  d (\PP \varphi|_{\{a\}})\circ \varpi_{\varphi(a)} w\circ \PP \xi,\quad 
\textrm{ by (\ref{eq4:2alphaMan})}\\ 
 &=& d \PP (\varphi|_{\{a\}}) \circ d\PP\kappa_{\varphi(a)} 
  \circ S_{\Gamma'} (w) \quad \textrm{ by (\ref{eq:4*'Man})}\\
&=&d\PP \kappa_a \circ d\PP \varphi \circ  S_{\Gamma'} (w) 
  \quad \textrm{since \eqref{eq:(4)Man} commutes} \\ 
  &=& \pi_a (d\PP \varphi ( S_{\Gamma'} (w))).
\end{eqnarray*}
\end{proof}
\noindent
Consequently we have (cf.\ Sections~\ref{sec:VEuc} and
\ref{subsec:cVFinVect})
\begin{theorem}\label{thm:2.12.3Man}
  Let $\cV^\Man$ denote the category of virtual groupoid-invariant
  dynamical systems constructed above, $\pi\colon \cV^\Man \to
  \opet{(\FinGraph/C)}$ the canonical projection, $U\colon
  \mathsf{Dynamical \,\,Systems} \to \Man$ the forgetful functor, and
  $\PP$ the phase space functor extended to finite graphs (q.v.\
  Remark~\ref{rmrk:PonFinGraph}). There exists a functor $S\colon
  \cV^\Man \to \mathsf{Dynamical \,\, Systems}$ making the
  diagram
\begin{equation}\label{eq:star-Man}
\xymatrix{
\cV^\Man  \ar[r]^{S \quad\quad } \ar[d]_{\pi}
&  \mathsf{Dynamical \,\,Systems}\ar[d]^{U}  \\
 \opet{(\FinGraph/C)}  \ar[r]^{\quad\PP }
&  \Man }
\end{equation}
commute.  On objects
\[
S(v,\Gamma) = (\PP \Gamma_0, S_\Gamma (v));
\]
on arrows
\[
S(h,(v,\Gamma), (v', \Gamma')) = (\PP \Gamma_0, S_\Gamma(v))
\stackrel{\PP h}{\to} (\PP\Gamma', S_{\Gamma'} (v)).
\]
\end{theorem}


\section{Groupoid invariance versus group invariance} \label{sec:inv}

\subsection{}
\mbox{}\\[-8pt]

\noindent
Suppose the  group of symmetries $H$a colored graph $\Gamma\to C$, 
\[
H =\{ \varphi:\Gamma \to \Gamma\mid \varphi \textrm{ is an isomorphism
  of graphs over } C\},
\]
is nontrivial.
Then given a phase function $\cP:C_0 \to \Euc$ we can associate to the
graph $\Gamma$ two kinds of invariant vector fields on the phase space
$\PP \Gamma_0$: the space virtual groupoid invariant vector fields
$\V(\Gamma)$  and the space $\chi(\PP \Gamma_0)^H$ of $H$-invariant vector fields.
In this section we show that the image of the map 
\[
S_\Gamma \colon \V(\Gamma) \to \chi (\PP\Gamma_0)
\]
always lands inside the space of $H$-invariant vector fields.  In
other words groupoid invariant vector fields are always group
invariant.  However, it is easy enough to construct examples (and we
do so in this section) where the inclusion $S_\Gamma (\V(\Gamma))
\subset \chi (\PP \Gamma_0)^H$ is strict.  Consequently there may be
many more group invariant vector fields than there are groupoid
invariant vector fields, and what may be generic for group invariant
vector fields need not be generic for groupoid invariant vector
fields.
\begin{remark}
  The same results with the essentially the same proofs hold in the category of manifolds (when $\cP$ takes values in $\Man$) and in the category of finite dimensional vector spaces (when $\cP$ takes values in $\FinVect$).
\end{remark} 
\begin{proposition}\label{prop:3.7}
Let  $\varphi\colon \Gamma \to \Gamma$ be an isomorphism of graphs over $C$.
Then $\V(\varphi)\colon \V\Gamma\to \V\Gamma$ is the identity map.
\end{proposition}

\begin{proof} Since $\varphi$ is an isomorphism of graphs, it is
  \'etale.  By definition \eqref{eq:3.3} of $\V\varphi$, it is the
  unique linear map so that
\begin{equation}\label{eq:3.11}
\varpi_a \circ \V \varphi = \Ctrl(\varphi_a)\inv \circ \varpi _{\varphi(a)}
\end{equation}
for all vertices $a$ of the graph $\Gamma$.  But $\varphi_a\colon I(a)\to
I(\varphi(a))$ is an arrow in the groupoid $G(\Gamma)$, hence, by
definition of $\V\Gamma$,
\[
\Ctrl(\varphi_a)\circ \varpi_a = \varpi_{\varphi(a)}.
\]
Consequently
\begin{equation}\label{eq:3.12}
\varpi_a \circ id_{\V\Gamma} =  
 \Ctrl(\varphi_a)\inv \circ \varpi _{\varphi(a)}
\end{equation}
Comparing (\ref{eq:3.12}) and (\ref{eq:3.11}) we see that $\V\varphi$ must be
$id_{\V\Gamma}$.
\end{proof}

\begin{theorem}\label{thm:5.1.5}
  Let $H$ be  the  group of automorphisms of  a colored graph $\Gamma \to C$.
  Then the image $S_\Gamma (\V \Gamma)$ of
  virtual groupoid invariant vector fields is a subspace of $H$-invariant
  vector fields $\chi (\PP \Gamma)^H$.
\end{theorem}

\begin{proof}
  By Theorem~\ref{thm:sync}, for any virtual vector field $f\in \V(\Gamma)$, the
  vector fields $S_\Gamma (f) $ and $S_\Gamma (\V(h)f)$ are $\PP
  h$-related for any isomorphism $H\ni h\colon \Gamma\to \Gamma$.  By
  Proposition~\ref{prop:3.7}, $\V(h)f = f$.  
  Hence $S_\Gamma (f)$ is $\PP h$ related to
  itself for any $h\in H$, i.e., is $H$-invariant.
\end{proof}

\begin{example}\label{ex:5.1.6}
  Let $C$ be a graph with one vertex $v$ and one edge: $C =
  \,\xymatrix{ *+[o][F]{} \ar@(dr,ur)}$\quad\quad and let $\cP\colon  C_0
  =\{\circ\}\to \Euc$ be given by $\cP (\circ) = \R$.  Let $\Gamma$ be a
  directed triangle:
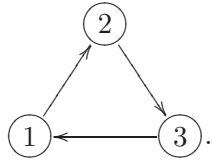
\[
\xy
(-10,0)*++[o][F]{1}="1"; 
(0,15)*++[o][F]{2}="2"; 
(10,0)*++[o][F]{3}="3"; 
{\ar@{->} "1";"2"};
{\ar@{->} "2";"3"};{\ar@{->} "3";"1"};
\endxy .
\]
Then $\PP\Gamma = \R^3$.   The cyclic group  of order $3$ 
generated by
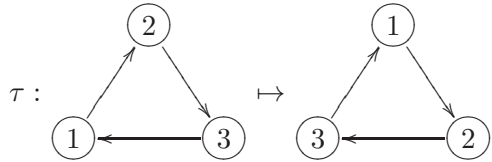
\[
\tau: 
\xy
(-10,-6)*++[o][F]{1}="1"; 
(0,9)*++[o][F]{2}="2"; 
(10,-6)*++[o][F]{3}="3"; 
{\ar@{->} "1";"2"};
{\ar@{->} "2";"3"};{\ar@{->} "3";"1"};
\endxy
\mapsto 
 \xy
(-10,-6)*++[o][F]{3}="1"; 
(0,9)*++[o][F]{1}="2"; 
(10,-6)*++[o][F]{2}="3"; 
{\ar@{->} "1";"2"};
{\ar@{->} "2";"3"};{\ar@{->} "3";"1"};
\endxy
\]
is the group $H$ of automorphisms of the graph $\Gamma$.
Then $\PP\tau\colon \R^3 \to \R^3$ is given by 
\[
\PP\tau (x_1, x_2, x_3) = (x_3, x_1, x_2).
\]
As before we identify the vector fields $\chi(\R^3)$ with smooth maps
$C^\infty(\R^3, \R^3)$.  Then a vector field $F(x) = (F_1(x), F_2(x),
F_3 (x))$ is $\tau$-invariant (and hence $H$-invariant) if and only if
\[
F (\PP \tau (x)) = \PP \tau (F(x)).
\]
Hence
\begin{eqnarray*}
F_1(x_2, x_3, x_1)&=&F_2 (x_1, x_2 , x_3 )\\
F_2 (x_2, x_3 , x_1 )&=& F_3 (x_1, x_2 , x_3 ).
\end{eqnarray*}
Therefore the space $\chi (\PP \Gamma_0)^H$ of $H$-invariant vector
fields is isomorphic to $C^\infty (\R^3, \R)$.  The isomorphism is
given by
\begin{eqnarray*}
C^\infty (\R^3, \R)& \to & C^\infty (\R^3, \R^3)^H = \chi (\PP\Gamma)^H\\
h(x)& \mapsto & (h(x), h(\PP\tau x) , h((\PP\tau)^2x).
\end{eqnarray*}
On the other hand, all the input trees of $\Gamma$ are of the form
\[\xy
(-20,0)*++[o][F]{}="1"; 
(0,0)*++[o][F]{}="2"; 
{\ar@{->}^{} "1";"2"};
\endxy
\]
and are all isomorphic to each other.  We conclude two things: for any
vertex $a$ of $\Gamma$, $\Ctrl (I(a)) \simeq C^\infty (\R, \R)$ and
the groupoid $G(\Gamma)$ is equivalent to the trivial category with
one object and one arrow.  Hence by Theorem~\ref{thm:inv-of-groups}
the space $\V(\Gamma)$ can be identified with the vector space
$C^\infty (\R, \R)$.  Tracing through the identifications we see that
$S_\Gamma\colon \V\Gamma \to \chi (\PP\Gamma)$ is given by the
injective map
\begin{eqnarray*}
  S_\Gamma \colon  \V\Gamma \simeq C^\infty (\R, \R)& \to & 
C^\infty (\R^3, \R)\simeq \chi (\PP\Gamma)^H \\
  f(u)& \mapsto & (f(x_3), f(x_1), f(x_2)).
\end{eqnarray*}
Clearly the space of groupoid-invariant vector fields $\V\Gamma$ is
much smaller than the space of group-invariant vector fields $\chi
(\PP\Gamma)^H$.

\end{example}

\section{Functoriality of limits}\label{sec:2slice}

\subsection{}
\mbox{}\\[-8pt]

\noindent
The goal of this section is to set up the category theoretic framework
alluded to in \ref{rmrk:2.10.11} and prove \ref{thm:ODEequiv} and
\ref{thm:inv-of-groups}.

\begin{notation}
The symbol $\Cat$ denotes the collection of all small categories.
\end{notation}
\begin{empt}\label{empt:5.1.1}
It is not too wrong to think of $\Cat$ as a category whose objects are
small categories and morphisms are functors.  But it is also not
completely accurate since there are also natural transformations
between functors.

Fix a category $\D$.  Assume that for any functor $f\colon \sfC \to \D$,
with $\sfC$ small, the limit $\lim (f\colon \sfC\to \D)$ exists in $\D$,
that is, assume that $\D$ is a {\em complete} category
(q.v.~Definition~\ref{def:complete-cat}).  The category $\D$ itself
need not be small.  For our applications we want $\D$ to be the
category $\Vect$ of not necessarily finite dimensional vector spaces,
but nothing that we do in this section requires $\D$ to be $\Vect$.

Consider the new category $\Cat/\D$ with
objects
\[
(\Cat/\D)_0 =\{ \X \stackrel{a}{\to}\D \mid \X\textrm{ is a small
  category, } a \textrm{ is a functor}\}.
\]
A morphism in $\Cat/\D$ from $\Y\stackrel{b}{\to}\D$ to
$X\stackrel{a}{\to} \D)$ is a pair $(f,\alpha)$ where $f\colon \Y\to \X$ is
a functor and $\alpha\colon  af\Rightarrow b$ a natural transformation.  We
picture morphisms as 2-commutative triangles:

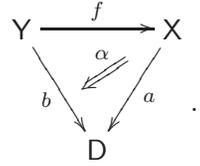
\[
\xy
(-10, 10)*+{\Y} ="1"; 
(10, 10)*+{\X} ="2";
(0,-6)*+{\D }="3";
{\ar@{->}_{b} "1";"3"};
{\ar@{->}^{f} "1";"2"};
{\ar@{->}^{a} "2";"3"};
{\ar@{=>}_<<<{\scriptstyle \alpha} (4,6)*{};(-2,2)*{}} ; 
\endxy .
\] 
The composition $*$ of morphisms in $\D$ is defined by
\[
(f,\alpha)* (g,\beta)  = (fg, \beta\circ_V (\alpha\circ_H g))
\]
for any pair $(g,\beta)\colon (\Z\stackrel{c}{\to} \D) \to
(\Y\stackrel{b}{\to}\D)$, $(f,\alpha)\colon (\Y\stackrel{b}{\to}\D)\to
(\X\stackrel{a}{\to}\D)$ of composable morphisms of $\Cat/\D$.  Here
$\circ_V$ and $\circ_H$ denote the vertical and horizontal
compositions (see \ref{def:comp-transf}).  We picture the composition
$*$ as pasting of two 2-commuting triangles:
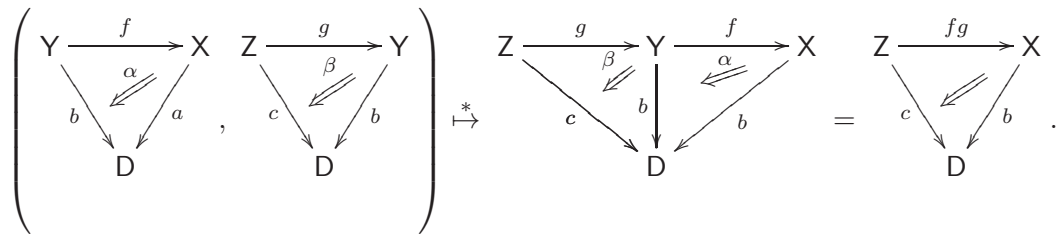
\[
\left(   
\xy
(-10, 10)*+{\Y} ="1"; 
(10, 10)*+{\X} ="2";
(0,-6)*+{\D }="3";
{\ar@{->}_{b} "1";"3"};
{\ar@{->}^{f} "1";"2"};
{\ar@{->}^{a} "2";"3"};
{\ar@{=>}_<<<{\scriptstyle \alpha} (4,6)*{};(-2,2)*{}} ; 
\endxy ,  
\xy
(-10, 10)*+{\Z} ="1"; 
(10, 10)*+{\Y} ="2";
(0,-6)*+{\D }="3";
{\ar@{->}_{c} "1";"3"};
{\ar@{->}^{g} "1";"2"};
{\ar@{->}^{b} "2";"3"};
{\ar@{=>}_<<<{\scriptstyle \beta} (4,6)*{};(-2,2)*{}} ; 
\endxy\right)  
\stackrel{*}{\mapsto}  
\xy
(-20, 10)*+{\Z} ="-1"; 
( 0, 10)*+{\Y} ="1"; 
(20, 10)*+{\X} ="2";
(0,-6)*+{\D }="3";
{\ar@{->}_{c} "-1";"3"};
{\ar@{->}_{c} "-1";"3"};
{\ar@{->}^{g} "-1";"1"}; 
{\ar@{->}_{b} "1";"3"};
{\ar@{->}^{f} "1";"2"};
{\ar@{->}^{b} "2";"3"};
{\ar@{=>}_<<<{\scriptstyle \alpha} (12,7)*{};(6,5)*{}} ; 
{\ar@{=>}_<<<{\scriptstyle \beta} (-3,7)*{};(-7,4)*{}} ; 
\endxy 
= \xy
(-10, 10)*+{\Z} ="1"; 
(10, 10)*+{\X} ="2";
(0,-6)*+{\D }="3";
{\ar@{->}_{c} "1";"3"};
{\ar@{->}^{fg} "1";"2"};
{\ar@{->}^{b} "2";"3"};
{\ar@{=>}_<<<{\scriptstyle } (4,6)*{};(-2,2)*{}} ; 
\endxy .
\]
One checks that $*$ is
associative and that $\Cat/\D$ is a category.
\end{empt}.
\begin{proposition}\label{prop:5.1.11}
We use the notation above.  The assignment
\[
L\colon  (\Cat/\D)_0 \to \D_0, \quad L(\X\stackrel{a}{\to}\D) : = 
\lim (\X\stackrel{a}{\to}\D)
\]
extends to a contravariant functor
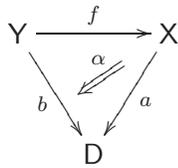
\[
L\colon  \op{(\Cat/\D)}\to \D.
\]
\end{proposition}
\begin{proof}
Consider a morphism $\xy
(-10, 10)*+{\Y} ="1"; 
(10, 10)*+{\X} ="2";
(0,-6)*+{\D }="3";
{\ar@{->}_{b} "1";"3"};
{\ar@{->}^{f} "1";"2"};
{\ar@{->}^{a} "2";"3"};
{\ar@{=>}_<<<{\scriptstyle \alpha} (4,6)*{};(-2,2)*{}} ; 
\endxy$ in $\Cat/D$.  Let $p_x\colon  \lim (\X\stackrel{a}{\to}\D) \to
a(x)$, $x\in \X_0$ and $q_y\colon  \lim (\Y\stackrel{b}{\to}\D) \to b(y)$,
$y\in \Y_0$ denote the canonical projections.  For any arrow
$y\stackrel{\gamma}{\to}y'$ in $\Y$ we have a commuting diagram
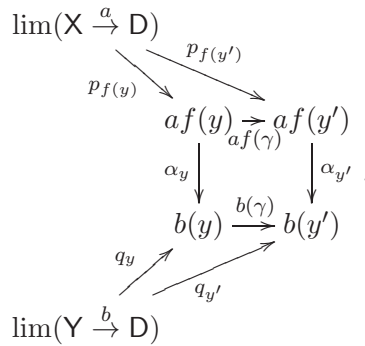
\[
\xy (-5,20)*+{\lim(\X\stackrel{a}{\to}\D) }="1";
(10,7)*+{af(y)}="2"; 
(25,7)*+{af(y')} ="3";
(-5,-20)*+{\lim(\Y\stackrel{b}{\to}\D) }="4";
(10,-7)*+{b(y)}="5"; 
(25,-7)*+{b(y')} ="6";
{\ar@{->}_{p_{f(y)}} "1";"2"}; 
{\ar@{->}^{p_{f(y')}} "1";"3"};
{\ar@{->}_{af(\gamma)} "2";"3"};
{\ar@{->}^{b(\gamma)} "5";"6"};
{\ar@{->}^{q_y} "4";"5"};
{\ar@{->}_{q_{y'}} "4";"6"};
{\ar@{->}_{\alpha_y} "2";"5"};
{\ar@{->}^{\alpha_{y'}} "3";"6"}; 
\endxy .
\]
By the universal property of $\{q_y\colon  \lim (\Y\stackrel{b}{\to}\D) \to
b(y)\}_{y\in \Y_0}$ we have a unique map $L(f,\alpha)\colon 
\lim(\X\stackrel{a}{\to}\D) \to \lim(\Y\stackrel{b}{\to}\D)$ making
the diagram
\[
\xy (-5,20)*+{\lim(\X\stackrel{a}{\to}\D) }="1";
(10,7)*+{af(y)}="2"; 
(25,7)*+{af(y')} ="3";
(-5,-20)*+{\lim(\Y\stackrel{b}{\to}\D) }="4";
(10,-7)*+{b(y)}="5"; 
(25,-7)*+{b(y')} ="6";
{\ar@{->}_{p_{f(y)}} "1";"2"}; 
{\ar@{->}^{p_{f(y')}} "1";"3"};
{\ar@{->}_{af(\gamma)} "2";"3"};
{\ar@{->}^{b(\gamma)} "5";"6"};
{\ar@{->}^{q_y} "4";"5"};
{\ar@{->}_{q_{y'}} "4";"6"};
{\ar@{->}_{\alpha_y} "2";"5"};
{\ar@{->}^{\alpha_{y'}} "3";"6"};
{\ar@{..>}_{\exists ! L(f,\alpha)} "1";"4"};
\endxy 
\] 
commute.
The fact that $L$ preserves the composition of morphisms 
\[
\xy
(-20, 10)*+{\Z} ="-1"; 
( 0, 10)*+{\Y} ="1"; 
(20, 10)*+{\X} ="2";
(0,-6)*+{\D }="3";
{\ar@{->}_{c} "-1";"3"};
{\ar@{->}^{g} "-1";"1"};
{\ar@{->}_{b} "1";"3"};
{\ar@{->}^{f} "1";"2"};
{\ar@{->}^{a} "2";"3"};
{\ar@{=>}_<<<{\scriptstyle \alpha} (12,7)*{};(6,5)*{}} ; 
{\ar@{=>}_<<<{\scriptstyle \beta} (-3,7)*{};(-7,4)*{}} ; 
\endxy   
\]
follows from the commutativity of the diagram below 
\[
\xy (-20,40)*+{\lim( X\stackrel{a}{\to} \D)}="1"; 
(-40,8)*+{\lim (Y\stackrel{b}{\to}\D)}="2"; 
(-20,-40)*+{\lim (Z\stackrel{c}{\to}\D)} ="3";  
(0,20)*+{afg(z)}="4";
(30,20)*+{afg(z')}="5"; 
(0,0)*+{bg(z)} ="6";
(30,0)*+{bg(z')} ="7";
(0,-20)*+{c(z)} ="8";
(30,-20)*+{c(z')} ="9";
 {\ar@{->}_(.7){p_{fg(z)}} "1";"4"}; 
{\ar@{->}^{p_{fg(z')}} "1";"5"};
{\ar@{->}_(.7){q_{g(z)}} "2";"6"};
{\ar@/^{1pc}/^{q_{g(z')}} "2";"7"};
{\ar@{->}^{r_z} "3";"8"};
{\ar@{->}_{r_{z'}} "3";"9"};
{\ar@{->}_{L(g,\beta)} "2";"3"};
{\ar@{->}_(.56){L((g,\beta)*(f,\alpha))} "1";"3"};
{\ar@{->}_{L(f, \alpha)} "1";"2"};
{\ar@{->}^{afg(\gamma)} "4";"5"};
{\ar@{->}^{} "4";"6"};
{\ar@{->}^{} "5";"7"};
{\ar@{->}_{bg(\gamma)} "6";"7"};
{\ar@{->}^{} "6";"8"};
{\ar@{->}^{} "7";"9"}; 
{\ar@{->}^{c(\gamma)} "8";"9"};
\endxy  
\]
for  any arrow $z\stackrel{\gamma}{\to}z'$ in $\Z$.
\end{proof}

\begin{remark} \label{L(nat-iso)} It follows from the proof above that
  if $a, a'\colon \X \to D$ is a pair of functors and $\alpha\colon  a\Rightarrow
  a'$ is a natural {\em isomorphism}, then $L(id,
  \alpha)\colon L(X\stackrel{a}{\to}D) \to L(X\stackrel{a'}{\to}D)$ is an
  isomorphism in $\D$.   Since any 2-commutative triangle $\xy
(-10, 10)*+{\Y} ="1"; 
(10, 10)*+{\X} ="2";
(0,-6)*+{\D }="3";
{\ar@{->}_{b} "1";"3"};
{\ar@{->}^{f} "1";"2"};
{\ar@{->}^{a} "2";"3"};
{\ar@{=>}_<<<{\scriptstyle \alpha} (4,6)*{};(-2,2)*{}} ; 
\endxy$ 
can be factored as
\[
 \xy
(-20, 10)*+{\Y} ="-1"; 
( 0, 10)*+{\Y} ="1"; 
(20, 10)*+{\X} ="2";
(0,-6)*+{\D }="3";
{\ar@{->}_{b} "-1";"3"};
{\ar@{->}^{id_Y} "-1";"1"};
{\ar@{->}^{af} "1";"3"};
{\ar@{->}^{f} "1";"2"};
{\ar@{->}^{a} "2";"3"};
{\ar@{=>}_<<<{\scriptstyle id} (12,7)*{};(6,5)*{}} ;  
{\ar@{=>}_<<<{\scriptstyle \alpha} (-3,7)*{};(-7,4)*{}} ;  
\endxy  ,      
\]
if $\alpha$ is a natural isomorphism then $L(f,\alpha)$ is an
isomorphism  if and only if $L(f, id)$ is an isomorphism.
\end{remark}
\noindent
We are now ready to state and prove the main result of the section.
\begin{theorem}\label{thm:iso}
  Let $\X$, $\Y$ be small categories and $\D$ a category that has all
  small limits.  Suppose $\Y\stackrel{F}{\to}\X$ is part of an
  equivalence of categories and $\alpha\colon  aF\Rightarrow b$ a natural
  isomorphism.  Then
\[
L(F,\alpha)\colon  \lim(\X\stackrel{a}{\to}\D)\to \lim (\Y\stackrel{b}{\to}\D)
\] 
is an isomorphism in $\D$.
\end{theorem}

\noindent
Note that by Remark~\ref{L(nat-iso)} it is no loss of generality to
assume that $\alpha$ is the identity natural isomorphism, that is, that $af
= b$.

There are several ways to write down the proof of
Theorem~\ref{thm:iso}. We present the proof that uses nothing more
than the universal property of limits.  The reader may wish to write
down a shorter proof that uses adjoint functors.
\begin{proof}[Proof of Theorem~\ref{thm:iso}]
  Without loss of generality we assume that $\alpha$ is the identity
  natural isomorphism so that $b= af$.
  Let $\{p_x\colon  L\to a(x)\}_{x\in X}$ denote a limit of $a\colon  \X \to \D$.
  Then by definition for any arrow $y\stackrel{f}{\to}y'$ of $\Y$ the
  diagram
\[
\xy (0, 7)*+{ L} ="1";
  (-14,-6)*+{aF(y) }="2"; 
(14,-6)*+{aF(y')} ="3"; 
{\ar@{->}_{ p_{F(y)}}    "1";"2"}; 
{\ar@{->}^{p _{F(y')}} "1";"3"}; 
{\ar@{->}^{aF(f)}    "2";"3"}; 
  \endxy
\]
commutes in $\D$.  We argue that $\{p_{F(y)}\colon L\to
aF(y)\}_{y\in \Y_0}$ is the terminal object in $\Cone(aF)$, that is, a
limit of $aF\colon \Y\to \D$. For then the induced map $L(F, id): L \to L$ is the identity map and the result follows.

Suppose $\{\mu_a\colon  d\to aF(y)\}_{y\in \Y_0}$ is a cone on $aF$.  We
need to produce a morphism of cones $\mu\colon  (d, \{\mu_y\}) \to (L,
\{p_{F(y)}\})$.  Since $F\colon \Y\to \X$ is part of an equivalence of
categories there is a functor 
\[
H\colon \X\to \Y 
\]
and natural isomorphism
\[
\tau\colon  FH\Rightarrow id_\X, \quad\nu\colon  HF\Rightarrow id_\Y.
\]
For any $x\in \X_0$ the arrow $\tau_x\colon  FH(x) \to x$ is an isomorphism
in $\X$.  Hence $a(\tau_x)\colon aFH (x)\to a(x)$ is an isomorphism in $\D$.
Consider
\[
\zeta_x\colon  d\to a(x), \quad \zeta_x := a(\tau_x) \circ \mu_{H(x)}.
\]
For any arrow $x\stackrel{g}{\to}x'$ in $\X$ the square
\[
\xymatrix{
FH(x) \ar[r]^{FH(g) } \ar[d]_{\tau_x}
& FH(x')\ar[d]^{\tau_{x'}}  \\
 x  \ar[r]^{g }
&  x' }
\] 
commutes in $\X$.  Hence
\begin{equation}\label{eq:6.1}
\xymatrix{
aFH(x) \ar[r]^{aFH(g) } \ar[d]_{a(\tau_x)}
& aFH(x')\ar[d]^{a(\tau_{x'})}  \\
 a(x)  \ar[r]^{a(g) }&  a(x') }
\end{equation}
commutes in $\D$. By assumption 
\begin{equation}\label{eq:6.2}
\xy (0, 7)*+{ d} ="1";
  (-14,-6)*+{aFH(x) }="2"; 
(14,-6)*+{aFH(x')} ="3"; 
{\ar@{->}_{ \mu_{H(x)}}    "1";"2"}; 
{\ar@{->}^{\mu _{H(x')}} "1";"3"}; 
{\ar@{->}^{aFH(g)}    "2";"3"};
  \endxy
\end{equation}
commutes in $\D$.  Equations \eqref{eq:6.1} and \eqref{eq:6.2}
together with the definition of $\zeta_x$ imply that
\[
\xy (0, 7)*+{ d} ="1";
  (-14,-6)*+{a(x) }="2"; 
(14,-6)*+{a(x')} ="3"; 
{\ar@{->}_{ \zeta_{x}}    "1";"2"}; 
{\ar@{->}^{\zeta _{x'}} "1";"3"}; 
{\ar@{->}^{a(g)}    "2";"3"};
  \endxy
\]
commutes in $\D$, that is, $(d, \{\zeta_x\})$ is a cone over $a$.
Since $(L, \{p_x \})$ is terminal in $\Cone(a)$, there is a unique map
of cones $\zeta\colon  (d, \{\zeta_x\}) \to (L, \{p_x\})$.  Hence
\begin{equation}\label{eq6.3}
\xy
(-10, 7)*+{d} ="1";
(10, 7)*+{L} ="2";
(0, -5)*+{aF(y)} ="3";
{\ar@{->}^{ \zeta}    "1";"2"}; 
{\ar@{->}_{ \zeta_{F(y)} }   "1";"3"}; 
{\ar@{->}^{ p_{F(y)}}    "2";"3"}; 
\endxy
\end{equation}
commutes for any object  $y$ of $\Y$.  However, 
\[
\zeta_{F(y)} = a (\tau_{F(y)}) \circ \mu _{HF(y)},
\]
which is not $\mu_y$.  Indeed, since $HF(y)\stackrel{\nu_y}{\to } y$
is an arrow in $Y$ (in fact an isomorphism), the diagram
\[
\xy (0, 7)*+{ d} ="1";
  (-14,-6)*+{aFHF(y) }="2"; 
(14,-6)*+{aF(y))} ="3"; 
{\ar@{->}_{ \mu_{HF(y)}}    "1";"2"}; 
{\ar@{->}^{\mu _{y}} "1";"3"}; 
{\ar@{->}^{aF (\nu_y)}    "2";"3"};
  \endxy
\]
commutes.  Hence
\begin{equation}\label{eq6.4}
\zeta _{F(y)} = a(\tau_{F(y)} ) \circ aF(\nu_y\inv) \circ \mu_y
\end{equation}
for any $y$ in $\Y$.  
We therefore can rewrite \eqref{eq6.3} as the commuting diagram
\begin{equation}\label{eq6.5}
\xy
(-10, 7)*+{d} ="1";
(10, 7)*+{L} ="2";
(-10, -5)*+{aF(y)} ="3";
(10, -5)*+{aF(y)} ="4";
{\ar@{->}^{ \zeta}    "1";"2"}; 
{\ar@{->}_{ \mu_{y} }   "1";"3"}; 
{\ar@{->}^{ p_{F(y)}}    "2";"4"};
 {\ar@{<-}_{ \kappa_y}    "3";"4"};
\endxy,
\end{equation}
where 
\[
\kappa_y\colon = \left(a(\tau_{F(y)} ) \circ aF(\nu_y\inv) \right)\inv,
\]
for any object $y$ of $\Y$.  Note that since $\tau$ and $\nu$ are
natural isomorphisms, the collection $\{\kappa_y\}$ is a natural
isomorphism from $aF$ to $aF$.  Consequently
\[
  (L, \{\kappa_y\circ p_{F(y)}\}_{y\in \Y_0})
\]
is a cone on $aF$. Moreover the preceding argument shows that for any
cone $(d, \{\mu_y\})$ on $aF$ there is a morphism of cones $\zeta$
from $(d, \{\mu_y\})$ to $(L, \{\kappa_y\circ p_{F(y)}\}$.  We leave
it to the reader to check that $\zeta$ is, in fact,  unique. 

We conclude that $(L, \{\kappa_y\circ p_{F(y)}\colon L \to
aF(y)\}_{y\in \Y_0})$ is a limit of $aF\colon \Y\to\D$, that is, a
terminal object in $\Cone(aF)$.  The natural isomorphism $\kappa
\colon aF \Rightarrow aF$ induces an isomorphism of categories
$\Cone(aF)\to \Cone (aF)$.  It follows that since $(L, \{\kappa_y\circ
p_{F(y)}\colon L \to aF(y)\}_{y\in \Y_0})$ is a terminal object in
$\Cone (aF)$, so is $(L, \{p_{F(y)}\colon L \to aF(y)\}_{y\in \Y_0})$.
\end{proof}

\begin{proof}[Proof of Theorem~\ref{thm:ODEequiv}]
  Since $G(\varphi)\colon G(\Gamma)\to G(\Gamma')$ is always fully
  faithful by construction, the assumption that $G(\varphi)$ is
  essentially surjective implies that $G(\varphi)$ is an equivalence
  of categories.  By \eqref{eq:3.2} the diagram
\[
\xy
(-10, 10)*+{G(\Gamma)} ="1"; 
(10, 10)*+{G(\Gamma')} ="2";
(0,-6)*+{\Vect }="3";
{\ar@{->}_{ \Ctrl |_{G(\Gamma)}} "1";"3"};
{\ar@{->}^{G(\varphi)} "1";"2"};
{\ar@{->}^{\Ctrl|_{G(\Gamma')}} "2";"3"};
{\ar@{=>}_<<<{\scriptstyle \phi} (4,6)*{};(-2,2)*{}} ; 
\endxy .    
\]
2-commutes (q.v. Remark~\ref{rmrk:2.10.11}).
By Theorem~\ref{thm:iso} $\V (\varphi)$ is an isomorphism.
\end{proof}

\begin{proof}[Proof of Theorem~\ref{thm:inv-of-groups}]
  Let $\varphi: \bigsqcup_{i=1}^k\Aut(I(a_i)) \hookrightarrow
  G(\Gamma)$ denote the canonical inclusion of the skeleton of the
  groupoid $G(\Gamma)$ into the groupoid. By definition of the
  skeleton, $\varphi$ is an equivalence of categories.  Clearly
\[
\xy
(-17, 10)*+{\bigsqcup_{i=1}^k\Aut(I(a_i))} ="1"; 
(17, 10)*+{G(\Gamma)} ="2";
(0,-6)*+{\Vect }="3";
{\ar@{->}_{ \Ctrl |_{\bigsqcup_{i=1}^k\Aut(I(a_i))}} "1";"3"};
{\ar@{->}^{\quad G(\varphi)} "1";"2"};
{\ar@{->}^{\Ctrl|_{G(\Gamma)}} "2";"3"};
\endxy 
\]
strictly commutes.  Hence by Theorem~\ref{thm:iso} 
\[
\V (\varphi):\V(\Gamma) \to \lim (\Ctrl\colon
\bigsqcup_{i=1}^k\Aut(I(a_i)) \to \Vect)
\]
is an isomorphism.  But
\[
 \lim (\Ctrl\colon
\bigsqcup_{i=1}^k\Aut(I(a_i)) \to \Vect) 
=  \prod_{i=1}^k \Ctrl(I(a_i))^{\Aut(I(a_i))}
\]
by \ref{rmrk:A.2.13}, and the result follows.
\end{proof}

\appendix

\section{Elements of  category theory}\label{sec:category}

\subsection{Basic notions}
\mbox{}\\[-8pt]

\noindent
We start by recalling the basic definitions of category theory, mostly
to fix our notation.  This appendix may be useful to the reader with
some background in category theory; the reader with little to no
experience in category theory may wish to 
consult a textbook such as \cite{Awodey}.

\begin{definition}[Category]\label{def:category}
  A {\bf category} $\A$ consists

\begin{enumerate} 
\item A collection\footnote{A collection may be too big to be a set.
    While for many constructions in category theory the \em size \em
    of $\A_0$ and $\A_1$ is important, it will play only a limited role
    in this paper. Thus we will mostly ignore the usual set-theoretic
    issues.}  $\A_0$ of {\em objects}; 

\item For any two objects $a,b\in \A_0$, a set $\Hom_\A (a,b)$ of of
  {\em morphisms} (or {\em arrows});
\item For any three objects $a,b,c\in \A_0$, and any two arrows
  $f\in\Hom_\A(a, b)$ and $g\in\Hom_\A(b,c)$, a {\em composite }
  $g\circ f\in\Hom_\A(a,c)$, i.e., for all triples of objects
  $a,b,c\in\A_0$ there is a {\em composition map}
\[
\circ\colon \Hom_\A(b,c)\times\Hom_\A(a,b)\rightarrow\Hom_\A(a,c),
\]
\[
\Hom_\A(b,c)\times\Hom_\A(a,b) \ni (g,f)\mapsto g\circ f\in \Hom_\A(a,c).
\] 
This composition operation is {\em
associative} and has {\em units}, that is,
\begin{itemize} 
\item[i.] for any triple of morphisms $f\in\Hom_\A(a,b)$,
  $g\in\Hom_\A(b,c)$ and $h\in\Hom_\A(c,d)$ we have
\[ 
h\circ (g\circ f)=(h\circ g)\circ f\,; 
\] 
\item[ii.] for any object $a\in \A_0$, there exists a morphism
  $1_a\in\Hom_\A(a,a)$, called the {\em identity}, which is such that
  for any $f\in\Hom_\A(a,b)$ we have
\[ 
f=f\circ
1_a=1_b\circ f\,.  
\] 
\end{itemize}
We denote the collection of all morphisms of a category $\A$ by $\A_1$:
\[
\A_1 = \bigsqcup _{a,b\in \A_0} \Hom_\A (a,b).
\]
\end{enumerate} 
\end{definition}
\begin{remark}
  The symbol ``$\circ$'' is customarily suppressed in writing out
  compositions of two morphisms.  Thus
\[
gf\equiv g\circ f.
\]
\end{remark}
\begin{remark}
  Given a category $\A$ its underlying graph has the collections $\A_0$
  as nodes, $\A_1$ as edges, and the source and target maps $s,t\colon \A_1
  \to \A_0$ are defines so that for any $f\in \Hom_\A(a,b)$ we have
\[
s(f) = a, \quad t(f) = b.
\]
In other words, if we forget the composition in a category we are left
with a (directed) graph.
\end{remark}
\begin{example}[Category $\Set$ of sets] 
The collection $\Set$ of all sets forms a category.
  The objects of $\Set$ are sets, the arrows of $\Set$ are ordinary
  maps and the composition of arrows is the composition of maps.
\end{example}

\begin{example}[Category $\Vect$ of vector spaces]
  The collection $\Vect$ of all real vector spaces (not necessarily
  finite dimensional) forms a category.  Its objects are vector spaces
  and its morphisms are linear maps.  The composition of morphisms is
  the ordinary composition of linear maps.
\end{example}

\begin{example}[Category $\Graph$ of directed graphs] The collection
  $\Graph$ of all directed graphs (Definition~\ref{def:directed
    graph}) forms a category.  Its objects are graphs and its
  morphisms are maps of graphs.
\end{example}

\begin{example}[Category $\Man$ of manifolds]
 The collection of all
  finite dimensional Hausdorff paracompact manifolds and smooth
  maps forms the category $\Man$ of manifolds.  
\end{example}

\begin{example}[Sets as discrete categories]\label{ex:discrete cat}
  Any set $X$ can be thought of as a category $\X$ with the collection
  of objects $\X_0 =X$ and the collection of arrows $\X_1
  =\{1_x\}_{x\in X}$.  The only pairs of arrows that can be composed
  are of the form $(1_x, 1_x)$ and we define their composite to be
  $1_x$. One refers to such categories as {\em discrete} categories.
\end{example}
\begin{definition}\label{def:subcat} 
  A {\em subcategory} $\A$ of a category $\B$ is a collection of some
  objects $\A_0$ and some arrows $\A_1$ of $\B$ such that:
\begin{itemize}
\item For each object $a\in \A_0$, the identity $1_a$ is in $\A_1$;
\item For each arrow $f\in \A_1$ its source and target $s(f), t(f)$
  are in $\A_0$;
\item for each pair $(f,g)\in \A_0 \times \A_0$ of composable arrows
  $a\stackrel{f}{\to}a'\stackrel{g}{\to}a''$ the composite $g\circ f$
is in $\A_1$ as well.
\end{itemize}
\end{definition}
\begin{remark}
Naturally a subcategory is a category in its own right.
\end{remark}

\begin{example}
  The collection $\Finset$ of all finite sets and all maps between
  them is a subcategory of $\Set$ hence a category.

  The collection $\FinGraph$ of finite directed graphs (that is graphs
  whose collections of nodes and edges are finite sets) and maps of
  graphs between them is a subcategory of $\Graph$.

  The collection $\FinVect$ of real finite dimensional vector spaces
  and linear maps is a subcategory of $\Vect$.
\end{example}

\begin{example}[The category $\Euc$ of Euclidean spaces]
  The collection of all finite dimensional real vector spaces is a
  collection of objects for two different categories:
\begin{enumerate}
\item It is the collection of objects of $\FinVect$.
\item  Since
  every finite dimensional vector space is canonically a Hausdorff
  second countable manifold, we also have the category $\Euc$
  (``Euclidean spaces'') with the same objects as $\FinVect$ but with
  morphisms defined to be all smooth maps. 
\end{enumerate}
The category $\FinVect$ can be thought of as a subcategory of the
category $\Euc$ of Euclidean spaces.  The category of Euclidean spaces
is naturally a subcategory of the category $\Man$ of manifolds.
\end{example}

\begin{example}[Groups are categories]\label{ex:groups are cats}
  In contrast with sets, which are categories with many objects and
  almost no morphisms, a group $G$ can be viewed as a category $\G$
  with one object as follows: the set of objects of $\G$ is a set with
  one element: $\G_0 = \{*\}$.  The set of arrows $\G_1$ is the set of
  elements of the group $G$: $\G_1 = G$.  The composition of arrows in
  $\G$ is multiplication in the group $G$.
\end{example}

\begin{example}[Disjoint union of groups as a category]
\label{ex:disjoint-union-gropus}
  Let $\{G_\alpha\}_{\alpha \in A}$ be a family of groups index by a
  set $A$. The disjoint union $\bigsqcup_{\alpha\in A} G_\alpha$ is not
  a group, since we cannot multiply elements of two different groups,
  but it can be thought of as a category $\A$.  Here are
  the details.

  The set of objects of the category $\A$ is the set $A$: $\A_0 = A$.
  The set of morphisms $\A_1$ is the disjoint union
  $\bigsqcup_{\alpha\in A} G_\alpha$ with $\Hom _A (\alpha,\alpha') =
  \emptyset$ if $\alpha\not= \alpha'$ and $\Hom_\A (\alpha,\alpha) =
  G_\alpha$.  The composition in $\Hom_\A (\alpha,\alpha)$ is the
  multiplication in $G_\alpha$.

  Elsewhere in the paper we  abuse the notation and write
  $\bigsqcup_{\alpha\in A} G_\alpha$ when we mean the category $\A$.
\end{example}

\begin{definition}[isomorphism]
  An arrow $f\in \Hom_\A(a,b)$ in a category $\A$ is an {\em
    isomorphism} if there is an arrow $g\in \Hom_A (b,a)$ with $g\circ
  f = 1_a$ and $f\circ g = 1_b$.  We think of $f$ and $g$ as inverses
  of each other and may write $g = f\inv$.  Clearly $g= f\inv$ is also
  an isomorphism.

  Two objects $a,b\in \A_0$ are {\em isomorphic} if there is an
  isomorphism $f\in \Hom_\A (a,b)$.  We will also say that $a$ is
  isomorphic to $b$.
\end{definition}

\begin{definition}[Groupoid] \label{def:groupoid} A {\em groupoid} is
  a category in which every arrow is an isomorphism.\footnote{Readers
    uneasy about the issues of size may further restrict the
    definition of a groupoid by requiring that its collection of
    objects is a set.}
\end{definition}

\begin{example}  Sets (thought of as discrete categories), groups
  (thought of as categories with just one object), and disjoint unions
  of groups \ref{ex:disjoint-union-gropus} are all groupoids.
\end{example}

\begin{definition}
A category $\X$ is {\em small} if its collections $\X_0$ of objects and $\X_1$ of morphisms are 
both sets.
\end{definition}
\begin{example}
  Any group thought of as a category is a small category.  Any set
  thought of as a discrete category is a small category.  However, the
  categories $\Set$, $\FinSet$, $\Vect$, $\Euc$, and $\Man$ are not
  small.
\end{example}

\begin{definition}[Opposite category]\label{def:opposite}
  Given a category $\A$, the {\em opposite category} $\op\A$ has the
  same objects as $\A$ and the arrows are reversed. That is $(\op\A)_0
  = \A_0$ and for any two objects $a,b\in\A_0= (\op\A)_0$ the set of
  morphisms is defined by
\begin{equation*}
  \Hom_{\op\A}(a,b) = \Hom_\A(b,a),
\end{equation*}
The composition 
\[
{\op\circ}\colon \Hom_{\op\A}(b,c)\times\Hom_{\op\A}(a,b)
\rightarrow\Hom_{\op \A}(a,c)
\] 
in $\op\A$ is defined by 
\[
\Hom_{\op\A}(b,c)\times\Hom_{\op\A}(a,b) (= \Hom_A(c,b)\times \Hom_A
(b,a)) \ni (g,f) \mapsto g{\op\circ}f := f\circ g.
\]
\end{definition}

\begin{definition}[Slice category] \label{empt:slice-cat}
Given a subcategory $\A$ of a category $\B$ and an object $b\in \B_0$
we can form a new category, $\A/b$.  The objects of $\A/b$ are arrows
of $\B$ of the form $x\stackrel{\alpha}{\to}b$ with $x\in \A_0$.  A
morphism from $x\stackrel{\alpha}{\to}b$ to $y\stackrel{\beta}{\to}b$
is defined to be a commuting triangle 
\[
\xy 
(-8, 4)*+{x}="1";
(8,4)*+{y}="2"; 
(0,-4)*+{b} ="3"; 
{\ar@{->}_{\alpha} "1";"3"}; 
{\ar@{->}^{\beta} "2";"3"}; 
{\ar@{->}^{h} "1";"2"};
\endxy
\]
in $\B$ (cf. ~Mac Lane \cite[p.\ 45]{MacLane}).  The category
$\A/b$ is  called the {\em slice category} and the {\em comma
  category}.  In this paper we will often abbreviate the commuting
triangle above as $h\colon x\to y$ with the rest of the triangle
understood.
\end{definition}

\begin{remark}
  The two slice categories that are important for us in this paper are
  $\FinSet/C_0$ where $C_0$ is a set, which is not necessarily finite,
  and $\FinGraph/C$ where $C$ is a directed graph, again, not
  necessarily finite.
\end{remark}

\begin{definition}[Functor]\label{def:functor}
  A (covariant) {\em functor} $F\colon\A\to\B$ from a category $\A$ to
a category $\B$ is a map on the objects and arrows of $\A$ such that
every object $a\in \A_0$ is assigned an object $Fa\in\B_0$, every
arrow $f\in\Hom_\A(a,b)$ is assigned an arrow $Ff\in\Hom_\B(Fa,Fb)$,
and such that composition and identities are preserved, namely
\begin{equation*}
  F(f\circ g) = Ff\circ Fg,\quad F 1_a = 1_{Fa}.
\end{equation*}
A {\em contravariant} functor from $\A$ to $\B$
is a covariant functor  $G\colon\op\A\to\B$.  This amounts to:
\[
G (f\circ g) = G(g)\circ G(f)
\] 
for all composable pairs of arrows $f,g$ of $\A$.
\end{definition}

\begin{example}
  Given a category $\A$ there is the {\em identity functor} $1_\A\colon \A
  \to \A$ which is the identity on objects and arrows.
\end{example}
\begin{example}
  If $\A$ is a subcategory of $\B$ then the natural inclusion $i\colon \A\to
  \B$ is a functor.
\end{example}

\begin{remark}
Since functors are maps, functors can be composed.
\end{remark}

\begin{remark}\label{rmrk:functors from discrete}
  A functor $F\colon \X \to \A$ from a {\em discrete} category $\X$ to a
  category $\A$ is a map $F$ from the set $X$ of objects of $\X$ to
  the collection of objects of $\A$: $F\colon X\to \A_0$.
\end{remark}

\begin{definition}\label{def:full}
A functor $F\colon \A\to \B$ is
\begin{enumerate}
\item {\em full} if $F\colon\Hom_\A(a,a')\to\Hom_\B(Fa,Fa')$
is surjective for all pairs of objects $a,a'\in \A_0$;
\item {\em faithful} if $F\colon\Hom_\A(a,a')\to\Hom_\B(Fa,Fa')$
is injective for all pairs of objects $a,a'\in \A_0$
\item {\em fully faithful} if $F\colon\Hom_\A(a,a')\to\Hom_\B(Fa,Fa')$
  is a bijection for all pairs of objects $a,a'\in \A_0$;
\item {\em essentially surjective} if for any object $b\in \B_0$ there
  is an object $a\in \A_0$ and an isomorphism $f\in \Hom_\B (F(a),
  b)$.  That is, for any object $b$ of $\B$ there is an object $a$ of
  $\A$ so that $b$ and $F(a)$ are isomorphic.
\end{enumerate}
\end{definition}

\begin{example}
  The inclusion $i\colon  \Euc \hookrightarrow \Man$ is fully faithful.  It
  is not essentially surjective since not every manifold is
  diffeomorphic to a vector space.

  The ``inclusion'' $\FinVect \hookrightarrow \Euc$ is faithful but
  not full since not every smooth map is linear.  It is essentially
  surjective since $\FinVect$ and $\Euc$ have the ``same'' objects.
\end{example}

\begin{definition}\label{def:full subcat}
  A subcategory $\A$ of a category $\B$ is {\em full} if the inclusion
  functor $i\colon \A\hookrightarrow \B$ is full (it is faithful by
  definition of what it means to be a subcategory).  That is, for any
  two objects $a,a'$ of $\A$,
\[
\Hom_\A (a,a') = \Hom_\B (a, a').
\]
\end{definition}

\begin{definition}[Natural Transformation]
  Let $F,G\colon\A\to\B$ be a pair of functors. A {\em natural
    transformation} $\tau\colon F\Rightarrow G$ is a family of
  $\{\tau_a\colon Fa\to Ga\}_{a\in\A_0}$ of morphisms in $\B$, one for
  each object $a$ of $\A$, such that, for any $f\in\Hom_\A(a,a')$, the
  following diagram commutes:
\begin{equation*}
  \xy
(-10, 10)*+{Fa} ="1"; 
(10, 10)*+{Fb} ="2"; 
(-10,-5)*+{Ga}="3";
(10,-5)*+{Gb} ="4";
{\ar@{->}^{Ff} "1";"2"};
{\ar@{->}^{\tau_b} "2";"4"};
{\ar@{->}_{\tau_a} "1";"3"};
{\ar@{->}^{Fg} "3";"4"};
\endxy
\end{equation*}
If each $\tau_a$ is an isomorphism, we say that $\tau$ is a {\em natural
isomorphism} (an older term is {\em natural equivalence}).
\end{definition}

\begin{definition}\label{def:comp-transf}
Natural transformations and functors can be composed in several distinct ways.
We will need two.

If $f,g,h\colon \Y\to \X$ are functors and $\alpha\colon f\Rightarrow g$, $\beta\colon 
g\Rightarrow h$ are natural transformations, then their {\em vertical
  composition} $\beta\circ_V \alpha\colon  f\Rightarrow h$ is defined by
\[
(\beta\circ_V \alpha)_y := \beta_y \circ \alpha _y
\]
for all $y\in \Y_0$.  The composition $\circ$ on the right is the
composition of arrows in $\X$.

If $k\colon \Z\to \Y$, $f,g\colon \Y\to \X$ are functors and $\alpha\colon f\Rightarrow
g$ a natural transformation, we define the {\em horizontal}
composition $\alpha \circ _H k\colon  fk\Rightarrow gk$ by
\[
(\alpha \circ _H k)_z:= \alpha_{k(z)}
\]
for all $z\in \Z_0$.
\end{definition}

\begin{definition}[Equivalence of
categories]
An {\em equivalence of categories} consists of a pair of functors
\[
F\colon \A\to \B, \quad E\colon \B\to \A
\]
and a pair of natural isomorphisms
\[
\alpha\colon  1_\A \Rightarrow E\circ F \quad \beta\colon  1_\B \Rightarrow F\circ E.
\]
\end{definition}
In this situation the functor $F$ is called {\em the pseudo-inverse}
or the {\em homotopy inverse} of $E$.  The categories $\A$ and $\B$
are then said to be {\em equivalent}.

\begin{proposition}\label{prop:eq-of-cats}
  A functor $F\colon \A \to \B$ is (part of) an equivalence of categories if
  and only if it is fully faithful and essentially surjective.
\end{proposition}
\begin{proof}
See \cite[Proposition~7.25]{Awodey}
\end{proof}

\begin{example}[Any groupoid is equivalent to a disjoint union of groups]
\label{ex:groupoid skeleton}
\footnote{This is a special case of the fact that any category is
  equivalent to its skeleton \cite[p.\ 93]{MacLane} } \quad Suppose a
category $\B$ is a groupoid.  Then either two objects $b,b'$ of $\B$
are isomorphic or $\Hom_\B (b,b')$ is empty.  Thus ``being
isomorphic'' is an equivalence relation on the collection of objects
$\B_0$ of $\B$.  Pick one representative for each equivalence class of
objects of $\B$ and call this collection $A$.  For every $a\in A$,
$\Hom_\B (a, a)$ is a group; call it $G_a$.  We have a natural
inclusion functor $i\colon  \bigsqcup _{a\in A} G_a \to \B$.  The functor
$i$ is fully faithful by construction.  It is essentially surjective
since any object $b\in \B_0$ is isomorphic to some object $a\in
(\bigsqcup _{a\in A} G_a)_0 = A$ by construction of $A$.  By
Proposition~\ref{prop:eq-of-cats} above, $i$ is an equivalence of
categories.

We will refer to the disjoint union $\bigsqcup _{a\in A} G_a$ as a
{\em skeleton} of the groupoid $\B$.
\end{example}

\subsection{Limits}
\begin{definition}[Limit]\label{def:limit} A {\em limit} $\lim  F\equiv
  \lim (F\colon \A\to \B)$ of a functor $F\colon \A\to \B$ (if it exists!) is an
  object $\lim F$ of $\B$ {\em together with a collection of arrows
  }$\{ p_a\colon  \lim F \to F(a)\}_{a\in \A_0}$ (called projections) so
  that
\begin{enumerate}
\item for any arrow
    $a\stackrel{f}{\to}a'$ in $\A$ the diagram
\[
\xy
(0, 10)*+{ \lim F} ="1"; 
(-20,-6)*+{F(a)}="2";
(20,-6)*+{F(a')} ="3";
{\ar@{->}_{ p_a} "1";"2"};
{\ar@{->}^{p_{a'}} "1";"3"};
{\ar@{->}^{F(f)} "2";"3"};
\endxy
\]
commutes;
\item Given an object $b$ of $\B$ and a family of arrows $ \{ \xi_a\colon 
  b\to F(a)\}_{a\in \A_0}$ so that $$\xy (0, 7)*+{ b} ="1";
  (-14,-6)*+{F(a) }="2"; (14,-6)*+{F(a')} ="3"; {\ar@{->}_{ \xi_a}
    "1";"2"}; {\ar@{->}^{\xi _{a'}} "1";"3"}; {\ar@{->}^{F(f)}
    "2";"3"};
  \endxy$$ commutes for any arrow $f\in\Hom_\A(a,a')$, there is a
  unique arrow $\xi\colon b\to \lim F$ making the diagram
\[
\xy
(-18, 10)*+{ \lim F} ="1"; 
(18, 10)*+{ b} ="1'"; 
(-18,-6)*+{F(a)}="2";
(18,-6)*+{F(a')} ="3";
{\ar@{->}_{ p_a} "1";"2"};
{\ar@{->}_(.3){p_{a'}} "1";"3"};
{\ar@{->}^(.3){ \xi_a} "1'";"2"};
{\ar@{->}^{\xi _{a'}} "1'";"3"};
{\ar@{-->}_{\xi} "1'";"1"};
{\ar@{->}_{F(f)} "2";"3"};
\endxy
\]
commute.
\end{enumerate}
\end{definition}

\begin{remark}[``Uniqueness'' of limits]\label{rmrk:uniqueness-of-lim}
 \quad It is not hard to show
  using property (2) of the definition of a limit, that while limits
  are not unique, any two limits of the same functor $F$ are
  isomorphic by way of a unique isomorphism.  One says that limits are
  unique up to a unique isomorphism.

  We will occasionally gloss over this fine point
  and talk about ``the limit'' of a functor $F$.
\end{remark}

\begin{definition}\label{def:complete-cat}
A category $\D$ is {\em complete} if for any {\em small} category $\X$ and  any functor $f\colon \X
\to \D$ the limit $\lim (f\colon \X\to \D)$ exists in $\D$.
\end{definition}
It is known that the categories $\Set$ of sets and the category
$\Vect$ of vector spaces are complete \cite{MacLane}, \cite{Awodey}.
The category of manifolds $\Man$ is not complete.

A categorical product is a special kind of a limit.  We will first
define categorical products, then give examples and then explain why
they (products)  are limits.

\begin{definition}[Categorical product]\label{def:product}
  A {\em categorical product} of a family $\{F_x\}_{x\in X}$ of
  objects of a category $\B$ indexed by the elements of a set $X$ is
  an object $\prod_{x\in X}F_x$ of $\B$ together with a set of arrows
  $\{p_x\colon  \prod_{y\in X}F_y\to F_x\}_{x\in X}$ of $\B$ enjoying the
  following universal property: for any object $b$ of $\B$ and any set
  of maps $\{\xi_x\colon b\to F_x\}$ there is a unique map $\xi\colon  b\to
  \prod_{x\in X}F_x$ making the diagram
\[
\xy
(-18, 10)*+{ \prod_{y\in X}F_y} ="1"; 
(2, 10)*+{ b} ="1'"; 
(-18,-6)*+{F_x}="2";
{\ar@{->}_{ p_x} "1";"2"};
{\ar@{->}^{ \xi_x} "1'";"2"};
{\ar@{-->}_(.3){\exists !\xi} "1'";"1"};
\endxy
\]
commute.
\end{definition}
\begin{remark}\label{rmkr:unique-prod}
  It is easy to show using the universal property of categorical
  products that categorical products (if they exist) are unique up to
  a unique isomorphisms. Again, we will gloss over non-uniqueness of
  products and talk about ``the product'' of a family $\{F_x\}_{x\in
    X}$.  See also \ref{empt:products-are-lim} below and
  Remark~\ref{rmrk:uniqueness-of-lim} above.
\end{remark}

\begin{example} \label{ex:A35}
  If $\B =\Set$ , the category of sets,and $\{F_1, F_2\}$ is a family
  of two sets then both Cartesian products $F_1\times F_2$ and
  $F_2\times F_1$ (together with the projection onto the factors
  $F_1$, $F_2$) are categorical products of $\{F_1, F_2\}$.  The
  unique isomorphism $F_1\times F_2 \to F_2\times F_1$ swaps the
  elements of the ordered pairs.

  If $\B$ is the category $\Vect$ of vector spaces over the reals and
  then the categorical product of $\{F_x\}_{x\in X}$ is the product
  vector space $\prod _{x\in X} F_x$.  If $X$ is finite, it is also
  the direct sum $\bigoplus _{x\in X} F_x$.
  If $X$ is empty then the zero vector space $\{0\}$ has the right
  universal property and is considered the (categorical) product of
  $\{F_x\}_{x\in \emptyset}$:
\[
\prod_{x\in \emptyset} F_x =\{0\}.
\]

If $X$ is infinite and $\B = \FinVect$ then the categorical product
$\prod _{x\in X} F_x$ does not exist in $\FinVect$, since the product
of infinitely many vector spaces it is not a finite dimensional vector
space (the Cartesian product $\prod_{x\in \emptyset} F_x$ is a vector
space, but it is not finite dimensional).

If $\B$ is the category $\Euc$ of Euclidean spaces and smooth maps and
$X$ is a finite set then the Cartesian product $\prod_{x\in X} F_x$ of
Euclidean spaces is a Euclidean space. Hence finite categorical
products exist in $\Euc$.  If $X$ is infinite then the categorical
product $\prod_{x\in X} F_x$ does not exist in $\Euc$.  If $X$ is
empty then the $\prod_{x\in \emptyset} F_x$ is the one point Euclidean
space $\{0\}$.

If $\B$ is the category $\Man$ of finite dimensional manifolds and
smooth maps and $X$ is a finite set then the categorical product
$\prod_{x\in X} F_x$ is the Cartesian product of manifolds.  Just as for the
categories of finite dimensional vector spaces and Euclidean spaces,
if the set $X$ is infinite then the categorical product $\prod_{x\in X} F_x$ does
not exist in $\Man$.
\end{example}

\begin{empt}[Categorical products are limits]\label{empt:products-are-lim}
To see that categorical products are limits think of a family of
objects $\{F_x\}_{x\in X}$ in a category $\B$ as a functor $F\colon \X \to
\B$, $F(x) = F_x$, from the corresponding discrete category $\X$ (q.v.\
Remark~\ref{rmrk:functors from discrete}).  Then $\lim (F\colon \X \to \B)$ has
exactly the universal properties of the categorical product
$\prod_{x\in X}F_x$.
\end{empt}

\begin{remark}[The set of maps into a product is a
  product]\label{rmrk:hom-into-product}
  Suppose we have a family of objects $\{F_x\}_{x\in X}$ in a category
  $\B$ and its product $\prod_{x\in X}F_x$ exists in $\B$.  Let
  $\{p_x\colon \prod_{x'\in X}F_{x'}\to F_x\}$ denote the family of the
  canonical projections.  For any morphism $f\colon Y\to \prod_{x\in X}F_x$
  in $\B$ and any $x\in X$ we have a morphism $p_x\circ f\colon Y\to F_x$.
  This defines a map
\[
\fp_x\colon  \Hom_\B (Y, \prod_{x\in X}F_x) \to \Hom_\B (Y, F_x), 
\quad \fp_x (f) := p_x\circ f \equiv (p_x)_* f.
\]
On the other hand, since arbitrary products exist in the category
$\Set$ of sets, we have the product $\prod_{x\in X}\Hom_B (Y,F_x)$.
Denote its family of the canonical projections by
$\{\pi_x\colon \prod_{x'\in X}\Hom_B (Y,F_{x'}) \to \Hom_\B(Y, F_x)\}$.  By
the universal property of the product $\prod_{x\in X}\Hom_B (Y,F_x)$
the maps $\{\fp_x\}_{x\in X}$ define a unique map
\[
\fp\colon  \Hom _B (Y, \prod_{x\in X}F_x) \to  \prod_{x\in X} \Hom _B (Y, F_x). 
\]
Informally speaking the map $\fp$ sends a map $f$ to the tuple of
maps $(p_x \circ f)_{x\in X} \in \prod_{x\in X} \Hom _B (Y, F_x)$.
{\em By the universal property of the product $\{p_x\colon \prod_{x'\in
    X}F_{x'}\to F_x\}$ the map $\fp$ is a bijection.}

There is an even better way to think of about the universal property
of  the family $\{\fp_x\colon 
\Hom_\B (Y, \prod_{x'\in X}F_x') \to \Hom_\B (Y, F_x)\}_{x\in X}$.  Namely, 
  $\Hom_\B (Y, \prod_{x\in X}F_x)$ together with the family
  $\{\fp_x\}_{x\in X}$ {\em is} a product of the family of sets $\{ \Hom _B
(Y, F_x)\}_{x\in X}$:
\[
\Hom_\B (Y, \prod_{x\in X}F_x) =  \prod_{x\in X} \Hom _B (Y, F_x).
\]
  Compare with \ref{def:product}.
\end{remark}

\begin{empt}[Invariants are limits]
  Let $\rho\colon G\to GL(V)$ be a representation of a group $G$.  Recall
  that the space of invariants, the space of $G$-fixed vectors, is
\[
V^G := \{ v\in V\mid \rho(g) v = v \textrm{ for all }g\in G\}.
\]
We now interpret $V^G$ together with its inclusion $V^G\hookrightarrow
V$ as a limit of functor.  View the group $G$ as a category $\G$ with
one object (see Example~\ref{ex:groups are cats} above).  Then the
representation $\rho$ defines a functor $\rho\colon \G \to \Vect$.  The {\em
  functor} $\rho$ sends the unique object $*$ of $\G$ to $V$ and an
arrow $g\in \G_1 = G$ to $\rho(g) \in \Hom_\Vect (V, V) = \Hom _\Vect
(\rho(*), \rho(*)).$ By definition of a limit, the limit of $\rho\colon \G
\to \Vect$ is a vector space $U$ together with a map $p\colon U\to \rho(*) =
V$ with the following universal property: any linear map $T\colon W\to V$
with 
\begin{equation}\label{eq:A26}
\rho(g) \circ T = T
\end{equation}
 for all $g\in G$  should factor
through $p\colon U\to V$.  Now \eqref{eq:A26} implies that for any $w\in W$
\[
\rho(g) T(w) = T(w) \textrm{ for all } g\in G.
\] 
Thus $T$ always factors through the inclusion $V^G\hookrightarrow V$
and therefore $\lim (\rho\colon \G\to \Vect) = V^G$.
\end{empt}

\begin{definition}
  Let $\A$ be a groupoid whose collection of objects $\A_0$ is a set.
  A {\em representation} $F$ of $\A$ is a functor from $\A$ to the
  category of vector spaces $\Vect$:
\[
F\colon  \A \to \Vect
\]
\end{definition}

\begin{definition}\label{def:invariants}
The {\em space of
    invariants} of a representation $F\colon \A \to \Vect$ of a groupoid
  $\A$ is the limit $\lim F$.
\end{definition}

\begin{remark}\label{rmrk:A.2.13}
  The space of invariants of $F\colon \A\to \Vect$ has the following
  concrete description: consider the product $\prod_{a\in \A_0}F(a)$
  together with its canonical projections $p_b\colon  \prod_{a\in \A_0}F(a)
  \to F(b)$.  Then
\[
U:= \{ v\in  \prod_{a\in \A_0}F(a) \mid F(\sigma) p_a (v) = p_b (v)
\textrm{ for all arrows } a\stackrel{\sigma}{\to}b\in \A_1\} 
\]
together with the projections  $\varpi_a = p_a |_U\colon  U \to
F(a)$.  To check this assertion just check the universal properties of  
the family $\{\varpi_a\colon  U \to F(a)\}_{a\in \A_0}$. 

If the groupoid $\A$ is a disjoint union of groups $\bigsqcup_{a\in A}
G_a$ then a representation $F\colon \A\to \Vect$ is a family of
representations $\{\rho_a\colon  G_a \to V_a\}_{a\in A}$, where $\rho_a =
F|_{\G_a}$ and $V_a = F(a)$.  It is not hard to check that the product
$\prod_{a\in A} V_a ^{G_a}$ together with the canonical maps
$\{\prod_{a\in A} V_a ^{G_a} \to V_a\}_{a\in A}$ has the universal
property of the limit $\lim F$.  Thus in this case
\[
\lim(F\colon \bigsqcup_a G_a \to \Vect) = \prod_{a\in A} V_a ^{G_a},
\]
i.e., it is the product of the invariants of constituent representations.
\end{remark}

\begin{remark}
  Since any groupoid $\A$ is equivalent to disjoint union of groups
  (Example~\ref{ex:groupoid skeleton}), it follows from Theorem~\ref{thm:iso}
   that the space of invariants of a
  representation of a groupoid is always the product of invariants of
  representations of groups.   C.f.\ \ref{thm:inv-of-groups} and its proof.
\end{remark}

We end the appendix by restating the definition of a limit of a
functor in terms of cones and terminal objects.  We will use the
restatement in Section~\ref{sec:2slice}.

\begin{definition}[Cones]  Let $F\colon \A \to \B$ be a functor. A {\em cone} over $F$ is an object 
$b$ of $B$ together with a collection of arrows 
$\{\xi_a\colon   b\to F(a)\}_{a\in \A_0}$ so that 
\[
\xy (0, 7)*+{ b} ="1";
  (-14,-6)*+{F(a) }="2"; 
(14,-6)*+{F(a')} ="3"; 
{\ar@{->}_{ \xi_a}    "1";"2"}; 
{\ar@{->}^{\xi _{a'}} "1";"3"}; 
{\ar@{->}^{F(f)}    "2";"3"};
  \endxy  
\]
 commutes for any arrow $f\in\Hom_\A(a,a')$. 

  A {\em morphism $\upsilon\colon  (b,\{\xi_a\})\to (b', \{\xi'_a\})$ of cones
  over $F$} is an arrow $\upsilon\colon  b\to b'$ in $\B$ making the triangle
\[ \xy
(-10, 7)*+{b} ="1";
(10, 7)*+{b'} ="2";
(0, -5)*+{F(a)} ="3";
{\ar@{->}^{ \upsilon}    "1";"2"}; 
{\ar@{->}_{ \xi_a}    "1";"3"}; 
{\ar@{->}^{ \xi'_a}    "2";"3"}; 
\endxy
\]
commute.  

We have an evident category $\Cone (F)$ of cones over a functor $F$.
\end{definition}

\begin{empt}
  Note that if $\upsilon\colon  (b,\{\xi_a\})\to (b', \{\xi'_a\})$ is a
  morphism of cones over $F$ then the diagram
\[
\xy
(-18, 10)*+{ b} ="1"; 
(18, 10)*+{ b} ="1'"; 
(-18,-6)*+{F(a)}="2";
(18,-6)*+{F(a')} ="3";
{\ar@{->}_{ \xi_a} "1";"2"};
{\ar@{->}_(.3){\xi_{a'}} "1";"3"};
{\ar@{->}^(.3){ \xi'_a} "1'";"2"};
{\ar@{->}^{\xi' _{a'}} "1'";"3"};
{\ar@{->}^{\upsilon} "1";"1'"};
{\ar@{->}_{F(f)} "2";"3"};
\endxy
\]
automatically commutes.
\end{empt}

\begin{definition}[Terminal object] An object $\tau$ of a category $\sfC$
  is {\em terminal} if for any object $c$ of $\sfC$ there is a unique
  arrow $x\colon c\to \tau$.
\end{definition}

\begin{example} Any one point set $\{*\}$ is terminal in the category
  $\Set$ of sets: given a set $X$ there is a unique function $f\colon X\to
  \{*\}$ that sends every element of $X$ to $*$.
\end{example}

\begin{definition}[Limit, restated]
  A limit $\lim F$ of a functor $F\colon \A\to \B$ is a terminal object in
  the category $\Cone(F)$ of cones over $F$.
\end{definition}

\begin{remark}
  Terminal objects need not exist in a given category.  But any two
  terminal objects $\tau$, $\tau'$ of a category $\sfC$ are necessarily
  isomorphic. That is, terminal objects are unique up to a unique
  isomorphism. We leave a proof of this fact as an easy and standard
  exercise.  Consequently limits are unique up to a unique
  isomorphism (compare \ref{rmrk:uniqueness-of-lim}).
\end{remark}


\begin{thebibliography}{10}

\bibitem{Rogers.Kincaid.book}
E.~M. Rogers and D.~L. Kincaid.
\newblock {\em Communication networks: toward a new paradigm for research}.
\newblock New York Free Press, 1981.

\bibitem{Knight.72}
B.~W. Knight.
\newblock Dynamics of encoding in a population of neurons.
\newblock {\em Journal of General Physiology}, 59(6):734--766, 1972.

\bibitem{EK84}
G.~B. Ermentrout and N.~Kopell.
\newblock Frequency plateaus in a chain of weakly coupled oscillators, {I}.
\newblock {\em SIAM Journal on Mathematical Analysis}, 15(2):215--237, 1984.

\bibitem{Kuramoto.91}
Y.~Kuramoto.
\newblock Collective synchronization of pulse-coupled oscillators and excitable
  units.
\newblock {\em Physica D}, 50(1):15--30, May 1991.

\bibitem{Kuramoto.book}
Y.~Kuramoto.
\newblock {\em Chemical oscillations, waves, and turbulence}, volume~19 of {\em
  Springer Series in Synergetics}.
\newblock Springer-Verlag, Berlin, 1984.

\bibitem{MacKay.book}
David MacKay.
\newblock {\em Information Theory, Inference, and Learning Algorithms}.
\newblock Cambridge University Press, 2003.

\bibitem{Dayan.Abbott.book}
Peter Dayan and L.F. Abbott.
\newblock {\em Theoretical Neuroscience: Computational and Mathematical
  Modeling of Neural Systems}.
\newblock MIT Press, 2001.

\bibitem{Bower.Bolouri.book}
James~M. Bower and Hamid Bolouri, editors.
\newblock {\em Computational Modeling of Genetic and Biochemical Networks}.
\newblock MIT Press, January 2001.

\bibitem{Shmulevich.etal.02}
Ilya Shmulevich, Edward~R. Dougherty, Seungchan Kim, and Wei Zhang.
\newblock Probabilistic Boolean networks: a rule-based uncertainty model for
  gene regulatory networks.
\newblock {\em Bioinformatics}, 18(2):261--274, 2002.

\bibitem{Alvarez-Buylla.etal.07}
Elena~R Alvarez-Buylla, Mariana~Ben\'\i tez, Enrique~Balleza D\'avila, \'Alvaro
  Chaos, Carlos Espinosa-Soto, and Pablo Padilla-Longoria.
\newblock Gene regulatory network models for plant development.
\newblock {\em Current Opinion in Plant Biology}, 10(1):83 -- 91, 2007.
\newblock Growth and Development / Edited by Cris Kuhlemeier and Neelima Sinha.

\bibitem{Wilkinson.book}
Darren~James Wilkinson.
\newblock {\em Stochastic modelling for systems biology}.
\newblock Chapman \& Hall/CRC Mathematical and Computational Biology Series.
  Chapman \& Hall/CRC, Boca Raton, FL, 2006.

\bibitem{Peskin.75}
C.~S. Peskin.
\newblock {\em Mathematical aspects of heart physiology}.
\newblock Courant Institute of Mathematical Sciences New York University, New
  York, 1975.
\newblock Notes based on a course given at New York University during the year
  1973/74, see
  {\texttt{http://math.nyu.edu/faculty/peskin/heartnotes/index.html}}.

\bibitem{Tyson.Keener.88}
John~J. Tyson and James~P. Keener.
\newblock Singular perturbation theory of traveling waves in excitable media (a
  review).
\newblock {\em Phys. D}, 32(3):327--361, 1988.

\bibitem{Kapral.Showalter.book}
Raymond Kapral and Kenneth Showalter, editors.
\newblock {\em Chemical Waves and Patterns}.
\newblock Springer, 1994.

\bibitem{Ermentrout.Rinzel.96}
G.~Bard Ermentrout and John Rinzel.
\newblock Reflected waves in an inhomogeneous excitable medium.
\newblock {\em SIAM Journal on Applied Mathematics}, 56(4):1107--1128, 1996.

\bibitem{Keener.Sneyd.book}
James~P. Keener and James Sneyd.
\newblock {\em Mathematical Physiology}, volume~8 of {\em Interdisciplinary
  Applied Mathematics}.
\newblock Elsevier, 1998.

\bibitem{Watts.Strogatz.98}
Duncan~J. Watts and Steven~H. Strogatz.
\newblock {Collective dynamics of `small-world' networks}.
\newblock {\em {Nature}}, {393}({6684}):{440--442}, {Jun 4} {1998}.

\bibitem{Barabasi.Albert.99}
Albert-L{\'a}szl{\'o} Barab{\'a}si and R{\'e}ka Albert.
\newblock Emergence of scaling in random networks.
\newblock {\em Science}, 286(5439):509--512, 1999.

\bibitem{Strogatz.01}
Steven~H. Strogatz.
\newblock Exploring complex networks.
\newblock {\em Nature}, 410:268--276, 8 March 2001.

\bibitem{Schwartz.etal.02}
N.~Schwartz, R.~Cohen, D.~ben Avraham, A.-L. Barab{\'a}si, and S.~Havlin.
\newblock Percolation in directed scale-free networks.
\newblock {\em Phys. Rev. E (3)}, 66(1):015104, 4, 2002.

\bibitem{Albert.Barabasi.02}
R{\'e}ka Albert and Albert-L{\'a}szl{\'o} Barab{\'a}si.
\newblock Statistical mechanics of complex networks.
\newblock {\em Rev. Modern Phys.}, 74(1):47--97, 2002.

\bibitem{Barabasi.03}
Albert-L{\'a}szl{\'o} Barab{\'a}si.
\newblock Emergence of scaling in complex networks.
\newblock In {\em Handbook of graphs and networks}, pages 69--84. Wiley-VCH,
  Weinheim, 2003.

\bibitem{Barabasi.etal.04}
Albert-L{\'a}szl{\'o} Barab{\'a}si, Zolt{\'a}n~N. Oltvai, and Stefan Wuchty.
\newblock Characteristics of biological networks.
\newblock In {\em Complex networks}, volume 650 of {\em Lecture Notes in
  Phys.}, pages 443--457. Springer, Berlin, 2004.

\bibitem{Boccaletti.etal.06}
S.~Boccaletti, V.~Latora, Y.~Moreno, M.~Chavez, and D.-U. Hwang.
\newblock Complex networks: structure and dynamics.
\newblock {\em Phys. Rep.}, 424(4-5):175--308, 2006.

\bibitem{Golubitsky.Stewart.84}
Martin Golubitsky and Ian Stewart.
\newblock Hopf bifurcation in the presence of symmetry.
\newblock {\em Bull. Amer. Math. Soc. (N.S.)}, 11(2):339--342, 1984.

\bibitem{Golubitsky.Stewart.85}
Martin Golubitsky and Ian Stewart.
\newblock Hopf bifurcation in the presence of symmetry.
\newblock {\em Arch. Rational Mech. Anal.}, 87(2):107--165, 1985.

\bibitem{Golubitsky.Stewart.86}
Martin Golubitsky and Ian Stewart.
\newblock Hopf bifurcation with dihedral group symmetry: coupled nonlinear
  oscillators.
\newblock In {\em Multiparameter bifurcation theory ({A}rcata, {C}alif.,
  1985)}, volume~56 of {\em Contemp. Math.}, pages 131--173. Amer. Math. Soc.,
  Providence, RI, 1986.

\bibitem{Golubitsky.Stewart.86.2}
Martin Golubitsky and Ian Stewart.
\newblock Symmetry and stability in {T}aylor-{C}ouette flow.
\newblock {\em SIAM J. Math. Anal.}, 17(2):249--288, 1986.

\bibitem{Golubitsky.Stewart.87}
Martin Golubitsky and Ian Stewart.
\newblock Generic bifurcation of {H}amiltonian systems with symmetry.
\newblock {\em Phys. D}, 24(1-3):391--405, 1987.
\newblock With an appendix by Jerrold Marsden.

\bibitem{Golubitsky.Stewart.Schaeffer.book}
Martin Golubitsky, Ian Stewart, and David~G. Schaeffer.
\newblock {\em Singularities and groups in bifurcation theory. {V}ol. {II}},
  volume~69 of {\em Applied Mathematical Sciences}.
\newblock Springer-Verlag, New York, 1988.

\bibitem{Field.Golubitsky.Stewart.91}
M.~Field, M.~Golubitsky, and I.~Stewart.
\newblock Bifurcations on hemispheres.
\newblock {\em J. Nonlinear Sci.}, 1(2):201--223, 1991.

\bibitem{Golubitsky.Stewart.Dionne.94}
Martin Golubitsky, Ian Stewart, and Benoit Dionne.
\newblock Coupled cells: wreath products and direct products.
\newblock In {\em Dynamics, bifurcation and symmetry ({C}arg\`ese, 1993)},
  volume 437 of {\em NATO Adv. Sci. Inst. Ser. C Math. Phys. Sci.}, pages
  127--138. Kluwer Acad. Publ., Dordrecht, 1994.

\bibitem{Dellnitz.Golubitsky.Hohmann.Stewart.95}
Michael Dellnitz, Martin Golubitsky, Andreas Hohmann, and Ian Stewart.
\newblock Spirals in scalar reaction-diffusion equations.
\newblock {\em Internat. J. Bifur. Chaos Appl. Sci. Engrg.}, 5(6):1487--1501,
  1995.

\bibitem{Dionne.Golubitsky.Silber.Stewart.95}
Benoit Dionne, Martin Golubitsky, Mary Silber, and Ian Stewart.
\newblock Time-periodic spatially periodic planforms in {E}uclidean equivariant
  partial differential equations.
\newblock {\em Philos. Trans. Roy. Soc. London Ser. A}, 352(1698):125--168,
  1995.

\bibitem{Dionne.Golubitsky.Stewart.96.1}
Benoit Dionne, Martin Golubitsky, and Ian Stewart.
\newblock Coupled cells with internal symmetry. {I}. {W}reath products.
\newblock {\em Nonlinearity}, 9(2):559--574, 1996.

\bibitem{Dionne.Golubitsky.Stewart.96.2}
Benoit Dionne, Martin Golubitsky, and Ian Stewart.
\newblock Coupled cells with internal symmetry. {II}. {D}irect products.
\newblock {\em Nonlinearity}, 9(2):575--599, 1996.

\bibitem{Golubitsky.Stewart.Buono.Collins.98}
Martin Golubitsky, Ian Stewart, Pietro-Luciano Buono, and J.~J. Collins.
\newblock A modular network for legged locomotion.
\newblock {\em Phys. D}, 115(1-2):56--72, 1998.

\bibitem{Golubitsky.Stewart.98}
Martin Golubitsky and Ian Stewart.
\newblock Symmetry and pattern formation in coupled cell networks.
\newblock In {\em Pattern formation in continuous and coupled systems
  ({M}inneapolis, {MN}, 1998)}, volume 115 of {\em IMA Vol. Math. Appl.}, pages
  65--82. Springer, New York, 1999.

\bibitem{Golubitsky.Stewart.99}
Martin Golubitsky and Ian Stewart.
\newblock Symmetry and pattern formation in coupled cell networks.
\newblock In {\em Pattern formation in continuous and coupled systems
  ({M}inneapolis, {MN}, 1998)}, volume 115 of {\em IMA Vol. Math. Appl.}, pages
  65--82. Springer, New York, 1999.

\bibitem{Golubitsky.Stewart.00}
Martin Golubitsky and Ian Stewart.
\newblock {\em The symmetry perspective}, volume 200 of {\em Progress in
  Mathematics}.
\newblock Birkh\"auser Verlag, Basel, 2002.
\newblock From equilibrium to chaos in phase space and physical space.

\bibitem{Golubitsky.Knobloch.Stewart.00}
M.~Golubitsky, E.~Knobloch, and I.~Stewart.
\newblock Target patterns and spirals in planar reaction-diffusion systems.
\newblock {\em J. Nonlinear Sci.}, 10(3):333--354, 2000.

\bibitem{Golubitsky.Stewart.02}
Martin Golubitsky and Ian Stewart.
\newblock Patterns of oscillation in coupled cell systems.
\newblock In {\em Geometry, mechanics, and dynamics}, pages 243--286. Springer,
  New York, 2002.

\bibitem{Golubitsky.Stewart.02.2}
Martin Golubitsky and Ian Stewart.
\newblock {\em The symmetry perspective}, volume 200 of {\em Progress in
  Mathematics}.
\newblock Birkh\"auser Verlag, Basel, 2002.
\newblock From equilibrium to chaos in phase space and physical space.

\bibitem{Stewart.Golubitsky.Pivato.03}
Ian Stewart, Martin Golubitsky, and Marcus Pivato.
\newblock Symmetry groupoids and patterns of synchrony in coupled cell
  networks.
\newblock {\em SIAM J. Appl. Dyn. Syst.}, 2(4):609--646 (electronic), 2003.

\bibitem{Golubitsky.Nicol.Stewart.04}
M.~Golubitsky, M.~Nicol, and I.~Stewart.
\newblock Some curious phenomena in coupled cell networks.
\newblock {\em J. Nonlinear Sci.}, 14(2):207--236, 2004.

\bibitem{Golubitsky.Pivato.Stewart.04}
M.~Golubitsky, M.~Pivato, and I.~Stewart.
\newblock Interior symmetry and local bifurcation in coupled cell networks.
\newblock {\em Dyn. Syst.}, 19(4):389--407, 2004.

\bibitem{Golubitsky.Stewart.05}
Martin Golubitsky and Ian Stewart.
\newblock Synchrony versus symmetry in coupled cells.
\newblock In {\em E{QUADIFF} 2003}, pages 13--24. World Sci. Publ., Hackensack,
  NJ, 2005.

\bibitem{Golubitsky.Stewart.Torok.05}
Martin Golubitsky, Ian Stewart, and Andrei T{\"o}r{\"o}k.
\newblock Patterns of synchrony in coupled cell networks with multiple arrows.
\newblock {\em SIAM J. Appl. Dyn. Syst.}, 4(1):78--100 (electronic), 2005.

\bibitem{Golubitsky.Stewart.06}
Martin Golubitsky and Ian Stewart.
\newblock Nonlinear dynamics of networks: the groupoid formalism.
\newblock {\em Bull. Amer. Math. Soc. (N.S.)}, 43(3):305--364, 2006.

\bibitem{Golubitsky.Josic.Brown.06}
M.~Golubitsky, K.~Josi{\'c}, and E.~Shea-Brown.
\newblock Winding numbers and average frequencies in phase oscillator networks.
\newblock {\em J. Nonlinear Sci.}, 16(3):201--231, 2006.

\bibitem{Golubitsky.Shiau.Stewart.07}
Martin Golubitsky, Liejune Shiau, and Ian Stewart.
\newblock Spatiotemporal symmetries in the disynaptic canal-neck projection.
\newblock {\em SIAM J. Appl. Math.}, 67(5):1396--1417 (electronic), 2007.

\bibitem{MacLane}
Saunders Mac~Lane.
\newblock {\em Categories for the working mathematician}, volume~5 of {\em
  Graduate Texts in Mathematics}.
\newblock Springer-Verlag, New York, second edition, 1998.

\bibitem{DeVille.Lerman.10.2}
R.~E.~Lee DeVille and Eugene Lerman.
\newblock Dynamics on networks {I}{I}. Combinatorial categories of modular
  discrete-time systems.
\newblock in preparation.

\bibitem{DeVille.Lerman.10.3}
R.~E.~Lee DeVille and Eugene Lerman.
\newblock Dynamics on networks {I}{I}{I}. Combinatorial categories of modular
  stochastic dynamical systems.
\newblock in preparation.

\bibitem{Pappas}
Paulo Tabuada and George~J. Pappas.
\newblock Quotients of fully nonlinear control systems.
\newblock {\em SIAM J. Control Optim.}, 43(5):1844--1866 (electronic), 2005.

\bibitem{Field}
Michael~J. Field.
\newblock {\em Dynamics and symmetry}, volume~3 of {\em ICP Advanced Texts in
  Mathematics}.
\newblock Imperial College Press, London, 2007.

\bibitem{Vigna1}
Paolo Boldi and Sebastiano Vigna.
\newblock Fibrations of graphs.
\newblock {\em Discrete Math.}, 243(1-3):21--66, 2002.

\bibitem{Vigna2}
Paolo Boldi, Violetta Lonati, Massimo Santini, and Sebastiano Vigna.
\newblock Graph fibrations, graph isomorphism, and {P}age{R}ank.
\newblock {\em Theor. Inform. Appl.}, 40(2):227--253, 2006.

\bibitem{Higgins}
Philip~J. Higgins.
\newblock {\em Notes on categories and groupoids}.
\newblock Van Nostrand Reinhold Co., London, 1971.
\newblock Van Nostrand Rienhold Mathematical Studies, No. 32.

\bibitem{VignaWP}
Sebastiano Vigna.
\newblock http://vigna.dsi.unimi.it/fibrations/.

\bibitem{Knutson.Tao.01}
Allen Knutson and Terence Tao.
\newblock Honeycombs and sums of {H}ermitian matrices.
\newblock {\em Notices Amer. Math. Soc.}, 48(2):175--186, 2001.

\bibitem{Li.Poon.03}
Chi-Kwong Li and Yiu-Tung Poon.
\newblock Principal submatrices of a {H}ermitian matrix.
\newblock {\em Linear Multilinear Algebra}, 51(2):199--208, 2003.

\bibitem{Awodey}
Steve Awodey.
\newblock {\em Category theory}, volume~49 of {\em Oxford Logic Guides}.
\newblock The Clarendon Press Oxford University Press, New York, 2006.

\end{thebibliography}
\end{document}